\documentclass[11pt,a4paper, reqno]{amsart}
\usepackage[utf8x]{inputenc}
\usepackage[margin=1.0in]{geometry}
\usepackage{amsfonts}
\usepackage{amsmath}
\usepackage{amssymb}
\usepackage{amsthm}
\usepackage{hyperref}
\usepackage{verbatim}
\usepackage[T1]{fontenc}
\usepackage{enumitem}
\usepackage{palatino}
\usepackage{tikz}
\usepackage{mathrsfs}
\usetikzlibrary{trees}
\usepackage{dsfont}
\usepackage{upgreek}
\usepackage[titletoc,title]{appendix}

\newcommand{\R}{\mathbb{R}}
\newcommand{\N}{\mathbb{N}}

\newcommand{\pr}{\mathbb{P}}
\newcommand{\ex}{\mathbb{E}}

\newcommand{\F}{\mathcal{F}}
\newcommand{\h}{\mathcal{H}}

\newcommand{\blangle}{\big\langle}
\newcommand{\brangle}{\big\rangle}

\newtheorem{lem}{Lemma}
\newtheorem{cor}{Corollary}
\newtheorem{prop}{Proposition}
\newtheorem{thm}{Theorem}
\newtheorem{dfn}{Definition}
\newtheorem{innercustomthm}{Hypothesis}
\newenvironment{customthm}[1]
{\renewcommand\theinnercustomthm{#1}\innercustomthm}
{\endinnercustomthm}

\theoremstyle{definition}
\newtheorem{rem}{Remark}
\makeatletter
\newcommand{\leqnomode}{\tagsleft@true\let\veqno\@@leqno}
\newcommand{\reqnomode}{\tagsleft@false\let\veqno\@@eqno}
\makeatother

\numberwithin{lem}{section}
\numberwithin{cor}{section}
\numberwithin{prop}{section}
\numberwithin{thm}{section}
\numberwithin{dfn}{section}

		\title{  Moderate deviations for systems of slow-fast stochastic reaction-diffusion equations        }
		\author{Ioannis Gasteratos\ , \ Michael Salins\  and \ Konstantinos Spiliopoulos\;\;* }
		\thanks{*Department of Mathematics and Statistics, Boston University, 111 Cummington Mall, Boston, MA , 02215, USA. E-mails: igaster@bu.edu, msalins@bu.edu, kspiliop@bu.edu.  This work was partially supported by the National Science Foundation (DMS 1550918, DMS 2107856) and Simons Foundation Award  672441.}
			\date{}

\begin{document}
		\maketitle
	
\begin{abstract}
\noindent The goal of this paper is to study the Moderate Deviation Principle (MDP) for a system of stochastic reaction-diffusion equations with a time-scale separation in slow and fast components and small noise in the slow component. Based on weak convergence methods in infinite dimensions and related stochastic control arguments, we obtain an exact form for the moderate deviations rate function in different regimes as the small noise and time-scale separation parameters vanish. Many issues that appear due to the infinite dimensionality of the problem are completely absent in their finite-dimensional counterpart. In comparison to corresponding Large Deviation Principles, the moderate deviation scaling necessitates a more delicate approach to establishing tightness and properly identifying the limiting behavior of the underlying controlled problem. The latter involves regularity properties of a solution of an associated elliptic Kolmogorov equation on Hilbert space along with a finite-dimensional approximation argument.
\end{abstract}


		\section{Introduction}

		\noindent
In this paper we study the asymptotic tail behavior of the following system of stochastic reaction-diffusion equations (SRDEs) with slow-fast dynamics on the interval $(0,L)\subset\R$ :
		
			\begin{equation}\label{model}
		\hspace*{-0.5cm}
		\left \{\begin{aligned}
		&\partial_tX^\epsilon(t,\xi)= \mathcal{A}_1 X^\epsilon(t,\xi)+ f\big(\xi, X^\epsilon(t,\xi),Y^\epsilon(t,\xi)\big)+\sqrt{\epsilon}\sigma\big(\xi, X^\epsilon(t,\xi),Y^\epsilon(t,\xi)\big)\partial_tw_1(t,\xi) \\&
		\partial_tY^\epsilon(t,\xi)= \frac{1}{\delta}\big[\mathcal{A}_2 Y^\epsilon(t,\xi)+ g\big(\xi, X^\epsilon(t,\xi),Y^\epsilon(t,\xi)\big)\big]+ \frac{1}{\sqrt{\delta}} \partial_tw_2(t,\xi)
		\\&
		X^\epsilon (0,\xi)= x_0(\xi)\;, Y^\epsilon(0,\xi)=y_0(\xi)\;,\;\; \xi\in(0,L)
		\\&
		\mathcal{N}_1X^\epsilon(t,\xi)= \mathcal{N}_2Y^\epsilon(t,\xi)=0\;\;,\;t\geq 0,\xi\in\{0,L\}.
		\end{aligned} \right.
		\end{equation}
		Here, $\epsilon$ is considered a small parameter,  $\delta=\delta(\epsilon) \rightarrow 0$ as $\epsilon\to 0$ and $L>0$.
		The operators $\mathcal{A}_1, \mathcal{A}_2$ are  second-order uniformly elliptic differential operators which encode the diffusive behavior of the dynamics, while the reaction terms are given by the (nonlinear) measurable functions  $f,g:[0,L]\times \R^2\rightarrow\R$. The operators $\mathcal{N}_1, \mathcal{N}_2$ correspond to either Dirichlet or Robin boundary conditions and the initial values $x_0,y_0$ are assumed to be in $L^2(0,L)$.
		
		The system is driven by two independent space-time white noises $\partial_tw_1, \partial_tw_2,$  defined on a complete filtered probability space $(\Omega,\F,\{\mathcal{F}_t\}_{t\geq 0},\pr)$. These are interpreted  as the distributional time-derivatives of two independent cylindrical Wiener processes $w_1, w_2$. The coefficient $\sigma:[0,L]\times\R^2\rightarrow\R$ is a measurable function multiplied by the noise $\partial_tw_1$.
		
	 Since $1/\delta$ is large as $\epsilon\to 0$,  we see that the first equation is perturbed by a small multiplicative noise of intensity $\sqrt{\epsilon}$ while the second contains large parameters and, at least formally, runs on a time-scale of order $1/\delta$. Thus, one can think of the solution $X^\epsilon$ of the former as the "slow" process (or slow motion) and the solution $Y^{\epsilon}$ of the latter as the "fast" process (or fast motion). Note that, since $\delta$ has a functional dependence on $\epsilon$, the $\delta$-dependence is suppressed from the notation.
	
	   As $\epsilon$ (and hence $\delta$) are taken to $0$ one expects, on the one hand, that the small noise will vanish. On the other hand, assuming that the fast dynamics exhibit ergodic behavior, $Y^{\epsilon}$ will converge in distribution to an equilibrium and its contribution to the limiting dynamics of $X^{\epsilon}$ will be averaged out with respect to the invariant measure. In \cite{cerrai2009khasminskii}, Cerrai demonstrated the validity of such an averaging principle for a system of reaction-diffusion equations in spatial dimension $d\geq 1$, perturbed by multiplicative (colored) noise in both components. The setting of the present paper is closer to that of \cite{cerrai2009averaging}, where Cerrai and Freidlin proved an averaging principle in spatial dimension $d=1$ and with (additive) noise only in the fast equation. In particular, letting $x\in L^2(0,L)$ and assuming that the coefficients are sufficiently regular, the fast process $Y^{\epsilon, x}$ with "frozen" slow component $x$ admits a unique strongly mixing invariant measure $\mu^x$ and the slow process $\{X^\epsilon\}_\epsilon$ converges in probability, as $\epsilon\to 0$, to the unique (deterministic) solution $\bar{X}$ of the averaged PDE
	   	
	   	\begin{equation}\label{x-aved}
	   	\left \{
	   	\begin{aligned}
	   	&\partial_t{\bar{X}(t,\xi)}= \mathcal{A}_1 \bar{X}(t,\xi)+ \bar{F}(\bar{X}(t))(\xi)\\&  \bar{X} (0,\xi)= x_0(\xi)\;,\;\; \xi\in(0,L)\\&
	   	\mathcal{N}_1\bar{X}(t,\xi)=0\;,\;\; t\geq 0, \xi\in\{0,L\}.
	   	\end{aligned}\right.
	   	\end{equation}
	   	The nonlinearity $\bar{F}$ is given by the averaged reaction term
	   	
	   	\begin{equation}\label{Nemave} \bar{F}(x)(\xi)=\bigg(\int_{\h} f(\cdot, x(\cdot), y(\cdot))\;d\mu^x(y)\bigg)(\xi).\end{equation}

        The averaging principle describes the typical dynamics of the slow process and thus can be viewed as a "Law of Large Numbers" for $X^{\epsilon}$. One may then study the problem of characterizing large deviations from the averaging limit. In the Large Deviation theory of multiscale stochastic dynamics, the relative rate at which the intensity of the small noise and the scale separation parameter vanish plays a significant role. In particular, we distinguish the following asymptotic regimes:
        \begin{equation}\label{Regimes}
        \begin{aligned}
        \lim_{\epsilon\to 0}\frac{\sqrt{\delta}}{\sqrt{\epsilon}}=\begin{cases}
        &0 \;,\;\;\quad\quad\quad \quad\text{Regime 1}\\&
        \gamma\in(0,\infty) \;,\;\;\text{Regime 2}\\&
        \infty\;,\;\;\quad\quad\quad\;\;\text{Regime 3}.
        \end{cases}
        \end{aligned}
        \end{equation}

       The problem of Large Deviations for slow-fast systems of stochastic reaction-diffusion equations has been considered in \cite{wang2012large} in dimension one, with additive noise in the fast motion and no noise component in the slow motion. In \cite{WSS}, the authors proved a Large Deviation Principle (LDP) in Regime 1, for a system with spatial dimension $d\geq 1$ and multiplicative noise,  using the weak convergence approach developed in \cite{budhiraja2008large}.

       Moderate deviations characterize the decay rates of rare event probabilities  that lie on an asymptotic regime between the Central Limit Theorem (CLT) and the corresponding LDP. The goal of the present paper is to prove a Moderate Deviation Principle (MDP) for system \eqref{model} in Regimes 1 and 2.  The latter is equivalent to deriving an LDP for the process
       	\begin{equation*}
       \eta^{\epsilon}(t,\xi)=\frac{X^{\epsilon}(t,\xi)-\bar{X}(t,\xi)}{\sqrt{\epsilon}h(\epsilon)}
       \end{equation*}
        with speed $h^2(\epsilon)$. The scaling factor $h(\epsilon)$ is such that \begin{equation}\label{h}
    h(\epsilon)\longrightarrow \infty\;\;,\;\;     \sqrt{\epsilon}h(\epsilon)\longrightarrow 0\;\;\text{as}\;\; \epsilon\to 0.
        \end{equation}
       Note that if we set $h\equiv 1$ and let $\epsilon\to 0$ we would observe the behavior of normal deviations (CLT) around $\bar{X}$ while if we naively set $h(\epsilon)=1/\sqrt{\epsilon}$ we would observe the Large Deviations behavior. Hence, the MDP fills an asymptotic gap between the CLT and the LDP and, as such, it inherits characteristics of both.

       One of the most effective methods in proving statements about the behavior of rare events (such as LDPs and MDPs)  is the weak convergence method (see \cite{boue1998variational}, \cite{budhiraja2008large}, as well as the books \cite{budhiraja2019analysis} and \cite{dupuis2011weak}) which is the method we are using in this paper. The core of this approach lies in the use of a variational representation of exponential functionals of Wiener processes (see \cite{boue1998variational} for SDEs and \cite{budhiraja2008large} for SPDEs). Roughly speaking, one can represent the exponential functional of the moderate deviation process $\eta^{\epsilon}$ that appears in the Laplace Principle (LP) (which is equivalent to an MDP) as a variational infimum of a family of controlled moderate deviation processes $\eta^{\epsilon,u}$, plus a quadratic cost, over a suitable family of stochastic controls $u$. In particular, for any bounded continuous function $\Lambda: C([0,T]; L^2(0,L))\rightarrow \R$:
       	\begin{equation}\label{varrepintro}
       	\small
       -\frac{1}{h^2(\epsilon)}\log\;\ex\big[ e^{-h^2(\epsilon)\Lambda(\eta^\epsilon)}\big]=\inf_{u\in\mathcal{P}^T(L^2(0,L)^2)}\ex\bigg[ \frac{1}{2} \int_{0}^{T} \big( \|u_1(t) \|^2_{L^2(0,L)}+\|u_2(t) \|^2_{L^2(0,L)}\big)\;dt+ \Lambda\big(\eta^{\epsilon,u}  \big)  \bigg],
       \end{equation}
       where $u=(u_1,u_2)$ and $\mathcal{P}^T(L^2(0,L)^2)$ is the family of $L^2(0,L)^2$-valued  progressively-measurable control processes,  where $u_i$ is measurable with respect to the filtration $\mathcal{F}^w_T$ generated by\\ $\{(w_1(t),w_2(t))\;, t\in[0,T]\}$ ($i=1,2$) and has finite $L^2([0,T]; L^2(0,L))$-norm.

      The  process $\eta^{\epsilon,u}$ that appears on the right hand side of \eqref{varrepintro} is defined by
      \begin{equation}\label{etau}
       \eta^{\epsilon,u}(t,\xi)=\frac{X^{\epsilon,u}(t,\xi)-\bar{X}(t,\xi)}{\sqrt{\epsilon}h(\epsilon)}\;.
      \end{equation}
      Here, $X^{\epsilon,u}$ corresponds to a controlled slow-fast system $(X^{\epsilon,u}, Y^{\epsilon,u})$ (see \eqref{controlledsystem} below) which results from \eqref{model} by perturbing the paths of the noise by an appropriately re-scaled control. It is due to the latter that this representation is called variational.

       In light of \eqref{varrepintro}, we see that in order to obtain a limit as $\epsilon\to 0$ of the Laplace
functional  (i.e. to prove an MDP),
one needs to analyze the limiting behavior of $\eta^{\epsilon,u}$ and, before doing so, obtain a priori estimates for the underlying controlled slow-fast system given in $\eqref{controlledsystem}$. The latter is the first technical part of the current work (Section \ref{Sec2}).   As in the LDP case, the difficulty in these estimates is in that the stochastic controls are only known to be square integrable.

Compared to the corresponding LDP,  the essential source of additional complexity in Moderate Deviations lies in the proof of tightness of the family $\{\eta^{\epsilon,u}; \epsilon,u\}$. What complicates the analysis is the singular moderate deviation scaling $1/\sqrt{\epsilon}h(\epsilon)$. We overcome this difficulty by following, in spirit, the general method developed by Papanicolaou, Stroock and Varadhan in \cite{papanicolaou3martingale}. This involves the study of fluctuations with the aid of an elliptic Kolmogorov equation, associated to the fast dynamics and posed on the infinite-dimensional space $L^2(0,L)$. After projecting the controlled fast process $Y^{\epsilon,u}$ to an $n$-dimensional eigenspace of the elliptic operator $\mathcal{A}_2$, we are able to apply It\^o's formula to the solution $\Phi^\epsilon$ of the Kolmogorov equation and derive an expression for $\eta^{\epsilon,u}$ that is free from asymptotically singular coefficients. Using the a priori estimates from Section \ref{Sec2} along with regularity results for $\Phi^\epsilon$ from \cite{cerrai2009averaging} and \cite{cerrai2001second} we are then able to show tightness (Section \ref{Sec4}).

Regarding the characterization of the limit in distribution of the process $\eta^{\epsilon,u}$, note that the presence of stochastic controls $u$ leads to a limiting invariant measure of the controlled fast process $Y^{\epsilon,u}$  which a priori depends on $u$. In order to deal with this in a unified manner across regimes we use the so-called “viable pair” construction (see \cite{WSS} and \cite{dupuis2012large}, \cite{spiliopoulos2013large} for the finite and infinite-dimensional settings respectively) to characterize the limit. The latter is a pair of a trajectory and measure $(\psi,P)$ that captures both the limit averaging dynamics of $\eta^{\epsilon,u}$ and the invariant measure of the controlled fast process $Y^{\epsilon,u}$.  In particular, the function $\psi$ is the solution of the limiting averaged equation for $\eta^{\epsilon,u}$  and the probability measure $P$  characterizes both the structure of the invariant measure of $Y^{\epsilon,u}$  and the control $u$. Although, in general, these two objects are intertwined and coupled together into the measure $P$, Regimes 1 and 2 lead to a decoupling of the form $P(dudydt) = \nu_t(du|y)\mu^{\bar{X}(t)}(dy)dt$, where $ \nu_t(du|y)$ is a stochastic kernel characterizing the control and $\mu^{\bar{X}(t)}$ is the local invariant measure.

The measure P is obtained as the limit of a family of  occupation measures $P^{\epsilon, \Delta}$,  that live on the product space of fast motion and control, with $\Delta=\Delta(\epsilon)\rightarrow 0$ to be specified later on.  The result on the weak convergence of the pair $(\eta^{\epsilon,u}, P^{\epsilon, \Delta})$ in Regimes 1 and 2 is the content of Theorem \ref{viablim1}.

With the analysis of the limit and the construction of a viable pair, we then prove the Laplace Principle (equivalently LDP) for the moderate deviation process $\eta^\epsilon$ in Regimes 1 and 2 (Section \ref{MDPsec}). The main result of the paper is stated in Theorem \ref{MDP1}. Proving the Laplace principle amounts to
 finding an appropriate functional $S$ such that for any bounded and continuous function
  $\Lambda : C([0, T ]; L^2(0,L))\rightarrow \R$
 \begin{equation*}
 \begin{aligned}
 &\lim_{\epsilon\to 0}\frac{1}{h^2(\epsilon)}\log\;\ex\big[ e^{-h^2(\epsilon)\Lambda(\eta^\epsilon)}\big]=-\inf_{\phi\in C([0,T];
 	L^2(0,L))}\big[ S(\phi)+\Lambda(\phi) \big].
 \end{aligned}
 \end{equation*}
 As is common in the relevant literature, the Laplace principle upper bound can be proven using the weak convergence of the pair $(\eta^{\epsilon,u}, P^{\epsilon, \Delta})$ per Theorem \ref{viablim1}. The situation is more complicated for the Laplace principle lower bound for which we need to construct nearly optimal controls in feedback form (i.e. they are functions of both time and the fast motion) that achieve the bound.

  In finite dimensions, the Large Deviation theory  for multiscale diffusions  with periodic coefficients has been established in all three interaction Regimes and with the use of the weak convergence approach    (see \cite{dupuis2012large}, \cite{spiliopoulos2013large} and the references therein). The problem of Moderate Deviations in finite dimensions has been treated in \cite{CL10,DupuisJohnson2015,Guillin2003,GuillinLipster2005,morse2017moderate} under different settings and assumptions. Specifically, the finite-dimensional work of \cite{morse2017moderate} makes use of solutions to associated elliptic equations to treat Regimes 1 and 2. While the well-posedness and regularity theory of such equations are well-studied in finite dimensions (see e.g. \cite{pardoux2001poisson}), their analysis on infinite-dimensional spaces becomes quite more involved and the relevant literature is more limited. The absence of available regularity results for a general class of such equations is the main reason why we only consider the fast equation with additive noise.

 To the best of our knowledge, the problem of moderate deviations for systems of slow-fast stochastic reaction-diffusion is being considered for the first time in this paper. Its contribution is twofold:

 On a theoretical level, it provides a way to study rare events for the infinite-dimensional dynamics in both Regimes 1 and 2. In the LDP setting, Regime 2 remains open as it does not lead to a decoupling of the limiting invariant measure of $Y^{\epsilon,u}$ and the control $u$. The regularity of the optimal controls  has been studied in finite dimensions using their characterization through solutions to Hamilton-Jacobi-Bellman equations (see \cite{spiliopoulos2013large}). Such techniques have not been established on an infinite-dimensional setting. However, as shown in this paper, Regime 2 can be studied in the context of Moderate Deviations. In this regime the control of the fast equation survives in the limit. This reflects the fact that we are studying fluctuations very close to the CLT and a certain derivative of the Kolmogorov equation (see the term $\Psi_2^0u_2$ in Theorem \ref{viablim1}) captures the contribution of these fluctuations. It is worth noting that normal deviations from the averaging limit for slow-fast stochastic reaction-diffusion equations have been studied in \cite{cerrai2009normal}. This was done with different techniques and no explicit connection was drawn between the covariance of the limiting Gaussian process and the solution of the Kolmogorov equation. More recently, the authors of \cite{rockner2020asymptotic} generalized the results of \cite{cerrai2009normal} and studied normal deviations from the averaging limit using the Kolmogorov equation approach.

 On a computational level, the solution to the stochastic control problem gives vital information for the design of efficient Monte Carlo methods for the approximation of rare event probabilities on the moderate deviation range. In particular, the fact that the limiting equation is affine in $\eta^{\epsilon,u}$  is expected to make moderate deviation-based importance sampling for stochastic PDE easier to implement than its large deviation-based counterpart, see \cite{spiliopoulos2020importance} for the related situation in finite dimensions. We plan to explore this in a future work.

 The outline of this paper is as follows: in Section \ref{notation} we give background definitions, set-up as well as our assumptions. In Section \ref{weakconv} we review basic facts about the weak convergence method in infinite dimensions and we define viable pairs and occupation measures as well as state our main results on averaging for the controlled moderate deviation process $\eta^{\epsilon,u}$ and the MDP. In Section \ref{Sec2} we prove a priori bounds for the solution of the controlled system $(X^{\epsilon,u}, Y^{\epsilon,u})$. In Section \ref{Sec3} we prove a priori bounds for the process $\eta^{\epsilon,u}$ with the aid of the elliptic Kolmogorov equation while Section \ref{Sec4} is devoted to the analysis of the limit of the pairs $(\eta^{\epsilon,u}, P^{\epsilon,\Delta})$. In Section \ref{MDPsec} we prove the MDP. Finally, Appendix \ref{AppA} contains some classical regularity results for stochastic convolutions adapted to our multiscale setting while Appendix \ref{AppB} contains the proof of Lemma \ref{IVitodeclem}.

 	\section{Notation and Assumptions}\label{notation}
 	
 We denote by $\h$ the Hilbert space $L^2(0,L)$ endowed with the usual inner product $\langle \cdot,\cdot\rangle_\h$. The norm induced by the inner product is denoted by $\|\cdot\|_\h$.
 Throughout this paper, $\oplus$ denotes the Hilbert space direct sum.  The closed unit ball of any Banach space $\mathcal{X}$, i.e. the set $\{x\in\mathcal{X} : \|x\|_{\mathcal{X}}\leq 1\}$, will be denoted by $B_\mathcal{X}$. The lattice notation $\wedge, \vee$ is used to indicate minimum and maximum respectively.

 For $\theta> 0$, we denote by $H^\theta(0,L)$ the fractional Sobolev space of $x\in\h$ such that
 \begin{equation*}
 [x]_{H^\theta}:=\int_{[0,L]^2}\frac{|x(\xi_2)-x(\xi_1)|^2}{|\xi_2-\xi_1|^{2\theta+1}}d\uplambda_2(\xi_1, \xi_2)<\infty\;,
 \end{equation*}
 where $\uplambda_2$ denotes Lebesgue measure on $[0,L]^2$. $H^\theta(0,L)$ is a Banach space when endowed with the norm $\|\cdot\|_{H^\theta}:=\|\cdot\|_\h+[\cdot]_{H^\theta}$.

 Moreover, for $T>0$ and $\beta\in[0,1)$, we denote by $C^\beta([0,T];\h)$ the space of $\beta$-H\"older continuous $\h$-valued paths defined on the interval $[0,T]$. $C^\beta([0,T];\h)$ is a Banach space when endowed with the norm \begin{equation*}
 \|X\|_{C^\beta([0,T];\h)}:= \|X\|_{C([0,T];\h)}+[X]_{C^\beta([0,T];\h)}:=\sup_{t\in[0,T]}\|X(t)\|_{\h}+\sup_{\overset{s,t\in[0, T]}{ t\neq s}}\frac{\|X(t)-X(s)\|_\h}{|t-s|^\beta}\;.
 \end{equation*}
For any two Banach spaces $\mathcal{X}, \mathcal{Y}$ and $k\in \N$ we denote the space of $k$-linear bounded operators $Q: \mathcal{X}^k\rightarrow\mathcal{Y}$ by $\mathscr{L}^k(\mathcal{X}; \mathcal{Y})$. The latter is a Banach space when endowed with the norm
 \begin{equation*}
 \|Q\|_{\mathscr{L}^k(\mathcal{X}; \mathcal{Y})}:=\sup_{x\in B^k_\mathcal{X}}\|Qx\|_{\mathcal{Y}}\;.
\end{equation*}
 When the domain coincides with the co-domain, we use the simpler notation $\mathscr{L}^k(\mathcal{X})$ while for $k=1$ we often omit the superscript and write $\mathscr{L}(\mathcal{X};\mathcal{Y})\equiv \mathscr{L}^1(\mathcal{X}; \mathcal{Y})$.

 The spaces of trace-class and Hilbert-Schmidt linear operators $B:\h\rightarrow\h$ are denoted by $\mathscr{L}_1(\h)$ and $\mathscr{L}_2(\h)$ respectively. The former is a Banach space
 when endowed with the norm
 \begin{equation*}
 \|B\|_{\mathscr{L}_1(\h)}:=\text{tr}(\sqrt{B^*B})
 \end{equation*}
 while the latter is a Hilbert space when endowed with the inner product
 \begin{equation*}
 \langle B_1, B_2\rangle_{\mathscr{L}_2(\h)}:= \text{tr}( B_2^* B_1).
 \end{equation*}

 The class of (globally) Lipschitz real-valued functions on $\h$ is denoted by $Lip(\h)$ and the space of $k$-times Fr\'echet differentiable real-valued functions on $\h$ with bounded and uniformly continuous derivatives up to the $k$-th order ($k\in\N$) is denoted by $C_b^k(\h)$. The latter is a Banach space when endowed with the norm
 \begin{equation*}
 \|X\|_{C_b^k(\h)}:=\sup_{x\in\h} |X(x)|+ \sup_{x\in\h}\|DX(x)\|_{\h}+ \sum_{i=2}^{k}\sup_{x\in\h}\|D^iX(x)\|_{\mathscr{L}^{i-1}(\h)}\;.
 \end{equation*}
 For $k=0$ we often omit the superscript and write $C_b(\h)\equiv C_b^0(\h)$ for the space of bounded uniformly continuous functions on $\h$.

 The operators $\mathcal{A}_1, \mathcal{A}_2,$ appearing in \eqref{model}, are uniformly elliptic  second-order differential operators with continuous coefficients on $[0,L]$. The operators $\mathcal{N}_1$ and $\mathcal{N}_2$ act on the boundary $\{0,L\}$ and can be either the identity operator (corresponding to Dirichlet boundary conditions) or first-order differential operators of the type
 \begin{equation*}
 \mathcal{N}u(\xi)=b(\xi)u'(\xi)+c(\xi)u(\xi)\;,\;\xi\in\{0, L\}
 \end{equation*}
 for some $b,c\in C^1[0,L]$ such that $b\neq 0$ on $\{0,L\}$ (corresponding to Neumann or Robin boundary conditions).

 For $i=1,2,$ $A_i$  denotes the realization of the differential operator $\mathcal{A}_i$ in $\h$, endowed with the boundary condition $\mathcal{N}_i$. It is  defined on the dense subspace
 \begin{equation*}
 Dom(A_i)=\{ x\in H^2(0,L): \mathcal{N}_ix(0)=\mathcal{N}_ix(L) =0  \}
 \end{equation*}
 and generates a $C_0$, analytic semigroup of operators $S_i=\{S_i(t)\}_{t\geq 0}\subset\mathscr{L}(\h)$.

 \noindent Regarding the spectral properties of $A_i$, we make the following assumptions:
 \begin{customthm}{1(a)}\label{A1a} For $i=1,2$ the operator $-A_i$ is self-adjoint. As a result (see Theorem 8.8.37 in \cite{gilbarg2015elliptic}),  there exists a countable complete orthonormal basis $\{e_{i,n}\}_{n\in\N}\subset\h$ of eigenvectors of $-A_i$. The corresponding  sequence of nonnegative eigenvalues is denoted by $\{a_{i,n}\}_{n\in\N}$.
 \end{customthm}
 \noindent As a consequence, for each $x\in\h, t\geq 0$, $i=1,2$, we have
 \begin{equation}
 \label{semignorm}
 \|S_i(t)x\|_\h^2=\sum_{n=1}^{\infty}e^{-2a_{i,n}t}\langle x, e_{i,n}\rangle^2_\h\leq e^{-2t \underset{n\in\N}{\inf} a_{i,n}}\|x\|_\h\leq \|x\|_\h\;.
 \end{equation}
 \begin{customthm}{1(b)} \label{A1b}
 	For $i=1,2$ we assume that \begin{equation}\label{eigenbound}
 	\sup_{n\in\N}\|e_{i,n}\|_{L^\infty(0,L)}<\infty.
 	\end{equation}
 \end{customthm}
 \begin{customthm}{1(c)} \label{A1c}
 	$A_2$ is self-adjoint and satisfies the strict dissipativity condition
 	\begin{equation}\label{A_2dis}\lambda:=\inf_{n\in\N} a_{2,n}>0.\end{equation}
 \end{customthm}
 \noindent Under this assumption it is straightforward to verify that
 \begin{equation}\label{S2decay}
 \|S_2(t)\|_{\mathscr{L}(\h)}\leq e^{-\lambda t}\;,\; t\geq 0.
 \end{equation}

 \begin{rem} Without loss of generality, we can replace the operator $A_1$ by $\tilde{A}_1=A_1-cI$ for some $c>0$ and the reaction term $f$ in \eqref{model}, by $\tilde{f}(\xi, x(\xi), y(\xi)):=f(\xi, x(\xi), y(\xi))+cx(\xi)$. The slow equation is invariant under this  transformation and, in light of Hypothesis \ref{A1a}, it follows that
 	$\|\tilde{S}_1(t)\|_{\mathscr{L}(\h)}\leq e^{-c t}$. 
 	Throughout the rest of this work we will be using $\tilde{A_1},\tilde{S_1}$ and $ \tilde{f}$ with no further distinction in notation.
 	
 \end{rem}
 Let $i=1,2$ and $\theta\geq 0$.  In view of Hypotheses \ref{A1a} and \ref{A1c}, along with the previous remark, it follows that $0$ is in the resolvent set of $A_i$. Hence the operator $-A_i$, restricted to its image, has a densely defined bounded inverse  $(-A_i)^{-1}$ which can then be uniquely extended to all of $\h$.
 One can then define $(-A_i)^{-\theta}$ via interpolation and show that it is also injective.

 Letting $(-A_i)^{\frac{\theta}{2}}:= ((-A_i)^{-\frac{\theta}{2}})^{-1}$ we define
 $\h_i^\theta:= Dom(-A_i)^\frac{\theta}{2}= Range(-A_i)^{-\frac{\theta}{2}}\subset\h$. The latter is a Banach space when endowed with the norm \begin{equation*}
 \|x\|_{\h_i^\theta}:=\big\|(-A_i)^\frac{\theta}{2}x\big\|_\h\;.
 \end{equation*}
 This norm is equivalent, due to injectivity, to the graph norm (see \cite{lunardi2012analytic}, Chapter 2.2).

 \begin{rem} For $\theta\in(0,\frac{1}{2})$ the spaces $H^\theta(0,L)$ and $\h_i^\theta$ coincide, in light of the identity
 	\begin{equation*}
 	\hspace*{-0.4cm}H^\theta(0,L)=\h^\theta_i
 	=\big\{x\in\h : \|x\|_{\theta,\infty}:=\sup_{t\in (0,1] }t^{-\theta/2} \|S_i(t)x-x\|_\h<\infty      \big\},
 	\end{equation*}
 	which holds with equivalence of norms. The latter implies that for each $t\geq 0$, the linear operator $S_i(t)-I\in\mathscr{L}(H^\theta;\h)$ and there exists a constant $C>0$ such that
 	\begin{equation}\label{sobcont}
 	\big\|S_i(t)-I\big\|_{\mathscr{L}(H^\theta;\h)}\leq Ct^{\theta/2}.
 	\end{equation}
 \end{rem}

 The analytic semigroups $S_i$ possess the following regularizing properties (see e.g. section 4.1.1 in \cite{cerrai2001second})     :

 (i) For $0\leq s\leq r\leq \frac{1}{2}$ and $t>0$, $S_i$ maps $H^{s}(0,L)$ to $H^{r}(0,L)$ and
 \begin{equation}\label{Sobosmoothing}
 \|S_i(t)x\|_{H^{r}}\leq C_{r,s}(t\wedge 1)^{-\frac{r-s}{2}}e^{c_{r,s}t}\|x\|_{H^{s}}\;\;,\;x\in H^{s}(0,L),
 \end{equation}
 for some positive constants $c_{r,s}, C_{r,s}$.

 (ii) $S_i$ is \textit{ultracontractive}, i.e. for $t>0,$ $S_i(t)$ maps $\h$ to $L^{\infty}(0,L)$ and furthermore, for any $1\leq p\leq r\leq\infty$,
 \begin{equation}\label{Lpsmoothing}
 \|S_i(t)x\|_{L^r(0,L)}\leq C(t\wedge 1)^{-\frac{r-p}{2pr}}\|x\|_{L^p(0,L)}\;\;,\;x\in L^p(0,L).
 \end{equation}

 \begin{rem}\label{selfad} The assumption that $A_1$ is self-adjoint is made to simplify the exposition and is not necessary for the results of this paper to hold. Indeed, assuming that $A_1$ has $C^1$ coefficients and in view of section 2.1 of \cite{cerrai2003stochastic}, we can write $A_1=C_1+L_1$, where $C_1$ is a non-positive uniformly elliptic self-adjoint operator  and $L_1$ a densely defined first-order operator. Moreover, we have $Dom(L_1)=Dom(L_1^*) =Dom((-C_1)^\frac{1}{2})$. The fractional powers of $-A_1$ can then be substituted throughout by fractional powers of $-C_1$. Finally, the mild formulations for $X^{\epsilon,u}$ and $\eta^{\epsilon,u}$ can be re-expressed in terms of the analytic semigroup $S_{C_1}$, generated by $C_1$, with the addition of a linear term corresponding to the operator $L_1$ (see Definition 3.1 and Proposition 3.1 in \cite{cerrai2003stochastic}).
 \end{rem}

 The next set of assumptions concerns the regularity of the nonlinear reaction terms in \eqref{model}.
 In particular, we assume that	$f,g:[0,L]\times\R^2\rightarrow\R$ are measurable functions and:
 \begin{customthm}{2(a)} \label{A2a}
 	For almost all $\xi\in(0,L)$, the map $(\mathrm{x}, \mathrm{y})\mapsto f(\xi,\mathrm{x}, \mathrm{y})$ is in $C^2(\R^2)$ and its derivatives are uniformly bounded with respect to  $\xi,\mathrm{x}, \mathrm{y}$.
 \end{customthm}
 \begin{customthm}{2(b)} \label{A2b}
 	(i) For almost all $\xi\in(0,L)$ and all $\mathrm{y}\in\R$, the map $\mathrm{x}\mapsto g(\xi,\mathrm{x}, \mathrm{y})$ is in $C^2(\R)$ and its derivatives are uniformly bounded with respect to $\xi,\mathrm{x},\mathrm{y}$ .
 	
 	\noindent  (ii) For almost all $\xi\in(0,L)$ and all $\mathrm{x}\in\R$, the map $\mathrm{y}\mapsto g(\xi,\mathrm{x}, \mathrm{y})$ is in $C^3(\R)$ with uniformly bounded derivatives with respect to $\xi,\mathrm{x},\mathrm{y}$ and
 	\begin{equation}\label{diss2}
 	\sup_{\xi,\mathrm{x}, \mathrm{y}}\big| \partial_{\mathrm{y}}g(\xi,\mathrm{x}, \mathrm{y})\big|=:L_g<\lambda,
 	\end{equation}
 	\noindent with $\lambda$ as in \eqref{A_2dis}.

 \end{customthm}
 \begin{customthm}{2(c)}\label{A2c} With $\lambda, L_g$ as in Hypothesis \ref{A2b} we assume that
 	\begin{equation}\label{extradiss}
 	\omega:= \frac{\lambda-3L_g}{2}>0.
 	\end{equation}
 \end{customthm}
 \noindent Hypothesis \ref{A2c} is used to prove that a partial Fr\'{e}chet derivative of the solution of the Kolmogorov equation associated to the fast process converges, as $\epsilon\to0$, to an operator-valued map that is Lipschitz continuous with respect to its arguments (see Lemma \ref{varicontlem} and Corollary \ref{Psicont}).

 \noindent 	The last set of assumptions concerns the behavior of the diffusion coefficient $\sigma$. In particular, we assume that $\sigma:[0,L]\times\R^2\rightarrow\R$ is measurable and satisfies either :
 \begin{customthm}{3(a)}\label{A3a} There exists $c>0$ and $\nu\in[0,1/2)$ such that for almost all $\xi\in[0,L]$ and all $(\mathrm{x}, \mathrm{y})\in\R^2$
 	\begin{equation}
 	\label{sigma}
 	|\sigma(\xi,\mathrm{x}, \mathrm{y})|\leq c(1+|\mathrm{x}|+| \mathrm{y}|^\nu ).
 	\end{equation}
 \end{customthm}
 or:
 \begin{customthm}{3(a')}\label{A3a'} There exist $c_1, c_2>0$  such that for almost all $\xi\in[0,L]$ and all $(\mathrm{x}, \mathrm{y})\in\R^2$
 	\begin{equation}
 	\label{sigma'}
 	c_1\leq\sigma(\xi,\mathrm{x}, \mathrm{y})\leq c_2.
 	\end{equation}
 \end{customthm}
 \begin{rem}The diffusion coefficient $\sigma$ is allowed to grow at most like $|\mathrm{y}|^{1/2}$ in the third argument. This is due to the fact that the stochastic controls are only known to be square integrable. As a result we can obtain estimates for $Y^{\epsilon,u}$ in $L^p([0,T];\h)$, for $p\leq 2$ (see \eqref{Yint} and \eqref{yapriori} in Section \ref{Sec2} below). \end{rem}
 \begin{customthm}{3(b)}\label{A3b} There exists $L_\sigma>0$ such that for almost all $\xi\in[0,L],$ the map $(\mathrm{x}, \mathrm{y})\mapsto \sigma(\xi,\mathrm{x}, \mathrm{y})$ is $L_\sigma$-Lipschitz continuous.
 \end{customthm}
 \begin{rem} The a priori estimates in Sections \ref{Sec2}-\ref{Sec3} hold by assuming only Hypothesis \ref{A3a}. For the analysis of the limit  (Section \ref{Sec4})  we assume \ref{A3a} along with \ref{A3b}. Finally, we  strengthen the assumptions on $\sigma$ and use the strictly stronger Hypothesis \ref{A3a'} along with \ref{A3b} to prove the Laplace Principle upper and lower bounds (Sections \ref{LPup} and \ref{LPlo} respectively).
 \end{rem}

 The reaction terms $f,g$ induce nonlinear superposition (or Nemytskii) operators denoted, respectively, by $F,G:\h\times\h\rightarrow\h$  and defined by
 \begin{equation}\label{Nem}
 F(x,y)(\xi)=f(\xi,x(\xi),y(\xi)), \;\;G(x,y)(\xi)=g(\xi,x(\xi),y(\xi))\;,\;\;\xi\in[0,L].
 \end{equation}
 In view of Hypotheses \ref{A2a} and \ref{A2b}, $F$ and $G$ are (globally) Lipschitz continuous. Moreover, $F$ and $G$ are G\^ateaux differentiable with respect to both variables and along the direction of any $\chi\in\h$. Their G\^ateaux derivatives are given by
 \begin{equation}\label{DF}
 D_xF(x,y)(\chi)(\xi)=\partial_{\mathrm{x}}f(\xi, x(\xi), y(\xi))\chi(\xi)\;,\;\; D_yF(x,y)(\chi)(\xi)=\partial_{\mathrm{y}}f(\xi, x(\xi), y(\xi))\chi(\xi)
 \end{equation}
 and
 \begin{equation*}
 D_xG(x,y)(\chi)(\xi)=\partial_{\mathrm{x}}g(\xi, x(\xi), y(\xi))\chi(\xi)\;,\;\; D_yG(x,y)(\chi)(\xi)=\partial_{\mathrm{y}}g(\xi, x(\xi), y(\xi))\chi(\xi)
 \end{equation*}
 for $\xi\in[0,L]$.
 Furthermore, for each fixed $y\in\h$ and $\chi_1\in \h$, the map
 $$\h\ni x\longmapsto D_xF(x,y)(\chi_1)\in L^1(0,L)$$ is G\^ateaux differentiable along the direction of any $\chi_2\in\h$.
 Equivalently, the nonlinear operator $F$, when considered as a map from $\h$ to $L^1(0,L)$, is twice G\^ateaux differentiable with respect to $x$, along any direction in $\h\times\h$. Its second partial G\^ateaux derivative is given by
 \begin{equation}\label{DxxF}
 D^2_{x}F(x,y)(\chi_1, \chi_2)(\xi)=\partial^2_{\mathrm{x}\mathrm{x}}f(\xi, x(\xi), y(\xi))\chi_1(\xi)\chi_2(\xi)\;,\;\xi\in[0,L].
 \end{equation}
 \begin{rem} Note that, for fixed $x,y$, all the first-order partial G\^ateaux derivatives above are in $\mathscr{L}(\h)$ and $D^2_{x}F(x,y)\in\mathscr{L}^2(\h; L^1(0,L))$. Nevertheless, $F$ and $G$,  considered as maps from $\h\times\h$ to $\h$, are not Fr\'echet differentiable with respect to any of their variables. In fact, it can be shown that a Nemytskii operator from $\h$ to $\h$ is Fr\'echet differentiable if and only if it is an affine map (see Proposition 2.8 in \cite{ambrosetti1995primer}).
 	
 \end{rem}

 The diffusion coefficient $\sigma$ is considered as a function multiplied by the noise and hence induces, for each $x,y\in\h$, a multiplication operator
 \begin{equation*}
 \big[\Sigma(x,y)\chi\big](\xi):=\sigma(\xi, x(\xi), y(\xi))\chi(\xi),\; \chi\in\h,\; \xi\in(0,L).
 \end{equation*}
 In view of Hypothesis \ref{A3a} it follows that $\Sigma(x,y)\in\mathscr{L}(L^\infty(0,L);\h)\cap \mathscr{L}(\h; L^1(0,L))$. Moreover, under Hypothesis \ref{A3a'}, we have $\Sigma(x,y)\in\mathscr{L}(\h)$.

 For the purposes of this paper we consider a Polish space to be a completely metrizable, separable topological space. For a given topological space $\mathcal{E}$ we denote the Borel $\sigma$-algebra by $\mathscr{B}(\mathcal{E})$ and the space of Borel probability measures on $\mathcal{E}$ by $\mathscr{P}(\mathcal{E})$. If $\mathcal{E}$ is Polish then $\mathscr{P}(\mathcal{E})$, endowed with the topology of weak convergence of measures, is also a Polish space.

 	\section{Weak convergence method and moderate deviations}\label{weakconv}

 In this section we review the weak convergence approach to large and moderate deviations (see \cite{dupuis2011weak} as well as the more recent \cite{budhiraja2019analysis}) and then we state our main results of the paper on the averaging principle for the controlled process $\eta^{\epsilon,u}$ (see \eqref{etau}) and on the moderate deviations for $\{X^{\epsilon}\}$.

 Let $j=1,2$ and consider the cylindrical Wiener process $w_j:[0,\infty)\times\h\rightarrow L^2(\Omega)$ appearing in \eqref{model}. For each fixed $t$, $\{w_j(t,\chi)\}_{\chi\in\h}$ is a Gaussian family of random variables and for each $t_1,t_2\geq 0$, $\chi_1,\chi_2\in\h$
 $$\ex[w_j(t_1,\chi_1)w_j(t_2,\chi_2)         ]=t_1\wedge t_2\langle \chi_1, \chi_2\rangle_\h.$$
 The first step of the weak convergence method relies on a variational representation for functionals of the driving noise. For the infinite-dimensional setting of this paper, we will use the variational representation for $Q$-Wiener processes that was proved in \cite{budhiraja2008large}, Theorem 3. In order to apply this result in the context of space-time white noise, we introduce  a separable Hilbert space $(\h_1, \langle .\;,.\rangle_{\h_1})$ such that $\h$ is a linear subspace of $\h_1$ and the inclusion map $\h\overset{i}{\rightarrow}\h_1$ is Hilbert-Schmidt (for more details on this construction we refer the reader to \cite{gross1962measurable}, \cite{gross1967abstract}). Given a complete orthonormal basis $\{e_{j,n}\}_{n\in\N}\subset\h$, the process
 $$\tilde{w}_j(t)=\sum_{n=1}^{\infty} w_j(t,e_{j,n})\;i(e_{j,n})\;\;,t\geq 0$$
 is an $\h_1$-valued $Q$-Wiener process with trace-class covariance operator $Q=ii^*\in\mathscr{L}_1(\h_1)$. In particular, for each $t,s\geq 0$ and $\chi_1,\chi_2\in\h_1$ we have
 $$ \ex[\langle \chi_1, \tilde{w}_j(t)\rangle_{\h_1} \langle \chi_2, \tilde{w}_j(s)\rangle_{\h_1}  ]= t\wedge s\langle \chi_1, Q\chi_2\rangle_{\h_1}= t\wedge s\langle i^*(\chi_1), i^*(\chi_2)\rangle_{\h}.$$
This construction allows us to use the following representation.
 \begin{thm}[\cite{budhiraja2008large}, Theorem 3]\label{budhirep} Let $T<\infty$, $\Uplambda: C([0,T];\h_1)\rightarrow \R$ be a bounded, Borel measurable map and $W$ be an $\h_1$-valued $Q$-Wiener process. Moreover, let  $\mathcal{P}^T(\h)$ denote the family of $\h$-valued, progressively-measurable stochastic processes for which
 	\begin{equation*}\label{L2-controls}
 	\pr\bigg[	\int_{0}^{T}\|u(s)\|^2_\h\;ds<\infty\bigg]=1.
 	\end{equation*}
 Then:
 	\begin{equation*}\label{varrep1}
 	-\log\;\ex[\exp(-\Uplambda(W))]=\inf_{u\in\mathcal{P}^T(\h)}\ex\bigg[ \frac{1}{2} \int_{0}^{T}  \|u(s)\|^2_\h\;ds+ \Uplambda\bigg( W+\int_{0}^{\cdot} u(s)\; ds   \bigg)  \bigg].
 	\end{equation*}
 \end{thm}
 \noindent  Since the processes $\tilde{w}_1, \tilde{w}_2$ are independent, it follows that $\widetilde{w}=(\tilde{w}_1,\tilde{w}_2 )$ is an $\h_1\oplus\h_1$-valued Wiener process  with covariance operator $(Q,Q)$. Hence, we can replace  $W$, $\h_1$ and $\h$ by $\widetilde{w}$, $\h_1\oplus\h_1$ and $\h\oplus\h$  respectively to obtain
 \begin{equation*}\label{varrepsum}
 -\log\;\ex[\exp(-\Uplambda(\widetilde{w}))]=\inf_{u\in\mathcal{P}^T(\h\oplus\h)}\ex\bigg[ \frac{1}{2} \int_{0}^{T} \big( \|u_1(s)\|^2_\h+ \|u_2(s)\|^2_\h\big)ds+ \Uplambda\bigg( \widetilde{w}+\int_{0}^{\cdot} u(s)\; ds   \bigg)  \bigg],
 \end{equation*}
 where $u=(u_1,u_2)$ and $\Uplambda: C([0,T];\h_1\oplus\h_1)\rightarrow \R$ is measurable and bounded. In order to obtain a representation in the moderate deviation scaling, we replace $u$ and $\Uplambda$ by $h(\epsilon)u$ and $h^2(\epsilon)\Uplambda$ respectively and then divide throughout by $h^2(\epsilon)$ to deduce that
 \begin{equation}\label{varrep3}
 -\frac{1}{h^2(\epsilon)}\log\;\ex\big[ e^{-h^2(\epsilon)\Uplambda(\tilde{w})}\big]=\inf_{u\in\mathcal{P}^T(\h\oplus\h)}\ex\bigg[ \frac{1}{2} \int_{0}^{T}  \big(\|u_1(s) \|^2_\h+\|u_2(s) \|^2_\h \big)\;ds+ \Uplambda\bigg( \tilde{w}+h(\epsilon)\int_{0}^{\cdot} u(s)\; ds  \bigg)  \bigg].
 \end{equation}

Now, the system \eqref{model} can be re-expressed in the mild formulation as
 \begin{equation*}\label{modelmild}
 \left \{\begin{aligned}
 &	X^{\epsilon}(t)= S_1(t)x_0+\int_{0}^{t} S_1(t-s)F( X^{\epsilon}(s), Y^{\epsilon}(s))ds\\&\quad\quad\quad+   \sqrt{\epsilon}\int_{0}^{t} S_1(t-s)\Sigma\big( X^{\epsilon}(s),Y^{\epsilon}(s)\big) dw_1(s)\\&
 Y^{\epsilon}(t)= S_2\bigg(\frac{t}{\delta}\bigg)y_0+\frac{1}{\delta}\int_{0}^{t} S_2\bigg(\frac{t-s}{\delta}\bigg)G( X^{\epsilon}(s), Y^{\epsilon}(s))ds\\& \quad\quad\quad+   \frac{1}{\sqrt{\delta}}\int_{0}^{t} S_2\bigg(\frac{t-s}{\delta}\bigg) dw_2(s),
 \end{aligned} \right.
 \end{equation*}
 where we recall that $A_1, A_2$ are the realizations of $\mathcal{A}_1, \mathcal{A}_2$ on $\h$ with the boundary conditions $\mathcal{N}_1, \mathcal{N}_2$, $\{S_1(t)\}_{t\geq0}$ is generated by $A_1$ and $\{S_2(t/\delta)\}_{t\geq0}$ is generated by $A_2/\delta$.

 For each fixed  $\epsilon, T$ and initial conditions $x_0,y_0\in\h$, the existence and uniqueness of a mild solution $(	X^{\epsilon,x_0,y_0}(t), 	Y^{\epsilon,x_0,y_0}(t)  )$ that takes values on $C([0,T];\h)^2$ implies the existence of a measurable solution map
 \begin{equation*}
 \mathcal{I}^{\epsilon,x_0,y_0}: C([0,T];\h_1\oplus\h_1)\longrightarrow C([0,T];\h)
 \end{equation*}
 such that
 \begin{equation*}
 \eta^{\epsilon}(t)\equiv\eta^{\epsilon,x_0,y_0}(t):=\frac{1}{\sqrt{\epsilon}h(\epsilon)}\big(	X^{\epsilon,x_0,y_0}(t)-\bar{X}^{x_0}(t)\big)=\mathcal{I}^{\epsilon,x_0,y_0}(\widetilde{w}).
 \end{equation*}
 Here, $\bar{X}^{x_0}$ is the solution of the averaged equation \eqref{x-aved}. Returning to \eqref{varrep3}, we replace $\Uplambda$ by $\Lambda\circ\mathcal{I}^{\epsilon,x_0,y_0}$, where $\Lambda:C([0,T];\h)\rightarrow \R$ is continuous and bounded, to obtain the representation
 \begin{equation}\label{varrep}
 -\frac{1}{h^2(\epsilon)}\log\;\ex\big[ e^{-h^2(\epsilon)\Lambda(\eta^\epsilon)}\big]=\inf_{u\in\mathcal{P}^T(\h\oplus\h)}\ex\bigg[ \frac{1}{2} \int_{0}^{T} \big( \|u_1(t) \|^2_\h+\|u_2(t) \|^2_\h\big)\;dt+ \Lambda\big(\eta^{\epsilon,u}  \big)  \bigg].
 \end{equation}
 The process $\eta^{\epsilon,u}$ on the right-hand side is defined by \begin{equation}\label{etau1}
 \eta^{\epsilon,u}(t)=\frac{X^{\epsilon,u}(t)-\bar{X}(t)}{\sqrt{\epsilon}h(\epsilon)}
 \end{equation}
 and $X^{\epsilon,u}$ corresponds to the controlled system of stochastic reaction-diffusion equations
 \begin{equation}\label{controlledsystem}
 \left \{\begin{aligned}
 &dX^{\epsilon,u}(t)= \big[A_1 X^{\epsilon,u}(t)+ F\big( X^{\epsilon,u}(t),Y^{\epsilon,u}(t)\big)+\sqrt{\epsilon}h(\epsilon)\Sigma\big(X^{\epsilon,u}(t),Y^{\epsilon,u}(t)\big)u_1(t)\big]dt\\&\quad\quad\quad\quad+\sqrt{\epsilon}\Sigma\big(X^{\epsilon,u}(t),Y^{\epsilon,u}(t)\big)dw_1(t)
 \\&
 dY^{\epsilon,u}(t)= \frac{1}{\delta}\big[A_2 Y^{\epsilon,u}(t)+ G\big( X^{\epsilon,u}(t),Y^{\epsilon,u}(t)\big)+\sqrt{\delta}h(\epsilon)u_2(t)\big]dt +\frac{1}{\sqrt{\delta}}\; dw_2(t)
 \\&
 X^{\epsilon,u} (0)= x_0\in\h\;, Y^{\epsilon,u}(0)=y_0\in\h.
 \end{aligned} \right.
 \end{equation}
 The mild solution of the latter is given by a pair of controlled stochastic processes  that satisfy
 \begin{equation}\label{controlledsystemild}
 \left \{\begin{aligned}
 &	X^{\epsilon,u}(t)= S_1(t)x_0+\int_{0}^{t} S_1(t-s)F( X^{\epsilon,u}(s), Y^{\epsilon,u}(s))ds\\&\quad\quad\quad+   \sqrt{\epsilon}h(\epsilon)\int_{0}^{t} S_1(t-s)\Sigma\big( X^{\epsilon,u}(s),Y^{\epsilon,u}(s)\big)u_1(s)ds\\&\quad\quad\quad+   \sqrt{\epsilon}\int_{0}^{t} S_1(t-s)\Sigma\big( X^{\epsilon,u}(s),Y^{\epsilon,u}(s)\big) dw_1(s)\\&
 Y^{\epsilon,u}(t)= S_2\bigg(\frac{t}{\delta}\bigg)y_0+\frac{1}{\delta}\int_{0}^{t} S_2\bigg(\frac{t-s}{\delta}\bigg)G( X^{\epsilon,u}(s), Y^{\epsilon,u}(s))ds\\& \quad\quad\quad+\frac{h(\epsilon)}{\sqrt{\delta}}\int_{0}^{t} S_2\bigg(\frac{t-s}{\delta}\bigg) u_2(s)ds+   \frac{1}{\sqrt{\delta}}\int_{0}^{t} S_2\bigg(\frac{t-s}{\delta}\bigg) dw_2(s).
 \end{aligned} \right.
 \end{equation}
 Next, let $N>0$ and define \begin{equation}
 \mathcal{P}^T_N=\bigg\{u=(u_1,u_2)\in\mathcal{P}^T(\h\oplus \h) :
 \int_{0}^{T}\big(\|u_1(s)\|^2_{\h}+\|u_2(s)\|^2_{\h}\big)ds\leq N,\;\pr-\text{a.s.}	\bigg\}.
 \end{equation}

 \noindent As in Theorem 10 of \cite{budhiraja2008large} and for each $u \in\mathcal{P}^T_N$ and $\epsilon>0$, there is a unique pair $(	X^{\epsilon,u}, 	Y^{\epsilon,u})$ in $L^p(\Omega ;C([0,T];\h)\times C([0,T];\h))$  that satisfies \eqref{controlledsystemild}.

 Now, proving a Laplace Principle for $\eta^{\epsilon}$ amounts to finding the limit as $\epsilon \to 0$ of the left hand side in \eqref{varrep}. This is equivalent to proving an LDP for the family $\{\eta^\epsilon,\epsilon>0\}$ with speed $h^2(\epsilon)$, which in turn is equivalent to an MDP for $\{X^{\epsilon}, \epsilon>0\}$. This is the path that we follow in this paper for proving  the MDP for the family $\{X^{\epsilon}, \epsilon>0\}$ in $C([0,T];\h)$.
 Also, as it is shown in \cite{boue1998variational}, the representation implies that we can consider, without loss of generality, $u = u^{\epsilon}\in\mathcal{P}^T_N$ for a sufficiently large but fixed $N>0$ (see also \cite{budhiraja2000variational}, p.22).

 	As discussed in the introduction, the analysis of the limiting behavior of $\eta^{\epsilon,u}$ is more complicated, compared to that of $X^{\epsilon,u}$, due to the singular coefficient $1/\sqrt{\epsilon} h(\epsilon)$. In view of \eqref{etau1} and \eqref{controlledsystemild} we can write
 \begin{equation}\label{etadec1}
 \hspace*{-0.2cm}
 \begin{aligned}
 \eta^{\epsilon,u}(t)= &\frac{1}{\sqrt{\epsilon}h(\epsilon)}\int_{0}^{t} S_1(t-s)\big[F\big(\bar{X}(s)+\sqrt{\epsilon}h(\epsilon)\eta^{\epsilon,u}(s), Y^{\epsilon,u}(s) \big)- F\big(\bar{X}(s), Y^{\epsilon,u}(s) \big) \big] ds
 \\&
 + \int_{0}^{t} S_1(t-s)\Sigma\big(X^{\epsilon,u}(s), Y^{\epsilon,u}(s) \big)u_1(s) ds\\&+\frac{1}{h(\epsilon)} \int_{0}^{t} S_1(t-s)\Sigma\big(X^{\epsilon,u}(s), Y^{\epsilon,u}(s) \big)dw_1(s)\\&
 +\frac{1}{\sqrt{\epsilon}h(\epsilon)}\int_{0}^{t} S_1(t-s)\big[F\big(\bar{X}(s), Y^{\epsilon,u}(s) \big)- \bar{F}\big(\bar{X}(s) \big)\big]ds,
 \end{aligned}
 \end{equation}
 where $h(\epsilon)\to\infty$, $\sqrt\epsilon h(\epsilon)\to0$ as $\epsilon\to 0$ and $\bar{F}$ denotes the averaged Nemytskii operator \eqref{Nemave}.

 The asymptotic analysis of the first term above, as $\epsilon\to0$, is straightforward. Indeed, its limiting behavior is captured by
 $$\int_{0}^{t} S_1(t-s)D_xF\big(\bar{X}(s), Y^{\epsilon,u}(s)\big)\eta^{\epsilon,u}(s)ds,$$
 (see \eqref{DF} and Proposition \ref{linlim}).
 Moreover, the second term is of order $1$ while the third is expected to vanish in the limit. In contrast,  the last term requires a more delicate approach. This is connected to the solution of the following elliptic \textit{Kolmogorov equation} on $\h$:
 \begin{equation}\label{Kolmeq}
 \begin{aligned}
 c(\epsilon)\Phi^\epsilon_\chi(x,y)-\mathcal{L}^x \Phi^\epsilon_\chi(x,y)=\blangle  F(x,y)-\bar{F}(x), \chi\brangle_\h\;,
 \end{aligned}
 \end{equation}
 where $\chi,x\in\h$, $y\in Dom(A_2)$ and $c(\epsilon)$ vanishes as $\epsilon\to0$. The exact dependence of $c$ on $\epsilon$ will be specified later (see Section \ref{4}). For $\psi:\h\times\h\rightarrow\R$ such that for each fixed $x,y\in\h,$ $\psi(x,\cdot)\in C^2(\h)$ and $D^2_y\psi(x,y)\in\mathscr{L}_2(\h)$,	the \textit{Kolmogorov operator} $\mathcal{L}^x$ is a  second-order differential operator defined by
 \begin{equation}\label{Kolmopintro}
 \mathcal{L}^x\psi(x,y)= \frac{1}{2}\text{tr} \big[ D^2_{y}\psi(x,y)  \big]+\blangle D_y\psi(x,y), A_2y+ G(x,y) \brangle_\h\;,\; y\in Dom(A_2).
 \end{equation}
 Formally,  $\mathcal{L}^x$ is called the infinitesimal generator of the (uncontrolled) fast process $Y^{x}$ with "frozen" slow component $x$. The latter satisfies the stochastic evolution equation
 \begin{equation}
\label{Yfrozen}
 \left\{ \begin{aligned}
 &dY^{x,y}(t)= A_2 Y^{x,y}(t)dt+ G\big( x,Y^{x,y}(t)\big)dt +dw_2(t)\\&
 Y^{x,y}(0)=y.
 \end{aligned}\right.
 \end{equation}
 \begin{rem}
 	If $A_2\in\mathscr{L}(\h)$, and hence $Dom(A_2)=\h$, then $\mathcal{L}^x$ coincides with the infinitesimal generator of the transition semigroup $P^x$ of the Markov process $Y^x$ defined by \begin{equation}\label{markovsemi}
 	P^x_t[\phi](y)=\ex[\phi(Y^{x,y}(t))]\;,\;t\geq 0, \phi\in Lip(\h).\end{equation}
 	The latter is not rigorous in the present setting. Indeed, since $A_2$ is a differential operator, the paths of $Y^{x,y}$ do not take values in $Dom(A_2)$ and It\^o's formula cannot be directly applied to smooth functionals of $Y^{x,y}$.
 \end{rem}
 As we have already mentioned in the introduction, our assumptions guarantee that for each $x\in\h$, the process $Y^{x}$ admits a unique, strongly mixing local invariant measure $\mu^x$ defined on $(\h,\mathscr{B}(\h))$ (see e.g. Chapters 8, 11 of \cite{da1996ergodicity} as well as \cite{cerrai2009averaging}). We state here an important result regarding the continuity properties of the averaged Nemytskii operator $\bar{F}$.
 \begin{lem}\label{Fbarlip} Assume that $F:\h \times\h\rightarrow \h$ is Lipschitz continuous. Then the map
 	$$\h\ni x\longmapsto \bar{F}(x)=\int_\h F(x,y)d\mu^x(y)\in\h$$
 	is Lipschitz continuous. In particular, under Hypothesis \ref{A2a}, the operator $\bar{F}$ in \eqref{etadec1} is Lipschitz.
 \end{lem}
 \noindent The proof relies on the ergodicity of the invariant measure $\mu^x$ and can be found e.g. in Lemma 3.1 of \cite{cerrai2009khasminskii}.

 \noindent Now, as shown in \cite{cerrai2009averaging}, \eqref{Kolmeq} has a strict solution which is explicitly given by the probabilistic representation
 \begin{equation}\label{Feynman}
 \Phi^\epsilon_\chi(x,y)=\int_{0}^{\infty}e^{-c(\epsilon)t}P_t^x[\langle F(x,\cdot)-\bar{F}(x),\chi \rangle](y)dt\;, x\in\h, y\in Dom(A_2),
 \end{equation}
 with $\ell:=(\lambda-L_g)/2$ (see \eqref{A_2dis}, \eqref{diss2}) and for some  some $C>0$ independent of $\epsilon$, the following estimates hold:
 \begin{equation}\label{phiestimates}
 \begin{aligned}
 &|\Phi^\epsilon_\chi(x,y)|\leq\frac{C}{\ell}\big(1+\|x\|_\h+\|y\|_\h\big)\|\chi\|_\h\;,\\&
 \|D_y\Phi^\epsilon_\chi(x,y)\|_\h\leq\frac{C}{\ell}\|\chi\|_\h\;,\\&
 \|D_x\Phi^\epsilon_\chi(x,y)\|_\h\leq\frac{ C}{c(\epsilon)}\|\chi\|_\h\;,\\&
 \big|\text{tr}\big[D^2_2\Phi^\epsilon_\chi(x,y)\big]\big|\leq\frac{C}{c(\epsilon)}\big(1+\|x\|_\h+\|y\|_\h\big)\|\chi\|_\h
 \end{aligned}
 \end{equation}
 (see 5.12-5.15 in \cite{cerrai2009averaging}). In light of \eqref{Feynman} and these estimates, we see that the maps
 \begin{equation*}
 \begin{aligned}
 &\h\ni \chi_1\longmapsto
 \Phi_{\chi_1}^\epsilon(x,y)\in\R,\\&
 \h\times\h\ni (\chi_1,\chi_2)\longmapsto
 \blangle D_x\Phi_{\chi_1}^\epsilon\big(x,y),\chi_2\brangle_\h \in\R,\\&
 \h\times\h\ni (\chi_1,\chi_2)\longmapsto
 \blangle D_y\Phi_{\chi_1}^\epsilon\big(x,y),\chi_2\brangle_\h \in\R
 \end{aligned}
 \end{equation*}
 are in $\mathscr{L}(\h;\R),\mathscr{L}^2(\h;\R)$ and $\mathscr{L}^2(\h;\R)$ respectively. From the Riesz representation theorem, there exist $\Psi^\epsilon:\h\times\h\to\h$ and $\Psi_1^\epsilon, \Psi^\epsilon_2:\h\times\h\to\mathscr{L}(\h)$ such that for all $\chi_1,\chi_2,x\in\h, \epsilon>0$ and $y\in Dom(A_2)$
 \begin{equation}\label{Riesz}
 \begin{aligned}
 &\Phi_{\chi}^\epsilon(x,y) =  \blangle\Psi^\epsilon(x,y), \chi\brangle_\h\;,\\&
 \blangle	D_x\Phi_{\chi_1}^\epsilon\big(x,y),\chi_2\brangle_\h=  \blangle\Psi_1^\epsilon(x,y)\chi_2,\chi_1\brangle_\h\;,\\& \blangle D_y\Phi_{\chi_1}^\epsilon\big(x,y),\chi_2\brangle_\h=  \blangle\Psi_2^\epsilon(x,y)\chi_2,\chi_1\brangle_\h\;.
 \end{aligned}
 \end{equation}
 \noindent As a consequence of \eqref{phiestimates} we have

 \begin{equation}\label{psinorm}
 \begin{aligned}
 &\big\|\Psi^\epsilon(x,y)\big\|_\h\leq\frac{C}{\ell}\big(1+\|x\|_\h+\|y\|_\h\big),\\&
 \big\|\Psi_1^\epsilon(x,y)\big\|_{\mathscr{L}(\h)}\leq \frac{C}{c(\epsilon)} \;,\\&
 \big\|\Psi_2^\epsilon(x,y)\big\|_{\mathscr{L}(\h)}\leq \frac{C}{\ell}\;.
 \end{aligned}
 \end{equation}
 Additionally, as shown in Lemma \ref{epsilondeplem} below, there exists a map $\Psi_2^{0}:\h\times\h\rightarrow\mathscr{L}(\h)$ such that  \begin{equation}\label{Psi0} \sup_{x,y\in\h}\big\| \Psi_2^\epsilon(x,y)- \Psi^0_2\big(x,y\big)\big\|_{\mathscr{L}(\h)}\longrightarrow 0\;,\;\text{as}\;\epsilon\to 0.\end{equation}
 Next, let $Y^{\epsilon,u}_n$ denote a projection of the $Y^{\epsilon,u} $ to an $n$-dimensional eigenspace of $A_2$. For each $n$, the paths of  $Y^{\epsilon,u}_n$ take values in $Dom(A_2)$. This allows us to apply It\^o's formula to the real-valued process
 \begin{equation*}\label{XiIto}
 \big\{\blangle\Psi^\epsilon(\bar{X}(s),Y^{\epsilon,u}_n(s)), S_1(t-s)\chi\brangle_\h\big\}_{s\in[0,t]}\;,t\in[0,T]
 \end{equation*}
 to show that the asymptotic behavior of the last term in \eqref{etadec1}, as $\epsilon\to0$, is captured by
 \begin{equation*}
 \frac{\sqrt{\delta}}{\sqrt{\epsilon}}\int_{0}^{t}S_1(t-s)\Psi^0_2\big( \bar{X}(s), Y^{\epsilon,u}(s)\big)u_2(s)ds
 \end{equation*}
 (see  Lemma \ref{IVitodeclem}, Proposition \ref{controlim} and \eqref{IVlim}  below).

 We need to understand not just the limit of the process $\eta^{\epsilon,u}$ but also the measure with respect to which the averaging is being done. As in \cite{spiliopoulos2013large}, \cite{morse2017moderate}, \cite{WSS}, the dependence of the dynamics on the unknown control process $u = u^\epsilon$ complicates the situation. Following the recipe of these works we introduce the family of random occupation measures
 \begin{equation}\label{occupation1}
 P^{\epsilon,\Delta}(B_1\times B_2\times B_3\times B_4 )=\frac{1}{\Delta}\int_{B_4}\int_{t}^{t+\Delta}\mathds{1}_{B_1}\big(u_1(s)\big)\mathds{1}_{B_2}\big(u_2(s)\big) \mathds{1}_{B_3}\big(Y^{\epsilon,u}(s)\big)dsdt,
 \end{equation}
defined on
 $\mathscr{B}\big(\h\times\h\times\h\times[0,T]\big)$. Here, the first two copies of $\h$ are endowed with the weak topology, the third with the norm topology and $[0,T]$ with the standard topology. For the sake of shortness we will call the resulting product topology WWNS.  The parameter $\Delta=\Delta(\epsilon)$ is such that
 \begin{equation}\label{Delta1}
 \Delta(\epsilon)\longrightarrow 0\;\;,\;\frac{\sqrt{\delta}h(\epsilon)}{\sqrt\Delta}\longrightarrow 0\;\;,\;\text{as}\;\epsilon\to 0.
 \end{equation}

 These occupation measures encode the behavior of the control and the fast process. It is the correct way to study the problem because the fast motion's behavior will not converge pathwise to anything, but its occupation measure will converge to a limiting measure. We adopt the convention that the control $u(t) = u^\epsilon(t) = 0$ for $t > T$. Then, we consider the joint limit in distribution of the pair $(\eta^{\epsilon,u}, P^{\epsilon,\Delta} )$ as $\epsilon\to 0$.

 In order to state our main results, we introduce the following definition of a viable pair corresponding to \cite{dupuis2012large}, but appropriately modified for the moderate deviation setting.

 \begin{dfn}\label{viable} Let $T<\infty$, $\Xi:\h^5\rightarrow\h$ and $\bar{X}\in C\big([0,T];\h\big)$ solve \eqref{x-aved}. For each $x\in\h$, let $\mu^x$ denote the unique invariant measure of \eqref{Yfrozen}.  A pair $(\psi, P)\in C\big([0,T];\h\big)\times \mathscr{P}(\h\times\h\times\h\times[0,T])$, where  $\h\times\h\times\h\times[0,T]$ is endowed with the WWNS topology, will be called viable with respect to $(\Xi, \mu^{\bar{X}})$ if\\
 	\noindent (i) The measure $P$ has finite second moments in the sense that there exists $\theta>0$ such that \begin{equation}\label{viable1}
 	\int_{\h\times\h\times\h\times[0,T]}\big(\|u_1\|^2_\h+\|u_2\|^2_\h+\|y\|^2_{H^\theta}\big) dP(u_1,u_2,y,t)<\infty.
 	\end{equation}
 	\noindent (ii) For all $B_1\times B_2\times B_3\times B_4\in\mathscr{B}(\h\times\h\times\h\times[0,T])$,
 	\begin{equation}\label{viable2}
 	P(B_1\times B_2\times B_3\times B_4)=\int_{B_4}\int_{B_3}\nu(B_1\times B_2|y,t)d\mu^{\bar{X}(t)}(y)dt,
 	\end{equation}
 	where $\nu:\mathscr{B}(\h\times\h)\times\h\times[0,T]\rightarrow[0,1]$ is a stochastic kernel on $\h$ given $\h\times[0,T]$ (see Appendix A.5 in \cite{dupuis2011weak} for stochastic kernels). This implies that the last marginal of $P$ is Lebesgue measure on $[0,T]$ and in particular
 	\begin{equation}\label{viableleb}
 	P(\h\times\h\times\h\times[0,t])=t\;,\;\text{for all}\; t\in[0,T].
 	\end{equation}
 	\noindent (iii) For all $t\in[0,T]$,
 	\begin{equation}\label{viable3}
 	\psi(t)=\int_{\h\times\h\times\h\times[0,t]}S_1(t-s)\Xi\big( \psi(s),\bar{X}(s), y, u_1,u_2\big) dP(u_1,u_2,y,s).
 	\end{equation}
 	The family of viable pairs with respect to $(\Xi, \mu^{\bar{X}})$ will be denoted by $\mathcal{V}_{(\Xi, \mu^{\bar{X}})}$.
 \end{dfn}
 \noindent  In view of \eqref{Regimes}, we also define
 \begin{equation}\label{gammai}
 \gamma_i=\begin{cases} & 0,\; i=1\\&
 \gamma\in(0,\infty),\; i=2.
 \end{cases}
 \end{equation}
 Using the viable pair definition, we can then state the main results of our paper.

 \begin{thm}\label{viablim1} (Averaging for $\eta^{\epsilon,u}$) Let $i=1,2,$ $T<\infty$, $a>0$ and $u\in\mathcal{P}_N^T$. Moreover let $(X^{\epsilon,u}, Y^{\epsilon,u})$ be the mild solution of \eqref{controlledsystem} with initial conditions $x_0,y_0\in H^a(0,L)$ and $\eta^{\epsilon,u}$ as in \eqref{etadec1}. Let $\Xi_i:\h^5\rightarrow\h$ be defined by
 	\begin{equation}
 	\label{xidef1}
 	\Xi_i(\psi,x, y, u_1,u_2):= D_xF(x,y)\psi+ \Sigma(x,y)u_1+\gamma_i\Psi^0_2(x, y)u_2\;,\;i=1,2\;,
 	\end{equation}
 	with $\gamma_i$ and $\Psi_2^0$ as in \eqref{gammai} and \eqref{Psi0} respectively. Assuming Hypotheses \ref{A1a}-\ref{A1c}, \ref{A2a}-\ref{A2c},  \ref{A3a}, \ref{A3b} and Regime $i$, the family of processes $\{\eta^{\epsilon,u}:\epsilon\in(0, 1), u\in\mathcal{P}_N^T\}$ is tight in $C([0, T ]; \h)$ and the family of occupation measures $\{P^{\epsilon,\Delta} :\epsilon\in(0, 1), u\in\mathcal{P}_N^T\}$ is tight in $\mathscr{P}(\h\times\h\times\h\times[0,T])$, where $\h\times\h\times\h\times[0,T]$ is endowed with the WWNS topology.
 	
 	\noindent Then for any sequence in $\{(\eta^{\epsilon,u}, P^{\epsilon,\Delta} )\;,\epsilon, \Delta>0,u\in\mathcal{P}_N^T\}$ there exists a subsequence that converges in distribution with limit $(\eta_i,P_i)$. With probability $1$,  $$(\eta_i,P_i)\in\mathcal{V}_{(\Xi_i, \mu^{\bar{X}})}.$$
 	 \end{thm}

\begin{thm}\label{MDP1}(Moderate Deviation Principle) Let $i=1,2$, $T<\infty$, $a>0$ arbitrarily small and $(X^{\epsilon,x_0,y_0}, Y^{\epsilon,x_0,y_0}), \bar{X}^x$ be the mild solutions to \eqref{model} and \eqref{x-aved} with initial conditions $x_0,y_0\in H^{a}$ .
 Define $\mathcal{S}_i:C([0, T ]; \h)\rightarrow [0,\infty]$, \begin{equation*}
\mathcal{S}_i(\phi):=\inf_{(\phi,P)\in\mathcal{V}_{(\Xi_i, \mu^{\bar{X}})} }\bigg[\frac{1}{2}\int_{\h\times\h\times\h\times[0,T]}\big( \|u_1 \|^2_\h+ \|u_2 \|^2_\h\big)\;dP(u_1,u_2,y,t)\bigg] \;\;,\phi\in C\big([0,T];\h\big)
\end{equation*}
 with the convention that $\inf\varnothing=\infty$.
Assuming Hypotheses \ref{A1a}-\ref{A1c}, \ref{A2a}-\ref{A2c},  \ref{A3a'},\ref{A3b} and Regime $i$ we have that for every bounded and continuous function $\Lambda : C([0, T ]; \h)\rightarrow \R$:
\begin{equation*}
\label{Laplace1}
\begin{aligned}
&\lim_{\epsilon\to 0}\frac{1}{h^2(\epsilon)}\log\ex\big[ e^{-h^2(\epsilon)\Lambda(\eta^\epsilon)}\big]=-\inf_{\phi\in C([0,T];
	\h)}\big[ \mathcal{S}_i(\phi)+\Lambda(\phi) \big],
\end{aligned}
\end{equation*}
where $$\eta^{\epsilon}=\frac{X^{\epsilon,x_0,y_0}-\bar{X}^{x_0}}{\sqrt{\epsilon}h(\epsilon)}\;.$$
\noindent In particular, $\{X^{\epsilon}\}$ satisfies a Moderate Deviation Principle in $C([0,T];\h)$ in Regime $i$  with rate function $\mathcal{S}_i$.
\end{thm}
 \noindent The proof of Theorem \ref{viablim1} can be found in Section \ref{aveproof} while Theorem \ref{MDP1} is proved in Section \ref{MDPsec}.
In fact, by letting $Q_i:\h\rightarrow \mathscr{L}(\h)$,
\begin{equation}
Q_i(x)=\int_{\h}\bigg(   \Sigma(x,y)\Sigma^*(x,y)+\gamma^2_i\Psi^0_2(x,y)\Psi^{0*}_2(x,y) \bigg)d\mu^x(y)
\end{equation}
with $\gamma_i$ and $\Psi_2^0$ as in \eqref{gammai} and \eqref{Psi0} respectively, we prove that our rate function $\mathcal{S}_i$ has an explicit non-variational form given by
\begin{equation}\label{nonvar}
\mathcal{S}_i(\psi)=\frac{1}{2}\int_{0}^{T}\bigg\|Q_i\big(\bar{X}(t)\big)^{-\frac{1}{2}}\big[ \partial_t\psi(t)- A_1\psi(t)-\overline{D_xF}\big(\bar{X}(t)\big)\psi(t)\big]\bigg\|^2_\h dt
\end{equation}
for $\psi\in  H_0^1([0,T];\h)\cap L^2([0,T];Dom(A_1))$ and $\mathcal{S}_i=\infty$ otherwise (see Proposition \ref{optimalprop}).

\section{A priori bounds for the solution of the controlled system}\label{Sec2}

\noindent As discussed in Section \ref{weakconv}, the variational representation \eqref{varrep} gives rise to a slow-fast pair of controlled stochastic reaction-diffusion equations. In this section we prove a priori estimates for the mild solution pair $(X^{\epsilon,u}, Y^{\epsilon,u})$  (see \eqref{controlledsystemild}) that are uniform over compact time intervals, $u\in\mathcal{P}_N^T$ and $\epsilon$ sufficiently small. These preliminary estimates hold in both Regimes $1$ and $2$ and we will use them to prove a priori bounds and tightness for the family $\{\eta^{\epsilon,u};\epsilon,u\}$ in Sections \ref{Sec3} and \ref{Sec4}.

We start with two auxiliary estimates for the moments of the space-time $L^2$ norm and  the $C([0,T];\h)$ norm of the controlled fast process $Y^{\epsilon,u}$. Due to the multiple scales, the latter is singular at $\delta=0$. The proofs rely on the dissipativity assumption \eqref{extradiss}. As is customary, we use the same notation for different but unimportant constants that may change from line to line.
\begin{lem}\label{Yprebnd} Let $T<\infty$,  $p\geq1$, $\epsilon\in(0,1)$ and $u\in\mathcal{P}^T_N$. In both Regimes 1 and 2, there exists a constant $C>0$, independent of $\epsilon$, such that
	\begin{equation}\label{Yint}
	\begin{aligned}
	\ex\|Y^{\epsilon,u}\|_{L^2([0,T];\h)}^{2p}& \leq C \bigg(1+ \|y_0\|^{2p}_\h+
	\int_{0}^{T}   \ex\|  X^{\epsilon,u}(t)\|^{2p}_\h dt\;\bigg).
	\end{aligned}
	\end{equation}
	Moreover, for any $\rho\in(1/2,1)$ and $\epsilon$ sufficiently small we have
	\begin{equation}\label{Ypresup}
	\begin{aligned}
	\ex\sup_{t\in[0,T]}\|Y^{\epsilon,u}(t)\|^2_{\h}& \leq C\bigg(   1  + \|y_0\|^2_\h+\ex\sup_{t\in[0,T]} \|  X^{\epsilon,u}(t)\|^2_\h+h^2(\epsilon)+\delta^{\rho-1}\bigg).
	\end{aligned}
	\end{equation}
\end{lem}
\begin{proof}
	
	Let $Y^{\epsilon,u}$ be the mild solution of the controlled fast equation (see \eqref{controlledsystemild}), $$w^\delta_{A_2}(t)=\frac{1}{\sqrt{\delta}}\int_{0}^{t} S_2\bigg(\frac{t-z}{\delta}\bigg) dw_2(z)$$
	be the stochastic convolution term and
	\begin{equation}\label{Gammadef}
	\Gamma^{\epsilon,u}(t):=Y^{\epsilon,u}(t)-w^\delta_{A_2}(t)\;,\;\;t\in[0,T].
	\end{equation}
  With probability $1$, the process $\Gamma^{\epsilon,u}$ has weakly differentiable paths  and satisfies
	\begin{equation*}
	\partial_t\Gamma^{\epsilon,u}(t)= \frac{1}{\delta}\big[A_2 \Gamma^{\epsilon,u}(t)+ G\big(X^{\epsilon,u}(t),\Gamma^{\epsilon,u}(t)+ w^\delta_{A_2}(t)\big)\big]+\frac{h(\epsilon)}{\sqrt{\delta}}u_2(t)
	\end{equation*}
	\noindent in a weak sense.	Hence,
	\begin{equation}\label{enesty}
	\begin{aligned}
	\frac{1}{2}\partial_t\|\Gamma^{\epsilon,u}(t)\|^2_{\h}&= \blangle \partial_t\Gamma^{\epsilon,u}(t),\Gamma^{\epsilon,u}(t) \brangle_\h
	= \frac{1}{\delta}\blangle A_2 \Gamma^{\epsilon,u}(t),\Gamma^{\epsilon,u}(t) \brangle_\h\\&+  \frac{1}{\delta}\blangle  G\big(X^{\epsilon,u}(t),\Gamma^{\epsilon,u}(t)+ w^\delta_{A_2}(t)\big), \Gamma^{\epsilon,u}(t)\brangle_\h
	+ \frac{h(\epsilon)}{\sqrt{\delta}}\blangle u_2(t), \Gamma^{\epsilon,u}(t)\brangle_\h.
	\end{aligned}
	\end{equation}
	\noindent For the first term above we invoke Hypothesis \ref{A1c}  to obtain
	\begin{equation}\label{Gamma1}
	\langle A_2 \Gamma^{\epsilon,u}(t),\Gamma^{\epsilon,u}(t) \rangle_\h=\sum_{n=1}^{\infty}(-a_{2,n}) \langle \Gamma^{\epsilon,u}(t),e_{2,n}\rangle^2_\h\leq -\lambda\|\Gamma^{\epsilon,u}(t)\|^2_\h\;.
	\end{equation}
	\noindent For the second term in \eqref{enesty} we invoke Hypothesis \ref{A2b} which implies that $G:\h\times\h\rightarrow \h$ is $L_g$-Lipschitz and with $C_g=(\|G(0_\h,0_\h)\|_\h\vee L_g)$ we have
	\begin{equation}\label{Gamma2}
	\begin{aligned}
	\big|\blangle G\big(X^{\epsilon,u}(t)&,\Gamma^{\epsilon,u}(t)+w^\delta_{A_2}(t)\big),  \Gamma^{\epsilon,u}(t)   \brangle_\h\big|\leq  \big|\blangle G\big(X^{\epsilon,u}(t),w^\delta_{A_2}(t)\big), \Gamma^{\epsilon,u}(t)   \brangle_\h\big|\\&
	+  \big|\blangle G\big(X^{\epsilon,u}(t),\Gamma^{\epsilon,u}(t)+w^\delta_{A_2}(t)\big)-G\big(X^{\epsilon,u}(t),w^\delta_{A_2}(t)\big),  \Gamma^{\epsilon,u}(t)   \brangle_\h\big|
\\&
	\leq C_g\|\Gamma^{\epsilon,u}(t)\|_\h\bigg(   1+\| w^\delta_{A_2}(t)\|_\h  +  \|  X^{\epsilon,u}(t)\|_\h      \bigg)+ L_g\|\Gamma^{\epsilon,u}(t)\|^2_\h.
	\end{aligned}
	\end{equation}
 Combining \eqref{enesty}, \eqref{Gamma1} and \eqref{Gamma2} we obtain
	\begin{equation*}
	\begin{aligned}
	\frac{1}{2}\partial_t\|\Gamma^{\epsilon,u}(t)\|^2_{\h}&\leq  \frac{C_g}{\delta}\|\Gamma^{\epsilon,u}(t)\|_\h\bigg(   1+\| w^\delta_{A_2}(t)\|_\h  +  \|  X^{\epsilon,u}(t)\|_\h      \bigg)\\&+ \frac{L_g-\lambda}{\delta}\|\Gamma^{\epsilon,u}(t)\|^2_\h+\frac{h(\epsilon)}{\sqrt{\delta}}\|\Gamma^{\epsilon,u}(t)\|_\h\|u_2(t)\|_\h\;.
	\end{aligned}
	\end{equation*}
	\noindent Next, let $\beta_1,\beta_2>0$. From an application of Young's inequality for products on the first and third terms,
	\begin{equation*}
	\begin{aligned}
	\frac{1}{2}\partial_t\|\Gamma^{\epsilon,u}(t)\|^2_{\h}&\leq  \frac{C_g\beta_1^2}{4\delta}\|\Gamma^{\epsilon,u}(t)\|^2_{\h} +\frac{2C_g}{2\delta
		\beta_1^2}\bigg(   1+\| w^\delta_{A_2}(t)\|^2_\h  +  \|  X^{\epsilon,u}(t)\|^2_\h      \bigg)\\&+ \frac{L_g-\lambda}{\delta}\|\Gamma^{\epsilon,u}(t)\|^2_\h+\frac{h(\epsilon)}{4\sqrt{\delta}}\beta^2_2\|\Gamma^{\epsilon,u}(t)\|^2_\h+\frac{2h(\epsilon)}{2\sqrt{\delta}\beta^2_2}\|u_2(t)\|^2_\h\;.
	\end{aligned}
	\end{equation*}
	From Hypothesis \ref{A2b} we have $\lambda-L_g>0$ and thus we can choose $\beta_1^2=(\lambda-L_g)/C_g$ and $\beta_2^2=(\lambda-L_g)/(h(\epsilon)\sqrt{\delta})$ to obtain
	\begin{equation}\label{Gammabnd}
	\begin{aligned}
	\frac{1}{2}\partial_t\|\Gamma^{\epsilon,u}(t)\|^2_{\h}&\leq  \frac{\lambda-L_g}{4\delta} \|\Gamma^{\epsilon,u}(t)\|^2_\h+\frac{C^2_g}{(\lambda-L_g)\delta}\bigg(   1+\| w^\delta_{A_2}(t)\|^2_\h  +  \|  X^{\epsilon,u}(t)\|^2_\h      \bigg)\\&+ \frac{L_g-\lambda}{\delta}\|\Gamma^{\epsilon,u}(t)\|^2_\h+\frac{\lambda-L_g}{4\delta}\|\Gamma^{\epsilon,u}(t)\|^2_\h+\frac{h^2(\epsilon)}{\lambda-L_g}\|u_2(t)\|^2_\h\\&
	=-\frac{1}{\delta}\bigg(\frac{\lambda-L_g}{2}\bigg) \|\Gamma^{\epsilon,u}(t)\|^2_\h+\frac{C^2_g}{(\lambda-L_g)\delta}\bigg(   1+\| w^\delta_{A_2}(t)\|^2_\h  +  \|  X^{\epsilon,u}(t)\|^2_\h      \bigg)\\&+\frac{h^2(\epsilon)}{\lambda-L_g}\|u_2(t)\|^2_\h\;.
	\end{aligned}
	\end{equation}
	\noindent Integrating this inequality yields
	\begin{equation}\label{Ymain}
	\begin{aligned}
	\frac{1}{2}\|\Gamma^{\epsilon,u}(t)\|^2_{\h}-\frac{1}{2}\|y_0\|^2_\h&\leq 	-\frac{1}{\delta}\bigg(\frac{\lambda-L_g}{2}\bigg)\int_{0}^{t} \|\Gamma^{\epsilon,u}(s)\|^2_\h ds\\&+\frac{C^2_g}{(\lambda-L_g)\delta}\int_{0}^{t}\bigg(   1+\| w^\delta_{A_2}(s)\|^2_\h  +  \|  X^{\epsilon,u}(s)\|^2_\h      \bigg)ds+\frac{h^2(\epsilon)N^2}{\lambda-L_g}\;,
	\end{aligned}
	\end{equation}
	\noindent where the last term follows from the fact that $u_2\in\mathcal{P}_N^T$. Letting $\ell=(\lambda-L_g)/2$, multiplying throughout by $\delta/\ell$ and dropping the nonnegative term $(\delta/2\ell)\sup_{t\in[0,T]}\|\Gamma^{\epsilon,u}(t)\|^2_{\h}$ we see that
	\begin{equation*}
	\begin{aligned}
	\int_{0}^{T} \|\Gamma^{\epsilon,u}(s)\|^2_\h ds	&\leq \frac{\delta}{2\ell}\sup_{t\in[0,T]}\|\Gamma^{\epsilon,u}(t)\|^2_{\h}+\int_{0}^{T} \|\Gamma^{\epsilon,u}(s)\|^2_\h ds\\&\leq \frac{\delta}{2\ell}\|y_0\|^2_\h+
	\frac{C^2_g}{(\lambda-L_g)\ell}\int_{0}^{T}\bigg(   1+\| w^\delta_{A_2}(s)\|^2_\h  +  \|  X^{\epsilon,u}(s)\|^2_\h      \bigg)ds+\frac{N^2\delta h^2(\epsilon)}{(\lambda-L_g)\ell}\;.
	\end{aligned}
	\end{equation*}
	\noindent Regarding the last term on the right-hand side, note that, in both Regimes $1$ and $2$ (see \eqref{Regimes}, \eqref{h}),
	$$\delta h^2(\epsilon)=\bigg(\frac{\delta}{\epsilon}\bigg)\epsilon h^2(\epsilon)\longrightarrow 0\;,$$
	as $\epsilon\to 0$. Hence, for all  sufficiently small $\epsilon$,
	\begin{equation*}
	\begin{aligned}
	\int_{0}^{T} \|\Gamma^{\epsilon,u}(s)\|^2_\h ds	\leq 1+\|y_0\|^2_\h+
	C\int_{0}^{T}\bigg(   1+\| w^\delta_{A_2}(s)\|^2_\h  +  \|  X^{\epsilon,u}(s)\|^2_\h      \bigg)ds,
	\end{aligned}
	\end{equation*}
	\noindent and	in view of \eqref{Gammadef} we have
	\begin{equation*}
	\begin{aligned}
	\int_{0}^{T} \|Y^{\epsilon,u}(s)\|^2_\h ds&\leq	C_1\int_{0}^{T} \|\Gamma^{\epsilon,u}(s)\|^2_\h ds+ 	C_2\int_{0}^{T} \|w^{\delta}_{A_2}(s)\|^2_\h ds\\& \leq C_1(1+ \|y_0\|^2_\h)+
	C_2\int_{0}^{T}\bigg(   1+\| w^\delta_{A_2}(s)\|^2_\h  +  \|  X^{\epsilon,u}(s)\|^2_\h      \bigg)ds.
	\end{aligned}
	\end{equation*}
	\noindent After taking expectation we deduce that
	\begin{equation*}
	\begin{aligned}
	\ex\bigg(\int_{0}^{T} \|Y^{\epsilon,u}(s)\|^2_\h ds\bigg)^p& \leq C_p(1+ \|y_0\|^{2p}_\h)+C'_p
	\int_{0}^{T}\bigg(   1+\ex\| w^\delta_{A_2}(s)\|^{2p}_\h  +  \ex\|  X^{\epsilon,u}(s)\|^{2p}_\h      \bigg)ds
	\end{aligned}
	\end{equation*}
	and \eqref{Yint} follows upon invoking Lemma \ref{stoconvb}(i).

		It remains to prove \eqref{Ypresup}. Returning to \eqref{Gammabnd}, we multiply throughout by $e^{2\ell t/\delta}$ to obtain
	\begin{equation}
	\begin{aligned}
	\partial_t\big(e^{2\ell t/\delta}\|\Gamma^{\epsilon,u}(t)\|^2_{\h}\big)&=e^{2\ell t/\delta}\partial_t\|\Gamma^{\epsilon,u}(t)\|^2_{\h}
	+\frac{2\lambda}{\delta}e^{2\ell t/\delta} \|\Gamma^{\epsilon,u}(t)\|^2_\h\\&\leq \frac{2C^2_g}{(\lambda-L_g)\delta}e^{2\ell t/\delta}\bigg(   1+\| w^\delta_{A_2}(t)\|^2_\h  +  \|  X^{\epsilon,u}(t)\|^2_\h      \bigg)\\&+\frac{2h^2(\epsilon)}{\lambda-L_g}e^{2\ell t/\delta}\|u_2(t)\|^2_\h\;.
	\end{aligned}
	\end{equation}
	\noindent Integrating the latter on $[0,t]$ then yields
	\begin{equation*}
	\begin{aligned}
	\|\Gamma^{\epsilon,u}(t)\|^2_{\h}&\leq \|y_0\|^2_\h+ \frac{2C^2_g}{(\lambda-L_g)\delta}\int_{0}^{t}e^{-2\ell (t-s)/\delta}\bigg(   1+\| w^\delta_{A_2}(s)\|^2_\h  +  \|  X^{\epsilon,u}(s)\|^2_\h      \bigg)ds \\&+\frac{2h^2(\epsilon)}{\lambda-L_g}\int_{0}^{t}e^{-2\ell (t-s)/\delta}\|u_2(s)\|^2_\h ds\leq  \|y_0\|^2_\h+C\bigg(   1+\sup_{s\in[0,t]}\| w^\delta_{A_2}(s)\|^2_\h  +  \sup_{s\in[0,t]}\|  X^{\epsilon,u}(s)\|^2_\h      \bigg)\\&
	+Ch^2(\epsilon)\int_{0}^{t}\|u_2(s)\|^2_\h ds.
	\end{aligned}
	\end{equation*}
	\noindent Taking expectation and applying Lemma \ref{stoconvb}(i) we deduce that					
	\begin{equation*}
	\begin{aligned}
	\ex\sup_{t\in[0,T]}\|\Gamma^{\epsilon,u}(t)\|^2_{\h}&\leq 	\|y_0\|^2_\h +C\bigg(  1+\delta^{\rho-1}+ \ex\sup_{s\in[0,T]}\|  X^{\epsilon,u}(s)\|^2_\h\bigg)+C_Nh^2(\epsilon)
	\end{aligned}
	\end{equation*}
	Hence, we can use Lemma \ref{stoconvb} (ii) to show that
\begin{equation*}
\begin{aligned}
\ex\sup_{t\in[0,T]}\|Y^{\epsilon,u}(t)\|^2_{\h}&\leq C \ex\sup_{t\in[0,T]}\|\Gamma^{\epsilon,u}(t)\|^2_{\h}+ C'\ex\sup_{t\in[0,T]}\|w^\delta_{A_2}(t)\|^2_{\h}\\& \leq
 C\bigg(   1  + \|y_0\|^2_\h+\ex\sup_{t\in[0,T]} \|  X^{\epsilon,u}(t)\|^2_\h+h^2(\epsilon)+\delta^{\rho-1}\bigg)
\end{aligned}
\end{equation*}
and the proof is complete.\end{proof}

\begin{rem} Due to the presence of the stochastic controls $u$, we can only prove uniform estimates for the fast process $Y^{\epsilon,u}$ in $L^p([0,T];\h)$ for $p\leq 2$. This limitation is also reflected in the choice of the growth exponent $\nu<1/2$ in Hypothesis \ref{A3a}.
\end{rem}
\noindent Using Lemma \ref{Yprebnd}, we can prove the following a priori bounds for $(X^{\epsilon,u}, Y^{\epsilon,u})$ by means of the Gr\"onwall inequality.

\begin{prop} Let $T<\infty$ and $\nu\in(0,1/2)$ be as in Hypothesis \ref{A3a}. In both Regimes $1$ and $2$, there exists $\epsilon_0>0$ and a constant $C>0$, independent of $\epsilon$, such that

	\begin{equation} \label{xapriori}
	\begin{aligned}
	&\sup_{0<\epsilon<\epsilon_0, u\in\mathcal{P}^T_N}\ex\sup_{t\in[0,T]}\|X^{\epsilon,u}(t)\|^{\frac{2}{\nu}}_\h
	\leq C\bigg(1+\|x_0\|^{\frac{2}{\nu}}_\h+\|y_0\|^{\frac{2}{\nu}}_\h\bigg)
	\end{aligned}
	\end{equation}
	\noindent and
	\begin{equation}\label{yapriori}
	\begin{aligned}
	\sup_{0<\epsilon<\epsilon_0, u\in\mathcal{P}^T_N}\ex\|Y^{\epsilon,u}\|_{L^2([0,T];\h)}^{\frac{2}{\nu}} &
	\leq C\bigg( 1+\|x_0\|^{\frac{2}{\nu}}_\h+\|y_0\|^{\frac{2}{\nu}}_\h\bigg).
	\end{aligned}
	\end{equation}
	\noindent Moreover, for any $\rho\in(1/2,1)$ and $\epsilon$ sufficiently small, there exists a positive constant  $C$, independent of $\epsilon$, such that
	\begin{equation} \label{yunifapriori}
	\begin{aligned}
	\sup_{u\in\mathcal{P}_N^T}\ex  \sup_{t\in[0, T] }\|Y^{\epsilon,u}(t)\|^2_\h \leq C \bigg(1+\|x_0\|_\h^2+\|y_0\|_\h^2+h^2(\epsilon)+\delta^{\rho-1}\bigg).
	\end{aligned}
	\end{equation}
\end{prop}
\noindent 	Estimates \eqref{xapriori} and \eqref{yapriori} are standard and their proofs will be omitted. Similar results can be found e.g. in \cite{cerrai2009khasminskii,WSS} among other places. The main difference here is in the moderate deviation scaling which does not change the proof in an essential way. Finally, \eqref{yunifapriori} follows from the combination of \eqref{Ypresup} and \eqref{xapriori}.

Next, we provide an estimate for the H\"older seminorm of the controlled fast process $Y^{\epsilon,u}$ which depends on the regularity of the initial conditions.
The estimate is singular at $\delta=0$. As seen in the proof below, there is a trade-off between the H\"older exponent and the rate of divergence of the right-hand side as $\epsilon\to 0$. 

	\begin{prop}\label{Schauderyprop} Let $T<\infty$, $a\in(0,2]$, 
			$x_0\in\h$ and $y_0\in H^{a}(0,L)$. For all $u\in\mathcal{P}^T_N$ and $\epsilon$ sufficiently small there exists  $\beta<\frac{1}{4}\land \frac{a} {2}$ and a constant $C>0$ independent of $\epsilon$ such that	
	\begin{equation}\label{Schaudery}
	\ex \big[Y^{\epsilon, u}\big]_{C^\beta([0,T];\h)}\leq Ch(\epsilon)\delta^{-\frac{1}{2}\vee \frac{a}{2}} \bigg(1+\|x_0\|_\h+\|y_0\|_{H^a}\bigg).
	\end{equation}
\end{prop}	

\begin{proof} Letting $0\leq s<t\leq T$ we can write
	\begin{equation*}\label{Schauderdec}
	\begin{aligned}
	Y^{\epsilon,u}(t)- Y^{\epsilon,u}(s)&= \bigg[S_2\bigg(\frac{t}{\delta}\bigg)-S_2\bigg(\frac{s}{\delta}\bigg)\bigg]y_0+\frac{1}{\delta}\int_{s}^{t} S_2\bigg(\frac{t-z}{\delta}\bigg)G\big( X^{\epsilon,u} (z), Y^{\epsilon,u} (z)   \big)dz \\&+ \frac{1}{\delta}
	\bigg[S_2\bigg(\frac{t-s}{\delta}\bigg)-I\bigg]\int_{0}^{s}S_2\bigg(\frac{s-z}{\delta}\bigg)G\big( X^{\epsilon,u} (z), Y^{\epsilon,u} (z)   \big)dz\\&+\frac{h(\epsilon)}{\sqrt\delta}\int_{s}^{t}S_2\bigg(\frac{t-z}{\delta}\bigg)u_2(z)dz \\&
	+\frac{h(\epsilon)}{\sqrt\delta}\bigg[S_2\bigg(\frac{t-s}{\delta}\bigg)-I\bigg]\int_{0}^{s}S_2\bigg(\frac{s-z}{\delta}\bigg)u_2(z)dz\\&+ w^\delta_{A_2}(t)-w^\delta_{A_2}(s)
	=:\sum_{k=1}^{6} J^{\epsilon, u}_{k}(s,t).
	\end{aligned}
	\end{equation*}
	\noindent We shall estimate each term of this decomposition separately. For $J^{\epsilon,u}_1$, we use the semigroup property and invoke \eqref{S2decay}, \eqref{sobcont} to obtain
	\begin{equation}\label{J_1}
	\begin{aligned}
	\big\|J^{\epsilon, u}_{1}(s,t)\big\|_\h& =\bigg\|S_2\bigg(\frac{s}{\delta}\bigg)\bigg[ S_2\bigg(\frac{t-s}{\delta}\bigg)-I   \bigg]y_0\bigg\|_\h\\&\leq \bigg\|S_2\bigg(\frac{s}{\delta}\bigg)\bigg\|_{\mathscr{L}(\h)}\bigg\|\bigg[ S_2\bigg(\frac{t-s}{\delta}\bigg)-I   \bigg]y_0\bigg\|_\h\\&
	\leq e^{-\lambda s/\delta}\bigg\|S_2\bigg(\frac{t-s}{\delta}\bigg)-I  \bigg\|_{\mathscr{L}(H^a;\h)}\|y_0\|_{H^a}
	\\&\leq C_T\delta^{-a/2}(t-s)^{a/2}\|y_0\|_{H^a}\;.
	\end{aligned}
	\end{equation}
	
		Next, we use the Lipschitz continuity of $G$ along with H\"older's inequality for $q\geq1$ to obtain
	\begin{equation*}
	\begin{aligned}
	\big\|J^{\epsilon, u}_{2}(s,t)\big\|_\h&\leq \frac{C_g}{\delta}\int_{s}^{t}e^{-\frac{\lambda(t-z)}{\delta}}\big(1+\big\|X^{\epsilon,u} (z)\big\|_\h+\big\|Y^{\epsilon,u} (z)\big\|_\h\big)dz \\&
	\leq  \bigg(1+\sup_{t\in[0,T]}\big\|X^{\epsilon,u} (t)\big\|_\h+\sup_{t\in[0,T]}\big\|Y^{\epsilon,u} (t)\big\|_\h\bigg)\frac{C_g}{\delta}\bigg(\int_{s}^{t}e^{-\frac{p\lambda(t-z)}{\delta}}dz\bigg)^{1/p} (t-s)^{1/q}\\&
	\leq C\delta^{-1/q}(t-s)^{1/q}\bigg(1+\sup_{t\in[0,T]}\big\|X^{\epsilon,u} (t)\big\|_\h+\sup_{t\in[0,T]}\big\|Y^{\epsilon,u} (t)\big\|_\h\bigg)\bigg(\int_{0}^{\infty}e^{-p\lambda\zeta}d\zeta\bigg)^{1/p}.
	\end{aligned}
	\end{equation*}
	Letting $\epsilon$ be sufficiently small, taking expectation and applying \eqref{xapriori} and \eqref{yunifapriori} we get
	\begin{equation*}\label{preJ_2}
	\ex\sup_{t,s\in[0,T], t\neq s }\frac{\big\|J^{\epsilon, u}_{2}(s,t)\big\|_\h}{|t-s|^{1/q}}\leq C_p\delta^{-\frac{1}{q}}\bigg(1+\|x_0\|_\h+\|y_0\|_\h+h(\epsilon)+\delta^{\frac{\rho-1}{2}}\bigg).	\end{equation*}
	Choosing $\rho=3/4\in(1/2,1)$ and $q=9$ yields
	$$ \frac{1}{q}+\frac{1-\rho}{2}=\frac{1}{9}+\frac{1}{8}<\frac{1}{4}\;.$$
	Hence, for $\beta\leq1/9$
	\begin{equation}\label{J_2}
	\ex\sup_{t,s\in[0,T], t\neq s }\frac{\big\|J^{\epsilon, u}_{2}(s,t)\big\|_\h}{|t-s|^{\beta}}\leq Ch(\epsilon)\delta^{-1/4}\bigg(1+\|x_0\|_\h+\|y_0\|_\h\bigg).	\end{equation}

		\noindent Next, for $J^{\epsilon,u}_3$, we shall invoke \eqref{sobcont} and then apply Lemma \ref{sigmacont}(i)  to obtain
	\begin{equation*}
	\begin{aligned}
	\big\|J^{\epsilon,u}_{3}(s,t)\big\|_\h&\leq  \frac{1}{\delta}
	\bigg\|S_2\bigg(\frac{t-s}{\delta}\bigg)-I\bigg\|_{\mathscr{L}(H^\theta;\h)}\int_{0}^{s}\bigg\|S_2\bigg(\frac{s-z}{\delta}\bigg)G\big( X^{\epsilon,u} (z), Y^{\epsilon,u} (z)   \big)\bigg\|_{H^\theta}dz \\&
	\leq \bigg(\frac{C}{\delta}\bigg)\delta^{-\theta/2}(t-s)^{\theta/2}\int_{0}^{s}\bigg\|(-A_2)^{\theta/2}S_2\bigg(\frac{s-z}{\delta}\bigg)G\big( X^{\epsilon,u} (z), Y^{\epsilon,u} (z)   \big)\bigg\|_{\h}dz \\&
	\leq C_g\delta^{-1-\theta/2}(t-s)^{\theta/2}\int_{0}^{s}\bigg(\frac{s-z}{\delta}\bigg)^{-(\rho+\theta)/2} e^{-\frac{\lambda(s-z)}{4\delta}}\bigg(1+\big\|X^{\epsilon,u} (z)\big\|_\h+\big\|Y^{\epsilon,u} (z)\big\|_\h\bigg)dz,
	\end{aligned}
	\end{equation*}
	which holds for $\theta\in(0,1/2)$, $\rho\in(1/2,1)$ and we used the Lipschitz continuity of $G$ to obtain the last line. Performing the substitution $\zeta=(s-z)/\delta$ then yields
	\begin{equation*}
	\begin{aligned}
	\big\|J^{\epsilon,u}_{3}&(s,t)\big\|_\h
	\leq C\delta^{-\theta/2}(t-s)^{\theta/2}\bigg(1+\sup_{t\in[0, T] }\big\|X^{\epsilon,u} (t)\big\|_\h+\sup_{t\in[0, T] }\big\|Y^{\epsilon,u} (t)\big\|_\h\bigg)\int_{0}^{s/\delta}\zeta^{-(\rho+\theta)/2} e^{-\frac{\lambda \zeta}{4}}d\zeta\\&
	\leq C_{\lambda, \theta}\delta^{-\theta/2}(t-s)^{\theta/2}\bigg(1+\sup_{t\in[0, T] }\big\|X^{\epsilon,u} (t)\big\|_\h+\sup_{t\in[0, T] }\big\|Y^{\epsilon,u} (t)\big\|_\h\bigg)\int_{0}^{\infty}(\lambda\zeta/4)^{-(\rho+\theta)/2}e^{-\lambda\zeta/4}d\zeta
	\end{aligned}
	\end{equation*}
	where $\rho+\theta<3/2$. The integral on the right-hand side is finite and, in fact, can be explicitly computed in terms of $\Gamma(1-\frac{\rho+\theta}{2})$ , where $\Gamma$ denotes the Gamma function. Letting $\epsilon$ be sufficiently small, taking expectation and using \eqref{xapriori} and  \eqref{yunifapriori} we deduce that
	\begin{equation*}\label{preJ_3}
	\ex\sup_{t,s\in[0,T], t\neq s }\frac{\big\|J^{\epsilon, u}_{3}(s,t)\big\|_\h}{|t-s|^{\theta/2}}\leq Ch(\epsilon)\delta^{\frac{\rho-1}{2}-\frac{\theta}{2}} \bigg(1+\|x_0\|_\h+\|y_0\|_\h\bigg).	\end{equation*}
	Choosing $\theta=2/9$ and $\rho=3/4$ we obtain, as we did for $J_2^{\epsilon,u}$, that for all $\beta<1/9$
	\begin{equation}\label{J_3}
	\ex\sup_{t,s\in[0,T], t\neq s }\frac{\big\|J^{\epsilon, u}_{3}(s,t)\big\|_\h}{|t-s|^{\beta}}\leq Ch(\epsilon)\delta^{-1/4}\bigg(1+\|x_0\|_\h+\|y_0\|_\h\bigg).	\end{equation}

	\noindent As for $J_4$,
		\begin{equation*}
	\begin{aligned}
	\big\|J^{\epsilon,u}_{4}(s,t)\big\|_\h&\leq \frac{h(\epsilon)}{\sqrt\delta}\bigg(\int_{s}^{t}\bigg\|S_2\bigg(\frac{t-z}{\delta}\bigg)\bigg\|^2_\h dz\bigg)^{\frac{1}{2}}    \|u\|_{L^2([0,T] ;\h )}\\&
	\leq N \frac{h(\epsilon)}{\sqrt\delta} \bigg(\int_{s}^{t} e^{-\frac{2\lambda(t-z)}{\delta}} dz\bigg)^{\frac{1}{2}}\\&=  N h(\epsilon) \bigg(\int_{0}^{\frac{t-s}{\delta}} e^{-2\lambda z} dz\bigg)^{\frac{1}{2}}
	\leq C_{N,\lambda} h(\epsilon) \delta^{-\frac{1}{2}}(t-s)^{\frac{1}{2}}
	\end{aligned}
	\end{equation*}
	\noindent with probability $1$. Thus, for $\beta\leq 1/2$,
	\begin{equation}\label{J_4}
	\ex\sup_{t,s\in[0,T], t\neq s }\frac{\big\|J^{\epsilon, u}_{4}(s,t)\big\|_\h}{|t-s|^{\beta}}\leq Ch(\epsilon)\delta^{-1/2}.\end{equation}
	
	\noindent The analysis for $J^{\epsilon, u}_{5}$ is similar to $J^{\epsilon, u}_{3}$. In particular,
	\begin{equation*}
	\begin{aligned}
	\big\|J^{\epsilon,u}_{5}(s,t)\big\|_\h&
	\leq \bigg(\frac{Ch(\epsilon)}{\sqrt{\delta}}\bigg)\delta^{-\theta/2}(t-s)^{\theta/2}\int_{0}^{s}\bigg\|(-A_2)^{\theta/2}S_2\bigg(\frac{s-z}{\delta}\bigg)u_2(z)\bigg\|_{\h}dz
	\\&
	\leq Ch(\epsilon)\delta^{-\theta/2}(t-s)^{\theta/2}\bigg(\frac{1}{\sqrt\delta}\bigg)\bigg(\int_{0}^{s}\bigg(\frac{s-z}{\delta}\bigg)^{-(\rho+\theta)} e^{-\frac{\lambda(s-z)}{2\delta}}dz\bigg)^{\frac{1}{2}}\|u_2\|_{L^2([0,T];\h  )}\\& \leq C_N h(\epsilon)\delta^{-\theta/2}(t-s)^{\theta/2}\bigg(\int_{0}^{\infty} \zeta^{-\rho+\theta} e^{-\lambda \zeta/2} d\zeta\bigg)^{\frac{1}{2}}\\&
	\leq C_{\lambda}h(\epsilon) \delta^{-\theta/2}(t-s)^{\theta/2} (\Gamma( 1-\rho-\theta))^{\frac{1}{2}},
	\end{aligned}
	\end{equation*}
	where we have chosen $\rho\in(1/2,1)$ and $\theta\in(0, 1/2)$ to satisfy $\rho+\theta<1$.
	Thus, for $\beta<\theta/2<1/4$
	
	\begin{equation}\label{J_5}
	\ex\sup_{t,s\in[0,T], t\neq s }\frac{\big\|J^{\epsilon, u}_{5}(s,t)\big\|_\h}{|t-s|^{\beta}}\leq Ch(\epsilon)\delta^{-1/2}.	\end{equation}

	\noindent Finally,  from \eqref{Holderconv} (see Appendix \ref{AppA}), there exists $\beta<1/4$ such that
	
	\begin{equation}\label{J_6}
	\begin{aligned}
	\ex\sup_{t,s\in[0,T], t\neq s }\frac{\big\|J^{\epsilon, u}_{6}(s,t)\big\|_\h}{|t-s|^{\beta}}&=\ex\big[ w_{A_2}^{\delta}\big]_{C^{\beta}([0,T];\h)}&\leq  C\delta^{\frac{\rho-1}{2}}\leq C\delta^{-1/4}
	\end{aligned}
	\end{equation}
	and the latter holds since $\rho\in(1/2, 1/2+2\beta)$. The argument is complete upon combining \eqref{J_1}-\eqref{J_6}.
\end{proof}
Before we conclude this section, let us gather some auxiliary estimates regarding the spatio-temporal regularity of the solution $\bar{X}$ of the averaged slow equation \eqref{x-aved}. These will be needed in the subsequent analysis of the controlled moderate deviations process $\eta^{\epsilon,u}$.

\begin{lem}\label{L:LimAverEq} (i) For $T<\infty$, there exists a constant $C>0$ such that
	\begin{equation}\label{xbarapriori}
	\begin{aligned}
	\sup_{t\in[0, T] }\|\bar{X}(t)\|^2_\h&\leq C(1+\|x_0\|^2_\h).
	\end{aligned}
	\end{equation}
	\noindent (ii) Let $T<\infty$, $a>0$ and $x_0\in H^{a}(0,L)$. For all $\theta<\frac{1}{4}\land \frac{a} {2}  $, there exists a constant $C>0$ such that
	\begin{equation}\label{Schauderbar}
	\|\bar{X}\|_{C^{\theta}([0,T];\h)}\leq C\big( 1+ \|x_0\|_{H^a}\big).
	\end{equation}\\
	\noindent (iii) Let $T<\infty, a\in(0,2]$ and $x_0\in H^{a}(0,L)$. Then, for all $t>0$ we have $\bar{X}(t)\in Dom(A_1)$. Moreover, there exists $C>0$ independent of $t$ such that for all $t\in(0,T]$
	\begin{equation}\label{deribar}
	\begin{aligned}
	\big\|A_1\bar{X}(t)\big\|_\h&\leq  C\big( t^{\frac{a}{2}-1}\|x_0\|_{H^a}+1+\big\|x_0\|_{H^a}\big).
	\end{aligned}
	\end{equation}
\end{lem}
To prove these estimates, one has to use the Lipschitz continuity of $\bar{F}$ (see Lemma \ref{Fbarlip}) along with the smoothing property \eqref{Sobosmoothing} of the analytic semigroup $S_1$. 	These results are well-known and we will only present the proof of \eqref{deribar} in Appendix \ref{AppA}.

	\section{A priori bounds for $\eta^{\epsilon,u}$ and the Kolmogorov equation}\label{Sec3}
In this section we aim to prove regularity estimates for the controlled moderate deviation process $\eta^{\epsilon, u}$, in Regimes $1$ and $2$, that are uniform over controls $u\in\mathcal{P}_N^T$ and small values of $\epsilon$. These will be used to show that the family $\{\eta^{\epsilon,u}, \epsilon\in(0,1), u\in\mathcal{P}^T_N\}$ is tight in $C([0,T];\h)$ (see Lemma \ref{etarzela} in Section \ref{Sec4}). To be precise, we are interested in studying the spatial Sobolev and temporal H\"older regularity of the process $\eta^{\epsilon,u}$. The main result of this section is given below:
\begin{prop}\label{etatightness}  Let $T<\infty$, $a>0$ and $x_0,y_0\in H^a(0,L)$. With $\nu$ as in Hypotheses \ref{A3a} and in both Regimes $1$ and $2$, there exist $\theta<(\frac{1}{2}-\nu)\wedge a  $, $\beta<(\frac{1}{4}-\frac{\nu}{2})\wedge\frac{a}{2}$, $\epsilon_0>0$ and $C>0$ independent of $\epsilon$ such that  \\
	\noindent (i) \begin{equation}\label{etaSob}
	\sup_{0<\epsilon<\epsilon_0, u\in\mathcal{P}_N^T}	\ex\sup_{t\in[0, T] } \|\eta^{\epsilon,u}(t)\|^2_{H^\theta}\leq C\big( 1+\|x_0\|^2_{H^a}+\|y_0\|^2_{H^a}\big)
	\end{equation}
	\noindent (ii) 	\begin{equation}\label{etaHolder}
	\begin{aligned}
	\sup_{0<\epsilon<\epsilon_0, u\in\mathcal{P}_N^T}\ex\big[\eta^{\epsilon,u}\big]_{C^\beta([0,T];\h)}\leq C\big( 1+\|x_0\|_{H^a}+  \|y_0\|_{H^a}\big).
	\end{aligned}
	\end{equation}
\end{prop}
\noindent To prove these estimates, we use a generalized version of decomposition \eqref{etadec1}. In particular, we fix
$\theta\in[0,1/2), 0\leq s< t\leq T$, $\chi\in Dom((-A_1)^{1+\frac{\theta}{2}})$ and write
\begin{equation}\label{etaspacetimedec}
\hspace*{-0.2cm}
\begin{aligned}
\blangle \eta^{\epsilon,u}(t)&-\eta^{\epsilon,u}(s)- \big(S_1(t-s)-I\big)\eta^{\epsilon,u}(s), (-A_1)^{\frac{\theta}{2}}\chi\brangle_\h\\&= \frac{1}{\sqrt{\epsilon}h(\epsilon)}\int_{s}^{t}\blangle F\big(X^{\epsilon,u}(z), Y^{\epsilon,u}(z) \big)- F\big(\bar{X}(z), Y^{\epsilon,u}(z) \big)     , S_1(t-z)(-A_1)^{\frac{\theta}{2}}\chi \rangle_\h dz
\\&
+ \int_{s}^{t}\blangle S_1(t-z)\Sigma\big(X^{\epsilon,u}(z), Y^{\epsilon,u}(z) \big)u_1(z),(-A_1)^{\frac{\theta}{2}}\chi \rangle_\h dz\\&+\frac{1}{h(\epsilon)} \int_{s}^{t} \langle S_1(t-z)\Sigma\big(X^{\epsilon,u}(z), Y^{\epsilon,u}(z) \big)dw_1(z), (-A_1)^{\frac{\theta}{2}}\chi\rangle_\h
\\&
+\frac{1}{\sqrt{\epsilon}h(\epsilon)}\int_{s}^{t}\blangle F\big(\bar{X}(z), Y^{\epsilon,u}(z) \big)- \bar{F}\big(\bar{X}(z) \big)    , S_1(t-z)(-A_1)^{\frac{\theta}{2}}\chi \rangle_\h dz\\&
=: I^{\epsilon,u}(s,t,\theta,\chi)+ II^{\epsilon,u}(s,t,\theta,\chi)+III^{\epsilon,u}(s,t,\theta,\chi)+IV^{\epsilon,u}(s,t,\theta,\chi).
\end{aligned}
\end{equation}

\noindent This decomposition allows us to study spatio-temporal regularity in a unified manner.  In Section \ref{1,2,3} we provide the necessary estimates for the terms $I^{\epsilon,u}$, $II^{\epsilon,u}$, $III^{\epsilon,u}$. As we mentioned in Section \ref{weakconv}, the term $IV^{\epsilon,u}$ requires a more careful analysis, which is done with the aid of the Kolmogorov equation \eqref{Kolmeq}. This is the  subject of Section \ref{4}. Finally, we prove Proposition \ref{etatightness} in Section \ref{etaproof}.
\begin{rem} The reason for choosing our test functions $\chi\in Dom((-A_1)^{1+\frac{\theta}{2}})$ is related to the treatment of term $IV^{\epsilon,u}$ and will become clear in Section \ref{4} (see Lemma \ref{IVitodeclem}).
\end{rem}
	
\subsection{Estimates for $I^{\epsilon,u}$, $II^{\epsilon,u}$, $III^{\epsilon,u}$  }\label{1,2,3}
The proofs of the three lemmas in this section  have the following structure: First, we prove a preliminary space-time estimate which depends linearly and continuously on the test function $\chi$ in the topology of $\h$. Since $\chi$ is smooth, we can extend the latter by density to arbitrary test functions in $\h$. Finally, we set $s=0$ to prove a spatial Sobolev-type estimate, or $\theta=0$ to prove a temporal equicontinuity-type estimate,  uniformly over $\chi\in B_\h$. These estimates hold in both Regimes $1$ and $2$ (see \eqref{Regimes}).

\begin{lem}\label{Ibnds} Let $T<\infty$, $t\in[0,T]$, $\theta\in[0,1/2)$ and $I^{\epsilon,u}$ as in \eqref{etaspacetimedec}. For all $\epsilon>0, u\in\mathcal{P}^T_N$, there exists a constant $C>0$, independent of $\epsilon$, such that
	\begin{equation}\label{ISob}
	\begin{aligned}
	&	\sup_{\chi\in B_\h}\big|I^{\epsilon,u}(0,t,\theta,\chi)\big|^2\leq C \int_{0}^{t}(t-z)^{-\theta}\sup_{r\in[0,z]}\big\|\eta^{\epsilon,u}(r)\big\|^2_{\h} dz,\;\;\pr-\text{a.s.}	       	\end{aligned}
	\end{equation}	
	and
	\begin{equation}\label{Iholder}
	\begin{aligned}
	&	\ex\bigg(\sup_{\overset{s,t\in[0, T]}{ t\neq s}}\sup_{\chi\in B_\h}\frac{\big|I^{\epsilon,u}(s,t,0,\chi)\big|}{  |t-s|  }\bigg)\leq C \ex\sup_{t\in[0,T]}\big\|\eta^{\epsilon,u}(t)\big\|_{\h}\;.
	\end{aligned}
	\end{equation}

\end{lem}
\begin{proof}
	Let $\chi\in Dom((-A_1)^{1+\theta/2})$. Using the analyticity of the semigroup $S_1$ and the Lipschitz continuity of $F$,
	\begin{equation*}\label{IVprebnd}
	\begin{aligned}
	\big|I^{\epsilon,u}(s,t,\theta,\chi)\big|&\leq \frac{1}{\sqrt{\epsilon}h(\epsilon)} \int_{s}^{t} \big\|(-A_1)^{\frac{\theta}{2}} S_1(t-z)\big[F\big(X^{\epsilon,u}(z), Y^{\epsilon,u}(z) \big)-F\big(\bar{X}(z) , Y^{\epsilon,u}(z) \big)\big]\big\|_\h\| \chi\|_\h\;dz
	\\&
	\leq   \frac{C_f}{\sqrt{\epsilon}h(\epsilon)} \| \chi\|_\h\int_{s}^{t} (t-z)^{-\theta/2}\big\|X^{\epsilon,u}(z)-\bar{X}(z)\big\|_{\h}\;dz
	\\&\leq  C\| \chi\|_\h\int_{s}^{t} (t-z)^{-\theta/2}\sup_{r\in[s,z]}\big\|\eta^{\epsilon,u}(r)\big\|_{\h}\;dz.
	\end{aligned}
	\end{equation*}
	Since $Dom((-A_1)^{1+\frac{\theta}{2}})$ is dense as a subspace of $\h$,
	we can approximate any element of $\h$ by a sequence $\{\chi_m\}_{m\in\N}\subset Dom((-A_1)^{1+\frac{\theta}{2}})$ in the topology of $\h$. Hence the last estimate holds, with probability $1$, for each $\chi\in \h$. Choosing $\chi\in B_\h$, we set $s=0$ and take expectation to obtain \eqref{ISob}. Setting $\theta=0$ yields
	\begin{equation*}
	\begin{aligned}
	&	\big|I^{\epsilon,u}(s,t,0,\chi)\big|
	\leq C(t-s)\sup_{t\in[0,T]} \big\|\eta^{\epsilon,u}(t)\big\|_{\h}
	\end{aligned}
	\end{equation*}
	and \eqref{Iholder} follows by taking expectation. The proof is complete.
\end{proof}

\begin{lem}\label{IIbnds} Let $T<\infty,$ $x_0,y_0\in\h$, $\nu<1/2$ as in Hypothesis \ref{A3a} and $II^{\epsilon,u}$ as in \eqref{etaspacetimedec}. There exist $\theta<1/2-\nu$, $\beta<1/4-\nu/2$ and a constant $C>0$, independent of $\epsilon$, such that
	
	\begin{equation}\label{IISob}
	\begin{aligned}
	&	\sup_{\epsilon>0, u\in\mathcal{P}^T_N}\ex\bigg(\sup_{t\in[0, T] }\sup_{\chi\in B_\h}\big|II^{\epsilon,u}(0,t,\theta,\chi)\big|^\frac{2}{\nu}\bigg)\leq C\big( 1+\|x_0\|_\h^{\frac{2}{\nu}}+\|y_0\|^\frac{2}{\nu}_\h\big)
	\end{aligned}
	\end{equation}	
	and
	\begin{equation}\label{IIholder}
	\begin{aligned}
	&		\sup_{\epsilon>0, u\in\mathcal{P}^T_N}\ex\bigg(\sup_{\overset{s,t\in[0, T]}{ t\neq s}}\sup_{\chi\in B_\h}\frac{\big|II^{\epsilon,u}(s,t,0,\chi)\big|}{  |t-s|^{\beta}  }\bigg)\leq C\big( 1+\|x_0\|_\h+\|y_0\|_\h\big).
	\end{aligned}
	\end{equation}
\end{lem}
\begin{proof} Let $\chi\in Dom((-A_1)^{1+\frac{\theta}{2}})$. An application of Lemma \ref{sigmacont}(i) yields
	\begin{equation*}
	\begin{aligned}
	|II^{\epsilon,u}(s,t,\theta,\chi)&\big|\leq	\int_{s}^{t} \|(-A_1)^{\frac{\theta}{2}} S_1(t-z)\Sigma\big(X^{\epsilon,u}(z), Y^{\epsilon,u}(z) \big)u_1(z)\|_\h \|\chi\|_\h  dz\\&\leq C \|\chi\|_\h\int_{s}^{t} (t-z)^{-(\rho+\theta)/2}\big\|\Sigma^{*}\big(X^{\epsilon,u}(z), Y^{\epsilon,u}(z) \big)\big\|_{\mathscr{L}(L^\infty(0,L);\h)}\|u_1(z)\|_\h  dz\\&
	\leq C   \|\chi\|_\h\int_{s}^{t} (t-z)^{-(\rho+\theta)/2}\bigg(1+\big\|X^{\epsilon,u}(z)\big\|_\h+\big\|Y^{\epsilon,u}(z)\big\|^{\nu}_\h\bigg)\|u_1(z)\|_\h  dz,
	\end{aligned}
	\end{equation*}
	where $\rho\in(1/2,1)$ and we used Hypothesis \ref{A3a} to obtain the third line. Using a density argument as in the proof of Lemma \ref{Ibnds} it follows that the estimate holds for each $\chi\in \h$. Choosing $\chi\in B_\h$, we apply the Cauchy-Schwarz inequality to deduce that
	\begin{equation*}
	\begin{aligned}
	|II^{\epsilon,u}(s,t,\theta,\chi)\big|&\leq C  \bigg(\int_{0}^{T} \|u_1(z)\|^2_\h dz\bigg)^{1/2}\bigg[\int_{s}^{t}(t-z)^{-\rho-\theta}\bigg(1+\big\|X^{\epsilon,u}(z)\big\|^2_\h+\big\|Y^{\epsilon,u}(z)\big\|^{2\nu}_\h\bigg) dz\bigg]^{\frac{1}{2}},\end{aligned}
	\end{equation*}
	 with probability $1$. Applying H\"older's inequality with $p=1/\nu$, $q=1/(1-\nu)$
	\begin{equation}\label{IIprebnd}
	\begin{aligned}
	|II^{\epsilon,u}(s,t,\theta,\chi)\big|&\leq C N\bigg[\int_{0}^{t-s}z^{-q(\rho+\theta)}dz\bigg]^{\frac{1}{2q}}\bigg[\int_{0}^{T}\bigg(1+\big\|X^{\epsilon,u}(z)\big\|^{2/\nu}_\h+\big\|Y^{\epsilon,u}(z)\big\|^{2}_\h\bigg) dz\bigg]^{\frac{\nu}{2}}. \end{aligned}
	\end{equation}
	\noindent Since $\nu<1/2$ we can choose $\rho\in(1/2, 1-\nu)$ and $\theta<1-\nu-\rho=-\rho+1/q$ so that $\int_{0}^{t-s}z^{-q(\rho+\theta)}dz\leq CT^{1-q(\rho+\theta)}$. Setting $s=0$ in \eqref{IIprebnd} we obtain
	\begin{equation*}
	\hspace*{-0.5cm}
	\begin{aligned}
	|II^{\epsilon,u}(0,t,\theta,\chi)\big|&\leq C_N T^{(1-\nu-\rho-\theta)/2}\bigg[1+\sup_{t\in[0,T]}\big\|X^{\epsilon,u}(z)\big\|^{2/\nu}_\h+\int_{0}^{T}\big\|Y^{\epsilon,u}(z)\big\|^{2}_\h dz\bigg]^{\frac{\nu}{2}} \end{aligned}
	\end{equation*}
	and \eqref{IISob} follows by taking expectation and applying \eqref{xapriori} and \eqref{yapriori}. As for \eqref{IIholder}, we set $\theta=0$ in \eqref{IIprebnd} to deduce that
	\begin{equation*}
	\hspace*{-0.5cm}
	\begin{aligned}
	\frac{|II^{\epsilon,u}(s,t,\theta,\chi)\big|}{(t-s)^\beta}&\leq C \bigg[1+\sup_{t\in[0,T]}\big\|X^{\epsilon,u}(z)\big\|^{2/\nu}_\h+\int_{0}^{T}\big\|Y^{\epsilon,u}(z)\big\|^{2}_\h dz\bigg]^{\frac{\nu}{2}}, \end{aligned}
	\end{equation*}
	for $\beta\leq (1-\nu-\rho)/2<(1-\nu)/2$. In view of the a priori bounds \eqref{xapriori} and \eqref{yapriori}, the proof is complete.\end{proof}

\begin{lem}\label{IIIbnds}  Let $T<\infty$, $\nu<1/2$ as in Hypothesis \ref{A3a} and $III^{\epsilon,u}$ as in \eqref{etaspacetimedec} . There exist $\epsilon_0>0$, $\theta<\frac{1}{2}-\nu$, $\beta<\frac{1}{4}-\frac{\nu}{2}$ and a constant $C>0$, independent of $\epsilon$, such that
	
	\begin{equation}\label{IIISob}
	\begin{aligned}
	&	\sup_{\epsilon<\epsilon_0, u\in\mathcal{P}^T_N}\ex\bigg(\sup_{t\in[0, T] }\sup_{\chi\in B_\h}\big|III^{\epsilon,u}(0,t,\theta,\chi)\big|^\frac{2}{\nu}\bigg)\leq C\big( 1 +\|x_0\|^{\frac{2}{\nu}}_\h+\|y_0\|_\h^\frac{2}{\nu}\big)
	\end{aligned}
	\end{equation}	
	and
	\begin{equation}\label{IIIholder}
	\begin{aligned}
	&		\sup_{\epsilon<\epsilon_0, u\in\mathcal{P}^T_N}\ex\bigg(\sup_{\overset{s,t\in[0, T]}{ t\neq s}}\sup_{\chi\in B_\h}\frac{\big|III^{\epsilon,u}(s,t,0,\chi)\big|}{  |t-s|^\beta  }\bigg)\leq  C\big(1+\|x_0\|_\h+\|y_0\|_\h\big).
	\end{aligned}
	\end{equation}
\end{lem}

\begin{proof}
	\noindent Let $\theta\in[0,1/2)$, $\chi\in Dom((-A_1)^{1+\frac{\theta}{2}})$ and $a\in(0,1/2)$. From the  stochastic factorization formula \eqref{stofact} we can write
	\begin{equation*}\label{V5fin}
	III^{\epsilon,u}(s,t,\theta,\chi)	=  \frac{\sin(a\pi)}{h(\epsilon) \pi}\bigg\langle  \int_{s}^{t}(t-z)^{a-1}(-A_1)^{\frac{\theta}{2}}S_1(t-z) M^{\epsilon,u}_a(s,z,z;1)dz\;,\chi\bigg\rangle_\h\;,
	\end{equation*}
	where, for $t_1\leq t_2\leq t_3$,
	\begin{equation*}\label{stofactoreg}
	M^{\epsilon,u}_a(t_1,t_2,t_3;1):=\int_{t_1}^{t_2}(t_3-\zeta)^{-a}S_1(t_3-\zeta)\Sigma\big( X^{\epsilon,u}(\zeta), Y^{\epsilon,u}(\zeta)    \big)dw_1(\zeta).
	\end{equation*}
	Thus,
	\begin{equation}\label{IIIprebnd}
	\begin{aligned}
	&\big|III^{\epsilon,u}(s,t,\theta,\chi)\big|\leq\frac{C_a}{h(\epsilon)}\|\chi\|_\h \int_{s}^{t}(t-z)^{a-1}\big\|(-A_1)^{\frac{\theta}{2}} M^{\epsilon,u}_a(s,z,z;1)\big\|_\h dz.
	\end{aligned}
	\end{equation}
	From a density argument (see proof of Lemma \ref{Ibnds}), the last estimate holds with probability $1$ for all $\chi\in B_\h$.

\noindent 	We start by proving \eqref{IIIholder}. To this end, set $\theta=0$ in \eqref{IIIprebnd} and apply H\"older's inequality for $q>1/a>2$ to deduce that
	\begin{equation*}
	\begin{aligned}
	\big|III^{\epsilon,u}(s,t,0,\chi)\big|&\leq\frac{C_a}{h(\epsilon)} \|\chi\|_\h \int_{s}^{t}(t-z)^{a-1}\big\|M^{\epsilon,u}_a(s,z,z;1)\big\|_\h dz\\&
	\leq\frac{C}{h(\epsilon)}\|\chi\|_\h \bigg(\int_{s}^{t}(t-z)^{p(a-1)}dz\bigg)^{\frac{1}{p}}\bigg(\int_{s}^{t}\big\|M^{\epsilon,u}_a(s,z,z;1)\big\|^q_\h dz\bigg)^{\frac{1}{q}}.
	\end{aligned}
	\end{equation*}
	\noindent Since $M^{\epsilon,u}_a(s,z,z)=M^{\epsilon,u}_a(0,z,z;1)-M^{\epsilon,u}_a(0,s,z;1),$
	\begin{equation*}
	\begin{aligned}
	h(\epsilon)\sup_{\chi\in B_\h}\frac{\big|III^{\epsilon,u}(s,t,0,\chi)\big|}{(t-s)^{a-1/q}}&
\leq C_q\bigg(\int_{0}^{T}\sup_{s\in[0,z]}\big\|M^{\epsilon,u}_a(0,s,z;1)\big\|^q_\h dz\bigg)^{\frac{1}{q}}.
	\end{aligned}
	\end{equation*}
	Taking expectation, we apply Jensen's inequality followed by the Burkholder-Davis-Gundy inequality to obtain
	\begin{equation*}
	\begin{aligned}
	\ex \sup_{\overset{s,t\in[0, T]}{ t\neq s}}&\sup_{\chi\in B_\h}\frac{\big|III^{\epsilon,u}(s,t,0,\chi)\big|}{|t-s|^{a-1/q}}
	\leq \frac{C}{h(\epsilon)} \bigg(\int_{0}^{T}\ex\sup_{s\in[0,z]}\big\|M^{\epsilon,u}_a(0,s,z;1)\big\|^q_\h dz\bigg)^{\frac{1}{q}}   \\&\leq \frac{C}{h(\epsilon)} \bigg( \int_{0}^{T}\bigg(  \int_{0}^{z}(t-\zeta)^{-2a}\ex\|S_1(t-\zeta)\Sigma\big( X^{\epsilon,u}(\zeta), Y^{\epsilon,u}(\zeta)    \big)\|^2_{\mathscr{L}_2(\h)}dw_1(\zeta)\bigg)^{\frac{q}{2}} dz\bigg)^{\frac{1}{q}}.
	\end{aligned}
	\end{equation*}
	\noindent 	From Lemma \ref{sigmacont}(ii) (with $B=\Sigma( X^{\epsilon,u}(\cdot), Y^{\epsilon,u}(\cdot)   ), P_n=I$) and Hypothesis \ref{A3a}
	\begin{equation}\label{IIIpebnd}
	\small
	\begin{aligned}
	\ex &\sup_{\overset{s,t\in[0, T]}{ t\neq s}}\sup_{\chi\in B_\h}\frac{\big|III^{\epsilon,u}(s,t,0,\chi)\big|}{|t-s|^{a-1/q}}
	\leq \frac{C}{h(\epsilon)}\bigg(\int_{0}^{T}\bigg(   \int_{0}^{z}(z-\zeta)^{-2a-\rho}\bigg(1+\ex\big\| X^{\epsilon,u}(\zeta)\big\|^2_\h+ \ex\big\|Y^{\epsilon,u}(\zeta)  \big\|^{2\nu}_\h \bigg) d\zeta\bigg)^{\frac{q}{2}} dz\bigg)^{\frac{1}{q}}.
	\end{aligned}
	\end{equation}
	Next, choose $a<\frac{1}{4}-\frac{\nu}{2}\in(0, 1/4) $  and $\rho<1-\nu-2a\in(1/2,1)$.  Applying H\"older's inequality with exponents $1/\nu$ and $1/(1-\nu)$, followed by Jensen's inequality, we obtain
	\begin{equation*}
	\small
	\begin{aligned}
	\ex\int_{0}^{z}(z-\zeta)^{-2a-\rho} &\bigg(1+\big\|X^{\epsilon,u}(\zeta)\big\|_\h^2+\big\|Y^{\epsilon,u}(\zeta)\big\|_\h^{2\nu} \bigg)d\zeta
	\leq CT^{1-\nu-2a-\rho}\bigg[\int_{0}^{T} \bigg(1+\ex\big\|X^{\epsilon,u}(\zeta)\big\|_\h^{\frac{2}{\nu}}+\ex\big\|Y^{\epsilon,u}(\zeta)\big\|_\h^{2} \bigg)d\zeta\bigg]^{\nu}.
	\end{aligned}
	\end{equation*}
	Letting $q=2/\nu>2$, it follows that
	\begin{equation*}
	\begin{aligned}
	\bigg[\int_{0}^{T}\bigg(\ex&\int_{0}^{z}(z-\zeta)^{-2a-\rho} \bigg(1+\big\|X^{\epsilon,u}(\zeta)\big\|_\h^2+\big\|Y^{\epsilon,u}(\zeta)\big\|_\h^{2\nu} \bigg)d\zeta\bigg)^{\frac{q}{2}} dz\bigg]^{\frac{1}{ q}}\\&\leq C T^{\nu/2} \bigg(\int_{0}^{T} \bigg(1+\ex\big\|X^{\epsilon,u}(\zeta)\big\|_\h^{2/\nu}+\ex\big\|Y^{\epsilon,u}(\zeta)\big\|_\h^{2} \bigg)d\zeta\bigg)^\frac{\nu}{2}.
	\end{aligned}
	\end{equation*}
	\noindent Combining the latter with \eqref{IIIpebnd} yields
	\begin{equation*}
	\begin{aligned}
	\ex &\sup_{\overset{s,t\in[0, T]}{ t\neq s}}\sup_{\chi\in B_\h}\frac{\big|III^{\epsilon,u}(s,t,0,\chi)\big|}{|t-s|^{a-1/q}}\leq\frac{C_{T,\nu}}{h(\epsilon)} \bigg(\int_{0}^{T} \bigg(1+\ex\big\|X^{\epsilon,u}(\zeta)\big\|_\h^{2/\nu}+\ex\big\|Y^{\epsilon,u}(\zeta)\big\|_\h^{2} \bigg)d\zeta\bigg)^\frac{\nu}{2}.
	\end{aligned}
	\end{equation*}
	\noindent Using estimates \eqref{xapriori} and \eqref{yapriori} and noting that $h(\epsilon)\to \infty$ as $\epsilon\to 0$, \eqref{IIIholder} follows.
	\noindent Similarly, \eqref{IIISob} can be proved by setting $s=0$ in \eqref{IIIprebnd}. This yields
	\begin{equation*}
	\begin{aligned}
	&\ex\bigg(\sup_{t\in[0, T] }\sup_{\chi\in B_\h}\big|III^{\epsilon,u}(0,t,\theta,\chi)\big|^\frac{2}{\nu}\bigg)\leq\frac{C_a}{h(\epsilon)}\ex\bigg\| \int_{0}^{T}(t-z)^{a-1}S_1(t-z)M^{\epsilon,u }_a(0,z,z;1)dz \bigg\|^\frac{2}{\nu}_{H^\theta}\\ &
	\leq   \frac{C_aT^{a-1/q}}{h(\epsilon)} \bigg(\int_{0}^{T}\ex\big\|(-A_1)^{\frac{\theta}{2}}M^{\epsilon,u}_a(0,z,z;1)\big\|^q_\h dz\bigg)^{\frac{2}{\nu q}} \\&\leq \frac{C_{T,a}}{h(\epsilon)}\bigg(\int_{0}^{T}\bigg(   \int_{0}^{z}(z-\zeta)^{-2a}\ex \big\|(-A_1)^{\frac{\theta}{2}}S_1(z-\zeta)\Sigma\big( X^{\epsilon,u}(\zeta), Y^{\epsilon,u}(\zeta)    \big)\big\|^2_{\mathscr{L}_2(\h)}d\zeta\bigg)^{\frac{q}{2}} dz\bigg)^{\frac{2}{q\nu}},
	\end{aligned}
	\end{equation*}
	for $\theta\in(0,1/2)$. 	In view of \eqref{HSbnd}, we can choose $\theta<\frac{1}{2}-\nu\in(0,\frac{1}{2})$, $a<\frac{1}{4}-\frac{\nu}{2}-\frac{\theta}{2}\in(0, 1/4) $  and $\rho<1-\nu-2a\in(\theta+1/2,1)$ and then apply H\"older's inequality with exponents $1/\nu$ and $1/(1-\nu)$ to obtain
	\begin{equation*}
	\begin{aligned}
	\ex\bigg(\sup_{t\in[0, T] }&\sup_{\chi\in B_\h}\big|III^{\epsilon,u}(0,t,\theta,\chi)\big|^\frac{2}{\nu}\bigg)\\&\leq \frac{C}{h(\epsilon)}\bigg(\int_{0}^{T}\bigg(   \int_{0}^{z}(z-\zeta)^{-2a-\rho}\ex\big\|\Sigma\big( X^{\epsilon,u}(\zeta), Y^{\epsilon,u}(\zeta)    \big)\big\|^2_{\mathscr{L}(L^{\infty}(0,L);\h)}d\zeta\bigg)^{\frac{q}{2}} dz\bigg)^{\frac{2}{q\nu}} \\&\leq C
	\int_{0}^{T} \bigg(1+\ex\big\|X^{\epsilon,u}(\zeta)\big\|_\h^{2/\nu}+\ex\big\|Y^{\epsilon,u}(\zeta)\big\|_\h^{2} \bigg)d\zeta.
	\end{aligned}
	\end{equation*}
	\noindent Noting that $a<1/2$ can be arbitrarily small, we apply \eqref{xapriori} and \eqref{yapriori} and the result follows.
\end{proof}
\begin{rem} The estimates derived in this section do not require any regularity for the initial conditions of the controlled system \eqref{controlledsystem}. Such considerations have to be taken into account in the next section.
\end{rem}

\subsection{The term $IV^{\epsilon,u}$}\label{4}

\noindent This section is devoted to the analysis of the last term in the decomposition \eqref{etaspacetimedec}. As we mentioned above, this term requires additional work due to the singular coefficient $1/\sqrt{\epsilon}h(\epsilon)$. Throughout the rest of this paper we choose the small parameter $c(\epsilon)$ in the Kolmogorov equation \eqref{Kolmeq} to be
\begin{equation}\label{cchoice}
c(\epsilon):=\sqrt{\epsilon}.
\end{equation}

\noindent	Now, let $P_n:\h\rightarrow\text{span}\{e_{2,1},\dots,e_{2,n}\}$ be an orthogonal projection onto the $n$-dimensional subspace spanned by the eigenvectors $e_{2,1},\dots,e_{2,n}$ of $A_2$ (see Hypothesis \ref{A1a}), $u_{2,n}:=P_nu_2$ be the projection of the control $u_2$ and
$$w_{2,n}(t)
=\sum_{k=1}^{n} e_{2,k}w_{2}(t, e_{2,k})      $$
be the projection of the cylindrical Wiener process $w_2$.
Consider the family of $n$-dimensional processes
$$Y_n^{\epsilon,u}:=P_nY^{\epsilon,u}\;,\;\;n\in\N.$$
These processes satisfy the controlled stochastic evolution equations
\begin{equation}\label{projy}
\left\{
\begin{aligned}
&dY_n^{\epsilon,u}(t)= \frac{1}{\delta}\big[A_2 Y_n^{\epsilon,u}(t)+ P_nG\big( X^{\epsilon,u}(t),Y^{\epsilon,u}(t)\big)\big]+\frac{h(\epsilon)}{\sqrt{\delta}}u_{2,n}(t)dt +\frac{1}{\sqrt{\delta}}\; dw_{2,n}(t)\\&
t>0,
Y_n^{\epsilon,u}(0)=P_ny_0\in\h.
\end{aligned}\right.
\end{equation}
Next, recall that
\begin{equation*}\label{T}
IV^{\epsilon,u}(s,t,\theta,\chi)=\frac{1}{\sqrt{\epsilon}h(\epsilon)}\int_{s}^{t}\blangle F\big(\bar{X}(z), Y^{\epsilon,u}(z) \big)- \bar{F}\big(\bar{X}(z) \big)    , S_1(t-z)(-A_1)^{\frac{\theta}{2}}\chi \rangle_\h dz.
\end{equation*}
\noindent For $\chi\in Dom((-A_1)^{1+\frac{\theta}{2}})$  we can further decompose  this into
\begin{equation}\label{Tdec}
\begin{aligned}
& \frac{1}{\sqrt{\epsilon}h(\epsilon)} \int_{s}^{t}\blangle F\big(\bar{X}(z), Y_n^{\epsilon,u}(z) \big)-\bar{F}\big(\bar{X}(z) \big),  S_1(t-z)(-A_1)^{\frac{\theta}{2}}\chi\brangle_\h dz\\&+  \frac{1}{\sqrt{\epsilon}h(\epsilon)} \int_{s}^{t}\blangle F\big(\bar{X}(z), Y^{\epsilon,u}(z) \big)-F\big(\bar{X}(z), Y_n^{\epsilon,u}(z) \big), S_1(t-z)(-A_1)^{\frac{\theta}{2}}\chi\brangle_\h dz  \\&
=: 	T_1^{\epsilon,u}(s,t,n,\theta,\chi)+T_2^{\epsilon,u}(s,t,n,\theta,\chi)
\end{aligned}
\end{equation}
\noindent and then rewrite $T_1^{\epsilon,u}$, with the aid of It\^o's formula, in order to deal with the asymptotically singular scaling.
In particular, consider the real-valued map
\begin{equation*}\label{Xi}
[s,t]\times\h\times Dom(A_2)\ni (z,x,y)\longmapsto \Theta(z,x,y):=\Phi^\epsilon_{ S_1(t-z)(-A_1)^{\frac{\theta}{2}}\chi}(x,y)\in\R,
\end{equation*}
where $\Phi^\epsilon_{\cdot}$ denotes the strict solution of the Kolmogorov equation given by \eqref{Feynman}. In view of \eqref{Riesz},
\begin{equation}
\label{Xiriesz}\Theta(z,x,y)=\langle \Psi^\epsilon(x,y), S_1(t-z)(-A_1)^{\frac{\theta}{2}}\chi\rangle_\h \end{equation} and
\begin{equation}\label{Xider}
\begin{aligned}
&\partial_z\Theta(z,x,y)=\langle \Psi^{\epsilon}(x,y), (-A_1)^{1+\frac{\theta}{2}}S_1(t-z)\chi\rangle_\h\;\;,\\&
D^v_x\Theta(z,x,y)= D^v_x\Phi^\epsilon_{ S_1(t-z)(-A_1)^{\frac{\theta}{2}}\chi}(x,y)=   \langle   \Psi^\epsilon_1(x,y)v   , S_1(t-z)(-A_1)^{\frac{\theta}{2}}\chi \rangle_\h\;\;,\\&
D^v_y\Theta(z,x,y)= D^v_y\Phi^\epsilon_{ S_1(t-z)(-A_1)^{\frac{\theta}{2}}\chi}(x,y)=   \langle   \Psi^\epsilon_2(x,y)v  , S_1(t-z)(-A_1)^{\frac{\theta}{2}} \chi \rangle_\h\;,
\end{aligned}
\end{equation}
where $D^v_{\cdot}$ denotes partial Fr\'{e}chet differentiation in the direction of $v\in\h$.
\noindent Moreover, from the last estimate in \eqref{phiestimates} and the Riesz representation theorem, there exists $\Psi^{\epsilon,n}_3(x,y)\in\h$ such that
\begin{equation}\label{phitr}
\begin{aligned}
&\text{tr}\big[(P_n-I)D^2_y \Phi^\epsilon_{ \chi}(x,y)\big]=\blangle \Psi^{\epsilon,n}_3(x,y), \chi\brangle_\h\;\text{and}\\&
\|\Psi^{\epsilon,n}_3(x,y)\|_\h\leq \frac{c}{c(\epsilon)}\big(1+\|x\|_\h+\|y\|_\h\big).
\end{aligned}
\end{equation}
The latter implies that
\begin{equation}\label{Xitr}
\begin{aligned}
\text{tr}\big[(P_n-I)D^2_y \Theta(z,x,y)\big]= \langle   \Psi^{\epsilon,n}_3(x,y) , S_1(t-z)(-A_1)^{\frac{\theta}{2}} \chi \rangle_\h.
\end{aligned}
\end{equation}
\noindent Noting that, for each $t\geq0$, $Y_n^{\epsilon,u}(t)\in Dom(A_2)$ almost surely,
we can apply It\^o's formula to $\Theta(t, \bar{X}(t), Y_n^{\epsilon,u}(t) )$ to obtain the following:
\begin{lem}\label{IVitodeclem} Let $n\in\N, T<\infty,\epsilon>0,\theta\geq 0$, $0\leq s\leq t\leq T$, $\chi\in Dom((-A_1)^{1+\theta/2})$ and define
		\begin{equation}\label{T3}
	\begin{aligned}
	&T_3^{\epsilon,u}( s,t,n,\theta,\chi):=  \frac{1}{2\sqrt{\epsilon}h(\epsilon)}  \int_{s}^{t}  \blangle \Psi^{\epsilon,n}_3\big(\bar{X}(z),Y_n^{\epsilon,u}(z)\big), S_1(t-z)(-A_1)^{\frac{\theta}{2}}\chi\brangle_\h dz\\&
	+\frac{1}{\sqrt{\epsilon}h(\epsilon)}\int_{s}^{t}   \blangle \Psi^\epsilon_2\big(\bar{X}(z),Y_n^{\epsilon,u}(z)\big)\big[ P_nG\big( \bar{X}(z), Y^{\epsilon,u}(z) \big)- G\big( \bar{X}(z), Y_n^{\epsilon,u}(z) \big)\big], S_1(t-z)(-A_1)^{\frac{\theta}{2}}\chi\brangle_\h dz.
\end{aligned}
	\end{equation}
	With $\Psi^\epsilon,\Psi^\epsilon_1,\Psi^\epsilon_2,\Psi^{\epsilon,n}_3, T_1^{\epsilon,u}, T_2^{\epsilon,u}$ as in \eqref{Riesz}, \eqref{phitr} and \eqref{Tdec}, we have
	\begin{equation}\label{Itoeta3}
	\hspace*{-0.2cm}
	\begin{aligned}
	IV^{\epsilon,u}&(s,t,\theta,\chi)=\\&
	-\frac{\delta}{  \sqrt{\epsilon}h(\epsilon)    }      \blangle \Psi^\epsilon\big(\bar{X}(t),Y_n^{\epsilon,u}(t)\big)-\Psi^\epsilon\big(\bar{X}(s),Y_n^{\epsilon,u}(s)\big), S_1(t-s)(-A_1)^{\frac{\theta}{2}}\chi\brangle_\h\\& +\frac{\delta}{  \sqrt{\epsilon}h(\epsilon)    }\int_{s}^{t}\blangle \Psi^\epsilon\big(\bar{X}(z),Y_n^{\epsilon,u}(z)\big)-\Psi^\epsilon\big(\bar{X}(t),Y_n^{\epsilon,u}(t)\big), S_1(t-z)(-A_1)^{1+\frac{\theta}{2}}\chi\brangle_\h dz\\&
	+\frac{\delta}{  \sqrt{\epsilon}h(\epsilon)    }\int_{s}^{t}    \blangle \Psi^\epsilon_1\big(\bar{X}(z),Y_n^{\epsilon,u}(z)\big)\big[ A_1 \bar{X}(z)+\bar{F}\big(\bar{X}(z)\big)\big], S_1(t-z)(-A_1)^{\frac{\theta}{2}}\chi \brangle_\h dz
	\\&+\frac{c(\epsilon)}{  \sqrt{\epsilon}h(\epsilon)    }\int_{s}^{t}\blangle \Psi^\epsilon\big(\bar{X}(z), Y_n^{\epsilon,u}(z)\big), S_1(t-z)(-A_1)^{\frac{\theta}{2}}\chi\brangle_\h dz \\&
	+\frac{\sqrt{\delta}} {\sqrt{\epsilon}  }  \int_{s}^{t}  \blangle \Psi^\epsilon_2\big(\bar{X}(z),Y_n^{\epsilon,u}(z)\big)u_{2,n}(z), S_1(t-z)(-A_1)^{\frac{\theta}{2}}\chi \brangle_\h dz\\&+  \frac{\sqrt{\delta}}{\sqrt{\epsilon}h(\epsilon)}    \int_{s}^{t}  \blangle  (-A_1)^{\frac{\theta}{2}}S_1(t-z)\Psi^\epsilon_2\big(\bar{X}(z),Y_n^{\epsilon,u}(z)\big) d w_{2,n}(z) ,\chi \brangle_\h + R^{\epsilon,u}( s,t,n,\theta,\chi)\\&=:\sum_{k=1}^{6} IV_k^{\epsilon,u}(s,t,n,\theta,\chi)+R^{\epsilon,u}( s,t,n,\theta,\chi),
	\end{aligned}
	\end{equation}
	where
	\begin{equation}
	\label{Rem}
	\begin{aligned}
	&R^{\epsilon,u}( s,t,n,\theta,\chi):=T_2^{\epsilon,u}( s,t,n,\theta,\chi)+T_3^{\epsilon,u}( s,t,n,\theta,\chi).
	\end{aligned}
	\end{equation}
\end{lem}
\noindent The proof of Lemma \ref{IVitodeclem} is deferred to Appendix \ref{AppB}.
\begin{rem}
	Note that the terms $IV_k^{\epsilon,u}$, $k=1,\dots,6$ are free from asymptotically singular coefficients. This comes at the cost of introducing the unbounded operator $(-A_1)$ in the term $IV_2^{\epsilon,u}$.
\end{rem}
\noindent We can now proceed to estimate each term in \eqref{Itoeta3} in both Regimes $1$ and $2$. The terms $IV_1^{\epsilon,u}, IV_2^{\epsilon,u}$ are the most challenging and will be handled similarly. In particular, we apply the mean value inequality for Fr\'echet differentials along with the Schauder estimates \eqref{Schaudery} and \eqref{Schauderbar} to obtain  temporal equicontinuity and spatial Sobolev regularity estimates. This is done in the following two lemmas. Note that extra care is required in the choice of H\"older exponents, due to the fact that \eqref{Schaudery} introduces singular coefficients in $\epsilon$  (see the comment preceding the proof of Proposition \ref{Schauderyprop}).

\begin{lem}\label{IV1bnds} Let $T<\infty$ $a>0,$ $x_0,y_0\in H^{a}(0,L)$ and $IV^{\epsilon,u}_1$ as in \eqref{Itoeta3}. There exist $\epsilon_0>0$, $\theta<\frac{1}{2}\wedge a$, $\beta<\frac{1}{4}\wedge\frac{a}{2}$ and a constant $C>0$, independent of $\epsilon$, such that
	\begin{equation}\label{IV1Sob}
	\begin{aligned}
	&	\sup_{\epsilon<\epsilon_0, u\in\mathcal{P}^T_N }\sup_{n\in\N}\ex\bigg(\sup_{t\in[0, T] }\sup_{\chi\in B_\h}\big|IV_1^{\epsilon,u}(0,t,n,\theta,\chi)\big|^2\bigg)\leq C\big( 1+\|x_0\|^2_{H^a}+\|y_0\|^2_{H^a}\big)
	\end{aligned}
	\end{equation}	
	and
	\begin{equation}\label{IV1holder}
	\begin{aligned}
	&	\sup_{\epsilon<\epsilon_0, u\in\mathcal{P}^T_N}\sup_{n\in\N}\ex\bigg(\sup_{\overset{s,t\in[0, T]}{ t\neq s}}\sup_{\chi\in B_\h}\frac{\big|IV_1^{\epsilon,u}(s,t,n,0,\chi)\big|}{  |t-s|^\beta  }\bigg)\leq C\big( 1+\|x_0\|_{H^a}+\|y_0\|_{H^a}\big).
	\end{aligned}
	\end{equation}
\end{lem}
\begin{proof}
	Let $\chi\in Dom((-A_1)^{1+\frac{\theta}{2}})$, $x_1, x_2,\psi\in\h$ and  $y_1, y_2\in Dom(A_2)$. Recall from  \eqref{Riesz} that
	\begin{equation*}
	\langle\Psi^\epsilon(x_1, y_1)-\Psi^\epsilon(x_2, y_2),\psi\rangle_\h=\Phi^\epsilon_\psi(x_1,y_1)-\Phi^\epsilon_\psi(x_2, y_2).
	\end{equation*}
	An application of the mean value inequality for Fr\'{e}chet derivatives then yields
	\begin{equation*}
	\begin{aligned}
	\big|	\langle\Psi^\epsilon(x_1,y_1)-\Psi^\epsilon(x_2,y_2),\psi\rangle_\h\big|&\leq \sup_{x,y\in\h} \|D_x\Phi^\epsilon_\psi(x,y)\|_\h\|x_1-x_2\|_\h\|\psi\|_\h \\&
	+\sup_{x,y\in\h} \|D_y\Phi^\epsilon_\psi(x,y)\|_\h\|y_1-y_2\|_\h\|\psi\|_\h\;.
	\end{aligned}
	\end{equation*}
	\noindent In view of estimates \eqref{phiestimates},
	\begin{equation}\label{meanval}
	\begin{aligned}
	\big|	\langle\Psi^\epsilon(x_1,y_1)-\Psi^\epsilon(x_2,y_2),\psi\rangle_\h\big|&\leq C \bigg(\frac{1}{c(\epsilon)}\|x_1-x_2\|_\h + \|y_1-y_2\|_\h\bigg)\|\psi\|_\h\;.
	\end{aligned}
	\end{equation}
	Using the latter, along with the self-adjointness of $A_1$ and the analyticity of $S_1$
	\begin{equation*}
	\begin{aligned}
	\frac{\delta}{  \sqrt{\epsilon}h(\epsilon)    }\big|&       \langle \Psi^\epsilon\big(\bar{X}(t),Y_n^{\epsilon,u}(t)\big)-\Psi^\epsilon\big(\bar{X}(s),Y_n^{\epsilon,u}(s)\big), S_1(t-s)(-A_1)^{\frac{\theta}{2}}\chi\rangle_\h\big|\\&\leq
	\frac{C\delta}{  \sqrt{\epsilon}h(\epsilon)    }\big\|    (-A_1)^{\frac{\theta}{2}} S_1(t-s)\big[\Psi^\epsilon\big(\bar{X}(t),Y_n^{\epsilon,u}(t)\big)-\Psi^\epsilon\big(\bar{X}(s),Y_n^{\epsilon,u}(s)\big)\big]\big\|_\h \|\chi\|_\h\\&
	\leq	\frac{C\delta}{  \sqrt{\epsilon}h(\epsilon)    }\|\chi\|_\h(t-s)^{-\theta/2}\big\|\Psi^\epsilon\big(\bar{X}(t),Y_n^{\epsilon,u}(t)\big)-\Psi^\epsilon\big(\bar{X}(s),Y_n^{\epsilon,u}(s)\big)\big\|_\h
	\\&\leq C\|\chi\|_\h(t-s)^{-\theta/2}\bigg(\frac{\delta}{  c(\epsilon)\sqrt{\epsilon}h(\epsilon) }\big\|
	\bar{X}(t)-\bar{X}(s)\big\|_\h+\frac{\delta}{  \sqrt{\epsilon}h(\epsilon)    }\big\|Y_n^{\epsilon,u}(t)-Y_n^{\epsilon,u}(s)\big\|_\h\bigg).
	\end{aligned}
	\end{equation*}
 In view of the Schauder estimates \eqref{Schauderbar} and \eqref{Schaudery}, $\bar{X}$ and $Y^{\epsilon,u}$ have finite H\"older seminorms with probability $1$ and
	\begin{equation*}
	\begin{aligned}
	\frac{\delta}{  \sqrt{\epsilon}h(\epsilon)    }&\big|       \langle \Psi^\epsilon\big(\bar{X}(t),Y_n^{\epsilon,u}(t)\big)-\Psi^\epsilon\big(\bar{X}(s),Y_n^{\epsilon,u}(s)\big), S_1(t-s)(-A_1)^{\frac{\theta}{2}}\chi\rangle_\h\big|
	\\&\leq C \|\chi\|_\h(t-s)^{-\theta/2}\bigg(\frac{\delta}{  c(\epsilon)\sqrt{\epsilon}h(\epsilon) }\big[ \bar{X}  \big]_{C^{\theta_1}([0,T];\h)}(t-s)^{\theta_1} +\frac{\delta}{  \sqrt{\epsilon}h(\epsilon) }\big[ Y^{\epsilon,u}  \big]_{C^{\theta_2}([0,T];\h)}(t-s)^{\theta_2}\bigg),
	\end{aligned}
	\end{equation*}
	\noindent where $\theta_1,\theta_2<\frac{1}{4}\wedge\frac{a}{2}$. By the density argument used in the proof of Lemma \ref{Ibnds}, this estimate holds for any $\chi\in\h$. Letting $\theta'=\theta_1\wedge\theta_2$ and $\chi\in B_\h$
	\begin{equation}\label{IV1prebnd}
	\begin{aligned}
	\frac{\delta}{  \sqrt{\epsilon}h(\epsilon)    }&\big|       \langle \Psi^\epsilon\big(\bar{X}(t),Y_n^{\epsilon,u}(t)\big)-\Psi^\epsilon\big(\bar{X}(s),Y_n^{\epsilon,u}(s)\big), S_1(t-s)A_1^\theta\chi\rangle_\h\big|
	\\&\leq C_T(t-s)^{\theta'-\theta/2}\bigg(\frac{\delta}{  c(\epsilon)\sqrt{\epsilon}h(\epsilon) }\big[ \bar{X}  \big]_{C^{\theta_1}([0,T];\h)} +\frac{\delta}{  \sqrt{\epsilon}h(\epsilon) }\big[ Y^{\epsilon,u}  \big]_{C^{\theta_2}([0,T];\h)}\bigg).
	\end{aligned}
	\end{equation}
	\noindent Setting $s=0$ and taking $\theta<2\theta'<(1/2)\wedge a$ we get
	\begin{equation*}
	\begin{aligned}
	&	\ex\bigg(\sup_{t\in[0, T] }\sup_{\chi\in B_\h}\big|IV_1^{\epsilon,u}(0,t,n,\theta,\chi)\big|^2\bigg)
	\leq C_T\bigg(\frac{\delta^2}{  c^2(\epsilon)\epsilon h^2(\epsilon) }\big[ \bar{X}  \big]^2_{C^{\theta_1}([0,T];\h)} +\frac{\delta^2}{  \epsilon h^2(\epsilon) }\ex\big[ Y^{\epsilon,u}  \big]^2_{C^{\theta_2}([0,T];\h)}\bigg).
	\end{aligned}
	\end{equation*}
	Next, note that the Schauder estimates \eqref{Schauderbar} and \eqref{Schaudery} can be easily seen to hold in $L^2(\Omega)$. In view of this we obtain
	\begin{equation*}
	\begin{aligned}
	\ex\bigg(\sup_{t\in[0, T] }\sup_{\chi\in B_\h}\big|IV_1^{\epsilon,u}(0,t,n,\theta,\chi)\big|^2\bigg)
	&\leq  \frac{C\delta^2}{  c^2(\epsilon)\epsilon h^2(\epsilon) }(1+\|x_0\|^2_{H^a}) \\&+\frac{C\delta^2}{  \epsilon h^2(\epsilon) }h^2(\epsilon)\delta^{-1\vee a} \big(1+\|x_0\|^2_\h+\|y_0\|^2_{H^a}\big).
	\end{aligned}
	\end{equation*}
	Since $c(\epsilon)=\sqrt{\epsilon}$ and the inclusion $H^a(0,L)\subset \h$ is continuous, we can choose $a<1$ to obtain
	\begin{equation*}
	\begin{aligned}
	&	\ex\bigg(\sup_{t\in[0, T] }\sup_{\chi\in B_\h}\big|IV_1^{\epsilon,u}(0,t,n,\theta,\chi)\big|^2\bigg)\leq \frac{C\delta^2}{  \epsilon^2 h^2(\epsilon) }\big(1+\|x_0\|^2_{H^a}\big) +\frac{C\delta}{\epsilon} \big(1+\|x_0\|^2_{H^a}+\|y_0\|^2_{H^a}\big).
	\end{aligned}
	\end{equation*}
	In view of \eqref{Regimes}, the coefficients
	\begin{equation*}
	\frac{\delta^2}{  \epsilon^2 h^2(\epsilon) }, \frac{\delta}{\epsilon}
	\end{equation*}
	are bounded in both Regimes $1$ and $2$, for $\epsilon$ sufficiently small  and \eqref{IV1Sob} follows.
	
\noindent	It remains to prove \eqref{IV1holder}. Setting $\theta=0$ in \eqref{IV1prebnd} we deduce that for any $\beta\leq\theta'$
	\begin{equation*}
	\begin{aligned}
	&	\ex\bigg(\sup_{\overset{s,t\in[0, T]}{ t\neq s}}\sup_{\chi\in B_\h}\frac{\big|IV_1^{\epsilon,u}(s,t,n,0,\chi)\big|}{  |t-s|^{\beta}  }\bigg)\leq \frac{C\delta}{  c(\epsilon)\sqrt{\epsilon}h(\epsilon) }\big[ \bar{X}  \big]_{C^{\theta_1}([0,T];\h)} +\frac{C\delta}{  \sqrt{\epsilon}h(\epsilon) }\ex\big[ Y^{\epsilon,u}  \big]_{C^{\theta_2}([0,T];\h)}
	\end{aligned}
	\end{equation*}
	and the estimate follows from the same argument.	\end{proof}

\begin{lem}\label{IV2bnds} Let $T<\infty$, $a>0$, $x_0,y_0\in H^{a}(0,L)$ and $IV^{\epsilon,u}_2$ as in \eqref{Itoeta3}. There exist $\epsilon_0>0$, $\theta<\frac{1}{2}\wedge a$, $\beta<\frac{1}{4}\wedge\frac{a}{2}$ and a constant $C>0$, independent of $\epsilon$, such that
	\begin{equation}\label{IV2Sob}
	\begin{aligned}
	&	\sup_{\epsilon<\epsilon_0, u\in\mathcal{P}^T_N}\sup_{n\in\N}\ex\bigg(\sup_{t\in[0, T] }\sup_{\chi\in B_\h}\big|IV_2^{\epsilon,u}(0,t,n,\theta,\chi)\big|^2\bigg)\leq C\big( 1+\|x_0\|^2_{H^a}+\|y_0\|^2_{H^a}\big)
	\end{aligned}
	\end{equation}	
	and
	\begin{equation}\label{IV2holder}
	\begin{aligned}
	&		\sup_{\epsilon<\epsilon_0, u\in\mathcal{P}^T_N}\sup_{n\in\N}\ex\bigg(\sup_{\overset{s,t\in[0, T]}{ t\neq s}}\sup_{\chi\in B_\h}\frac{\big|IV_2^{\epsilon,u}(s,t,n,0,\chi)\big|}{  |t-s|^\beta  }\bigg)\leq C\big( 1+\|x_0\|_{H^a}+\|y_0\|_{H^a}\big).
	\end{aligned}
	\end{equation}
\end{lem}

\begin{proof} 	Let $\chi\in Dom((-A_1)^{1+\frac{\theta}{2}})$. From the analyticity of $S_1$ along with \eqref{meanval}
	\begin{equation*}
	\begin{aligned}
	\big|&IV_2^{\epsilon,u}(s,t,n,\theta,\chi)\big|\\&\leq
	\frac{\delta}{  \sqrt{\epsilon}h(\epsilon)    }\|\chi\|_\h \int_{s}^{t}\big\| (-A_1)^{1+\frac{\theta}{2}}S_1(t-z)\big[ \Psi^\epsilon\big(\bar{X}(z),Y_n^{\epsilon,u}(z)\big)-\Psi^\epsilon\big(\bar{X}(t),Y_n^{\epsilon,u}(t)\big)\big]\big\|_\h dz\\&
	\leq  \frac{C\delta}{  \sqrt{\epsilon}h(\epsilon)    }\|\chi\|_\h \int_{s}^{t}(t-z)^{-1-\theta/2}\big\| \Psi^\epsilon\big(\bar{X}(z),Y_n^{\epsilon,u}(z)\big)-\Psi^\epsilon\big(\bar{X}(t),Y_n^{\epsilon,u}(t)\big)\big\|_\h dz\\&
	\leq C\|\chi\|_\h \int_{s}^{t}(t-z)^{-1-\theta/2}\bigg(\frac{\delta}{  c(\epsilon)\sqrt{\epsilon}h(\epsilon) }\big[ \bar{X}  \big]_{C^{\theta_1}([0,T];\h)}(t-z)^{\theta_1} +\frac{\delta}{  \sqrt{\epsilon}h(\epsilon) }\big[ Y^{\epsilon,u}  \big]_{C^{\theta_2}([0,T];\h)}(t-z)^{\theta_2}\bigg) dz.
	\end{aligned}
	\end{equation*}
	\noindent As in the proof of Lemma \ref{IV1bnds}, this estimate can be shown to hold for all $\chi\in B_\h$ and, letting $\theta'=\theta_1\wedge\theta_2$,
	\begin{equation}\label{IV2prebnd}
	\hspace*{-0.2cm}
	\small
	\begin{aligned}
	\big|IV_2^{\epsilon,u}&(s,t,n,\theta,\chi)\big|
	\leq C \bigg(\frac{\delta}{  c(\epsilon)\sqrt{\epsilon}h(\epsilon) }\big[ \bar{X}  \big]_{C^{\theta_1}([0,T];\h)} +\frac{\delta}{  \sqrt{\epsilon}h(\epsilon) }\big[ Y^{\epsilon,u}  \big]_{C^{\theta_2}([0,T];\h)}\bigg) \int_{s}^{t}(t-z)^{-1+\theta'-\theta/2}dz.
	\end{aligned}
	\end{equation}
	\noindent Thus, for $s=0$ and $\theta<2\theta'$
	\begin{equation*}
	\begin{aligned}
	\big|IV_2^{\epsilon,u}&(0,t,n,\theta,\chi)\big|
	\leq CT^{\theta'-\theta/2} \bigg(\frac{\delta}{  c(\epsilon)\sqrt{\epsilon}h(\epsilon) }\big[ \bar{X}  \big]_{C^{\theta_1}([0,T];\h)} +\frac{\delta}{  \sqrt{\epsilon}h(\epsilon) }\big[ Y^{\epsilon,u}  \big]_{C^{\theta_2}([0,T];\h)}\bigg)
	\end{aligned}
	\end{equation*}
	and \eqref{IV2Sob} follows using the same argument as in the proof of \eqref{IV1Sob}. Finally, letting $\theta=0$ in \eqref{IV2prebnd} and taking $\beta<\theta'$, we obtain \eqref{IV2holder}.
\end{proof}
\noindent Next, we estimate the term $IV_3^{\epsilon,u}$ in \eqref{Itoeta3}. The main ingredients of the proof are the spatial regularity estimate \eqref{deribar} along with the continuity of the averaged operator $\bar{F}$ (see Lemma \ref{Fbarlip}).

\begin{lem}\label{IV3bnds} Let $T<\infty$,  $a>0$, $x_0\in H^{a}(0,L)$ 	and $IV^{\epsilon,u}_3$ as in \eqref{Itoeta3}. There exist $\epsilon_0>0$, $\theta< a$, $\beta\leq\frac{a}{2}$ and a constant $C>0$, independent of $\epsilon$, such that
	
	\begin{equation}\label{IV3Sob}
	\begin{aligned}
	&	\sup_{\epsilon<\epsilon_0, u\in\mathcal{P}^T_N}\sup_{n\in\N}\ex\bigg(\sup_{t\in[0, T] }\sup_{\chi\in B_\h}\big|IV_3^{\epsilon,u}(0,t,n,\theta,\chi)\big|^2\bigg)\leq C\big( 1+\|x_0\|^2_{H^a}\big)
	\end{aligned}
	\end{equation}	
	and
	\begin{equation}\label{IV3holder}
	\begin{aligned}
	&		\sup_{\epsilon<\epsilon_0, u\in\mathcal{P}^T_N}\sup_{n\in\N}\ex\bigg(\sup_{\overset{s,t\in[0, T]}{ t\neq s}}\sup_{\chi\in B_\h}\frac{\big|IV_3^{\epsilon,u}(s,t,n,0,\chi)\big|}{  |t-s|^\beta  }\bigg)\leq C\big( 1+\|x_0\|_{H^a}\big).
	\end{aligned}
	\end{equation}
\end{lem}

\begin{proof} 	Let $\chi\in Dom((-A_1)^{1+\frac{\theta}{2}})$. Using the analyticity of $S_1$ along with the first estimate in \eqref{psinorm}
	\begin{equation*}
	\begin{aligned}
	\small
	\big|IV_3^{\epsilon,u}(s,t,n,\theta,\chi)\big|&\leq \frac{\delta}{  \sqrt{\epsilon}h(\epsilon)    }\|\chi\|_\h\int_{s}^{t}    \big\| S_1(t-z)(-A_1)^{\frac{\theta}{2}}\Psi^\epsilon_1\big(\bar{X}(z),Y_n^{\epsilon,u}(z)\big)\big[ A_1 \bar{X}(z)+\bar{F}\big(\bar{X}(z)\big)\big]\big\|_\h dz
	\\&\leq \frac{C\delta}{  c(\epsilon)\sqrt{\epsilon}h(\epsilon)    }\|\chi\|_\h\int_{s}^{t} (t-z)^{-\theta/2}\bigg( \big\| A_1 \bar{X}(z)\big\|_\h+\big\|\bar{F}\big(\bar{X}(z)\big)\big\|_\h\bigg)dz,
	\end{aligned}
	\end{equation*}
 with probability $1$.
	As in the proof of Lemma \ref{IV1bnds}, a density argument allows us to choose $\chi\in B_\h$ and apply \eqref{deribar} to deduce that
	\begin{equation}\label{IV3prebnd}
	\begin{aligned}
	\big|IV_3^{\epsilon,u}&(s,t,n,\theta,\chi)\big|
	\leq \frac{C\delta}{  c(\epsilon)\sqrt{\epsilon}h(\epsilon)    }\int_{s}^{t} (t-z)^{-\theta/2}\big[(z^{-1+a/2}+1)\|x_0 \|_{H^a}+\big\|\bar{F}\big(\bar{X}(z)\big)\big\|_\h\big]dz.
	\end{aligned}
	\end{equation}
	\noindent Setting $s=0$ and choosing $p$ large enough to satisfy $\theta<2/p<a$, we apply H\"older's inequality to obtain
	\begin{equation*}
	\begin{aligned}
	\big|IV_3^{\epsilon,u}&(0,t,n,\theta,\chi)\big|
	\leq  \frac{C T^{1/p-\theta/2}   \delta}{  c(\epsilon)\sqrt{\epsilon}h(\epsilon)    }    \bigg\{\|x_0\|_{H^a}\bigg[\int_{0}^{T}  (z^{\frac{a}{2}-1}+1)^q  dz \bigg]^{\frac{1}{q}} +T^\frac{1}{q}\sup_{z\in[0,T]}\big\|\bar{F}\big(\bar{X}(z)\big)\big\|_\h\bigg\}.
	\end{aligned}
	\end{equation*}
	\noindent From the Lipschitz continuity of $\bar{F}$ and the fact that $c(\epsilon)=\sqrt\epsilon$ (see \eqref{cchoice}) we have
	\begin{equation*}
	\begin{aligned}
	\big|IV_3^{\epsilon,u}&(0,t,n,\theta,\chi)\big|
	\leq \frac{C_{T,\theta,p}\delta}{ \epsilon h(\epsilon)    }  \big(1+\|x_0\|_{H^a}\big).
	\end{aligned}
	\end{equation*}
	This proves \eqref{IV3Sob} since
	$ \delta/ (\epsilon h(\epsilon) ) $ is bounded for $\epsilon$ small enough.
	As for \eqref{IV3holder}, let $\theta=0$ and $c(\epsilon)=\sqrt{\epsilon}$ in \eqref{IV3prebnd} to obtain
	\begin{equation*}
	\begin{aligned}
	\big|IV_3^{\epsilon,u}&(s,t,n,0,\chi)\big|
	\leq\frac{C\delta}{  \epsilon h(\epsilon)    }\int_{s}^{t} \big[(z^{-1+a/2}+1)\|x_0 \|_{H^a}+\big\|\bar{F}\big(\bar{X}(z)\big)\big\|_\h\big]dz.
	\end{aligned}
	\end{equation*}
	In view of the Lipschitz continuity of $\bar{F}$, the proof is complete.	
\end{proof}
The following two lemmas provide estimates for the terms $IV_k^{\epsilon,u}$, $k=4,5$ in \eqref{Itoeta3}.  These estimates do not require regularity of initial conditions and in fact are straightforward consequences of the analyticity of $S_1$ and the a priori bounds \eqref{xbarapriori} and \eqref{yapriori} from Section \ref{Sec2}.

\begin{lem}\label{IV4bnds} Let $T<\infty,$ $x_0,y_0\in \h$ and $IV^{\epsilon,u}_4$ as in \eqref{Itoeta3}. There exist $\epsilon_0>0$ and a constant $C>0$, independent of $\epsilon$, such that for all $\theta<1/2$ and $\beta\leq1/2$
		\begin{equation}\label{IV4Sob}
	\begin{aligned}
	&	\sup_{\epsilon<\epsilon_0, u\in\mathcal{P}^T_N}\sup_{n\in\N}\ex\bigg(\sup_{t\in[0, T] }\sup_{\chi\in B_\h}\big|IV_4^{\epsilon,u}(0,t,n,\theta,\chi)\big|^2\bigg)\leq C\big(1+\|x_0\|^2_{\h}+\|y_0\|^2_{\h} \big)
	\end{aligned}
	\end{equation}	
	and
	\begin{equation}\label{IV4holder}
	\begin{aligned}
	&		\sup_{\epsilon<\epsilon_0, u\in\mathcal{P}^T_N}\sup_{n\in\N}\ex\bigg(\sup_{\overset{s,t\in[0, T]}{ t\neq s}}\sup_{\chi\in B_\h}\frac{\big|IV_4^{\epsilon,u}(s,t,n,0,\chi)\big|}{  |t-s|^\beta  }\bigg)\leq C\big( 1+\|x_0\|_{\h}+\|y_0\|_{\h}\big).
	\end{aligned}
	\end{equation}
\end{lem}
\begin{proof} Let $\chi\in Dom((-A_1)^{1+\frac{\theta}{2}})$ . Using the analyticity of $S_1$ along with \eqref{psinorm} we obtain
	\begin{equation*}
	\begin{aligned}
	\big|IV_4^{\epsilon,u}(s,t,n,\theta,\chi)\big|&\leq\frac{c(\epsilon)}{  \sqrt{\epsilon}h(\epsilon)    }\int_{s}^{t}\big\|S_1(t-z)(-A_1)^{\frac{\theta}{2}} \Psi^\epsilon\big(\bar{X}(z), Y_n^{\epsilon,u}(z)\big)\big\|_\h \|\chi\|_\h dz	\\&
	\leq \frac{Cc(\epsilon)}{  \sqrt{\epsilon}h(\epsilon)    }\|\chi\|_\h \int_{s}^{t}(t-z)^{-\theta/2}\bigg(1+\big\|\bar{X}(z)\big\|_\h+\big\| Y_n^{\epsilon,u}(z)\big\|_\h \bigg)dz.
	\end{aligned}
	\end{equation*}
	\noindent Since $\theta<1/2$, the Cauchy-Schwarz inequality yields
	\begin{equation}\label{IV4prebnd}
	\begin{aligned}
	\big|IV_4^{\epsilon,u}(s,t,n,\theta,\chi)\big|&
	\leq  \frac{Cc(\epsilon)}{  \sqrt{\epsilon}h(\epsilon)    }\|\chi\|_\h (t-s)^{1/2-\theta/2}\bigg(\int_{0}^{T}\big[1+\big\|\bar{X}(z)\big\|^2_\h+\big\| Y_n^{\epsilon,u}(z)\big\|^2_\h \big]dz\bigg)^{1/2}.
	\end{aligned}
	\end{equation}
	
	\noindent As in the proof of Lemma \ref{IV1bnds} we can use a density argument to show that the last estimate holds for all $\chi\in\h$. Setting $s=0$ and taking expectation, we apply Jensen's inequality along with \eqref{xbarapriori} and \eqref{yapriori}   to obtain
	\begin{equation*}
	\begin{aligned}
	\ex\sup_{t\in[0,T]}\sup_{\chi\in B_\h}\big|IV_4^{\epsilon,u}(0,t,n,\theta,\chi)\big|^2&
	\leq \frac{Cc^2(\epsilon)}{  \epsilon h^2(\epsilon)    } \int_{0}^{T}\big[1+\big\|\bar{X}(z)\big\|^2_\h+\ex\big\| Y_n^{\epsilon,u}(z)\big\|^2_\h \big]dz\\&
	\leq \frac{Cc^2(\epsilon)}{  \epsilon h^2(\epsilon)    }(1+\|x_0\|^2_\h+\|y_0\|^2_\h).
	\end{aligned}
	\end{equation*}
	This completes the proof of \eqref{IV4Sob} since
	\begin{equation*}
	\frac{c^2(\epsilon)}{  \epsilon h^2(\epsilon)    } =\frac{1}{h^2(\epsilon)}\longrightarrow 0\;\;\text{as}\;\;\epsilon\to 0.
	\end{equation*}
	As for \eqref{IV4holder}, we set $\theta=0$ in \eqref{IV4prebnd} to conclude that
	\begin{equation*}
	\begin{aligned}
	\ex\sup_{s\neq t\in[0,T]}\sup_{\chi\in B_\h}\frac{\big|IV_4^{\epsilon,u}(s,t,n,0,\chi)\big|}{|t-s|^{\frac{1}{2}}}&
	\leq \frac{C}{  h(\epsilon)    } \bigg(\int_{0}^{T}\big[1+\big\|\bar{X}(z)\big\|^2_\h+\ex\big\| Y_n^{\epsilon,u}(z)\big\|^2_\h \big]dz\bigg)^{1/2}\\&
	\leq  \frac{C}{  h(\epsilon)    }(1+\|x_0\|^2_\h+\|y_0\|^2_\h)^{1/2},
	\end{aligned}
	\end{equation*}
	\noindent for $\epsilon$ sufficiently small.
\end{proof}

\begin{lem}\label{IV5bnds} Let $T<\infty$, $x_0,y_0\in\h$ and $IV^{\epsilon,u}_5$ as in \eqref{Itoeta3}. There exist $\epsilon_0>0$ and a constant $C>0$, independent of $\epsilon$, such that for all $\theta<1/2$ and $\beta\leq1/2$
	
	\begin{equation}\label{IV5Sob}
	\begin{aligned}
	&	\sup_{\epsilon<\epsilon_0, u\in\mathcal{P}^T_N}\sup_{n\in\N}\ex\bigg(\sup_{t\in[0, T] }\sup_{\chi\in B_\h}\big|IV_5^{\epsilon,u}(0,t,n,\theta,\chi)\big|^2\bigg)\leq C
	\end{aligned}
	\end{equation}	
	and
	\begin{equation}\label{IV5holder}
	\begin{aligned}
	&		\sup_{\epsilon<\epsilon_0, u\in\mathcal{P}^T_N}\sup_{n\in\N}\ex\bigg(\sup_{\overset{s,t\in[0, T]}{ t\neq s}}\sup_{\chi\in B_\h}\frac{\big|IV_5^{\epsilon,u}(s,t,n,0,\chi)\big|}{  |t-s|^\beta  }\bigg)\leq C.
	\end{aligned}
	\end{equation}
\end{lem}
\begin{proof} Let $\chi\in Dom((-A_1)^{1+\frac{\theta}{2}})$ . Using the analyticity of $S_1$ along with the second estimate in \eqref{psinorm} we have that, with probability $1$,
	\begin{equation}\label{IV5prebnd}
	\begin{aligned}
	\big|IV_5^{\epsilon,u}(s,t,n,\theta,\chi)\big|&\leq
	\frac{ C\sqrt{\delta}} {\sqrt{\epsilon}  } \|\chi \|_\h  \int_{s}^{t} (t-z)^{-\theta/2}\big\| \Psi^\epsilon_2\big(\bar{X}(z),Y_n^{\epsilon,u}(z)\big)\big\|_{\mathscr{L}(\h)}\|u_{2,n}(z)\|_\h dz
	\\&
	\leq \frac{C\sqrt{\delta}} {\sqrt{\epsilon}  }\|\chi \|_\h (t-s)^{1/2-\theta/2}\|u_2\|_{L^2([0,T];\h)}\leq \frac{CN\sqrt{\delta}} {\sqrt{\epsilon}  }\|\chi \|_\h (t-s)^{1/2-\theta/2},
	\end{aligned}
	\end{equation}
	\noindent where we applied the Cauchy-Schwarz inequality and the fact that $u_2\in\mathcal{P}^T_N$ to obtain the last line. From a density argument (see proof of Lemma \ref{IV1bnds}), the last estimate holds for all $\chi\in\h$. In view of \eqref{Regimes}, $\sqrt{\delta}/\sqrt{\epsilon}$ is bounded in both Regimes $1,2$, for $\epsilon$ sufficiently small. Thus we set $s=0$ in \eqref{IV5prebnd} to obtain \eqref{IV5Sob} and $\theta=0$ to obtain \eqref{IV5holder}. 	\end{proof}
\noindent Next, we bound the stochastic convolution term $IV_6^{\epsilon,u}$. The estimates rely on the stochastic factorization formula and, to avoid repetition, many of the arguments will be omitted.

\begin{lem}\label{IV6bnds}  Let $T<\infty$ and $IV_6^{\epsilon,u}$ as in \eqref{Itoeta3}. There exist $\epsilon_0>0$ and a constant $C>0$, independent of $\epsilon$, such that for all $\theta<\frac{1}{2}$ and $\beta<\frac{1}{4}$
	
	\begin{equation}\label{IV6Sob}
	\begin{aligned}
	&	\sup_{\epsilon<\epsilon_0, u\in\mathcal{P}^T_N}\sup_{n\in\N}\ex\bigg(\sup_{t\in[0, T] }\sup_{\chi\in B_\h}\big|IV_6^{\epsilon,u}(0,t,n,\theta,\chi)\big|^2\bigg)\leq C
	\end{aligned}
	\end{equation}	
	and
	\begin{equation}\label{IV6holder}
	\begin{aligned}
	&		\sup_{\epsilon<\epsilon_0, u\in\mathcal{P}^T_N}\sup_{n\in\N}\ex\bigg(\sup_{\overset{s,t\in[0, T]}{ t\neq s}}\sup_{\chi\in B_\h}\frac{\big|IV_6^{\epsilon,u}(s,t,n,0,\chi)\big|}{  |t-s|^\beta  }\bigg)\leq C.
	\end{aligned}
	\end{equation}
\end{lem}
\begin{proof} 	\noindent Let $\chi\in Dom((-A_1)^{1+\frac{\theta}{2}})$ and apply the  stochastic factorization formula (see \eqref{stofact}) to obtain
	\begin{equation}\label{IV6fin}
	IV_6^{\epsilon,u}(s,t,n,\theta,\chi)	=  \frac{\sqrt{\delta}\sin(a\pi)}{\sqrt{\epsilon}h(\epsilon)\pi}\bigg\langle  \int_{s}^{t}(t-z)^{a-1}(-A_1)^{\frac{\theta}{2}}S_1(t-z)M^{n,\epsilon,u}_a(s,z,z)dz\;,\chi\bigg\rangle_\h\;,
	\end{equation}
	where,
	\begin{equation}\label{stofactoreg2}
	M^{n,\epsilon,u}_a(t_1,t_2,t_3;1)=\int_{t_1}^{t_2}(t_3-\zeta)^{-a}S_1(t_3-\zeta)\\\Psi^\epsilon_2 \big(\bar{X}(\zeta), Y_n^{\epsilon,u}(\zeta)\big)P_ndw_2(\zeta)
	\end{equation}
	and $P_n$ is an orthogonal projection on an $n$-dimensional eigenspace of $A_2$. It follows that
	\begin{equation}\label{IV6prebnd}
	\begin{aligned}
	&\big|	IV_6^{\epsilon,u}(s,t,n,\theta,\chi)\big|\leq\frac{C\sqrt{\delta}}{\sqrt{\epsilon}h(\epsilon)}\|\chi\|_\h \int_{s}^{t}(t-z)^{a-1}\big\|(-A_1)^{\frac{\theta}{2}}	M^{n,\epsilon,u}_a(s,z,z;1)\big\|_\h dz.
	\end{aligned}
	\end{equation}
	From a density argument (see proof of Lemma \ref{Ibnds}), the last estimate holds with probability $1$ for all $\chi\in B_\h$.

	Due to the similarity of the estimates with those in Lemma \ref{IIIbnds}, we will only prove \eqref{IV6holder}. To this end, set $\theta=0$ in \eqref{IV6prebnd} and let $q>1/a>2$. Repeating the arguments of Lemma \ref{IIIbnds} we see that
		\begin{equation*}
	\begin{aligned}
	\ex \sup_{\overset{s,t\in[0, T]}{ t\neq s}}&\sup_{\chi\in B_\h}\frac{\big|IV_6^{\epsilon,u}(s,t,n,0,\chi)\big|}{|t-s|^{a-1/q}}\\&\leq
	\frac{C\sqrt{\delta}}{\sqrt{\epsilon}h(\epsilon)} \bigg( \int_{0}^{T}\bigg(   \int_{0}^{z}(z-\zeta)^{-2a}\ex \big\|S_1(z-\zeta)\Psi^\epsilon_2 \big(\bar{X}(\zeta), Y_n^{\epsilon,u}(\zeta)\big)P_n\big\|^2_{\mathscr{L}_2(\h)}d\zeta\bigg)^{\frac{q}{2}} dz\bigg)^{\frac{1}{q}}.
	\end{aligned}
	\end{equation*}
	\noindent Invoking Lemma \ref{sigmacont}(ii) (with  $B(\zeta)=\Psi^\epsilon_2 (\bar{X}(\zeta), Y_n^{\epsilon,u}(\zeta))$ ) along with the first estimate in \eqref{psinorm}, we can choose $a<\frac{1}{4}$ and $\frac{1}{2}<\rho<1-2a$  so that
	\begin{equation*}
	\begin{aligned}
	\ex \sup_{\overset{s,t\in[0, T]}{ t\neq s}}\sup_{\chi\in B_\h}\frac{\big|IV_6^{\epsilon,u}(s,t,n,0,\chi)\big|}{|t-s|^{a-1/q}}&\leq \frac{C\sqrt{\delta}}{\sqrt{\epsilon}h(\epsilon)}\bigg(\int_{0}^{T}\bigg(   \int_{0}^{z}(z-\zeta)^{-2a-\rho}d\zeta\bigg)^{\frac{q}{2}} dz\bigg)^{\frac{1}{q}}\\&\leq\frac{C\sqrt{\delta}}{\sqrt{\epsilon}h(\epsilon)}\bigg(\int_{0}^{T}z^{\frac{q}{2}(1-2a-\rho)}dz\bigg)^{\frac{1}{q}}<\infty.
	\end{aligned}
	\end{equation*}
	\noindent Since  $\sqrt{\delta}/\sqrt{\epsilon}$ is bounded for $\epsilon$ sufficiently small and $h(\epsilon)\to \infty$ as $\epsilon\to 0$, \eqref{IV6holder} follows.
	
	Taking \eqref{IV6fin}, \eqref{stofactoreg2} and \eqref{psinorm} into account, we see that the proof of \eqref{IV6Sob} is nearly identical to that of estimate \eqref{IIISob} and thus will be omitted.
\end{proof}
\noindent	The last remaining step before estimating $IV^{\epsilon,u}$ involves bounding the finite-dimensional approximation error $R^{\epsilon,u}$ in \eqref{Itoeta3}, given by \eqref{Rem}. This term has singular prefactors of order $1/\sqrt{\epsilon}h(\epsilon)$. However, if we fix $\epsilon$ and let $n\to\infty$,  $R^{\epsilon,u}$ vanishes. Thus, for each $\epsilon>0$, we can choose an integer $n(\epsilon)$ that makes $R^{\epsilon,u}$  small. This is done in the following lemma.

\begin{lem}\label{Rbnds} Let $T<\infty$, $\theta<1/2$ and $R^{\epsilon,u}$ as in \eqref{Rem}. For all $\epsilon>0$ there exists $n(\epsilon)\in\N$ such that
\begin{equation}\label{RSob}
\sup_{u\in\mathcal{P}^T_N}\ex\sup_{t\in[0, T]}\sup_{\chi\in B_\h}\big|R^{\epsilon,u}( 0,t,n(\epsilon),\theta,\chi)\big|^2\leq \epsilon
\end{equation}
\noindent and
\begin{equation}\label{Rholder}
\begin{aligned}
\sup_{u\in\mathcal{P}^T_N}\ex\sup_{\overset{s,t\in[0, T]}{ t\neq s}}\sup_{\chi\in B_\h}\frac{\big|R^{\epsilon,u}( s,t,n(\epsilon),0,\chi)\big|}{|t-s|^{1/2}}\leq \epsilon.
\end{aligned}
\end{equation}

\end{lem}
\begin{proof} Let $\chi\in Dom((-A_1)^{\frac{\theta}{2}}), n\in\N$ and recall that
\begin{equation}\label{Remlem}
\begin{aligned}
&R^{\epsilon,u}( s,t,n,\theta,\chi)=\frac{1}{\sqrt{\epsilon}h(\epsilon)} \int_{s}^{t}\blangle F\big(\bar{X}(z), Y^{\epsilon,u}(z) \big)-F\big(\bar{X}(z), Y_n^{\epsilon,u}(z) \big), S_1(t-z)(-A_1)^{\frac{\theta}{2}}\chi\brangle_\h\; dz\\&
+\frac{1}{\sqrt{\epsilon}h(\epsilon)}\int_{s}^{t}   \blangle \Psi^\epsilon_2\big(\bar{X}(z),Y_n^{\epsilon,u}(z)\big)\big[ P_nG\big( \bar{X}(z), Y^{\epsilon,u}(z) \big)- G\big( \bar{X}(z), Y_n^{\epsilon,u}(z) \big)\big], S_1(t-z)(-A_1)^{\frac{\theta}{2}}\chi\brangle_\h dz\\&
+\frac{1}{2\sqrt{\epsilon}h(\epsilon)}  \int_{s}^{t}  \blangle \Psi^{\epsilon,n}_3\big(\bar{X}(z),Y_n^{\epsilon,u}(z)\big), S_1(t-z)(-A_1)^{\frac{\theta}{2}}\chi\brangle_\h dz.
\end{aligned}
\end{equation}
We start by estimating the first term in the last display. Using the analyticity of $S_1$ along with the Lipschitz continuity of $ F$
\begin{equation*}\label{R1prebnd}
\begin{aligned}
\bigg|\frac{1}{\sqrt{\epsilon}h(\epsilon)}& \int_{s}^{t}\blangle F\big(\bar{X}(z), Y^{\epsilon,u}(z) \big)-F\big(\bar{X}(z), Y_n^{\epsilon,u}(z) \big), S_1(t-z)(-A_1)^{\frac{\theta}{2}}\chi\brangle_\h\; dz\bigg|
\\&
\leq\frac{C}{\sqrt{\epsilon}h(\epsilon)}  \|\chi\|_\h
\int_{s}^{t}(t-z)^{-\theta/2} \big\| F\big(\bar{X}(z), Y^{\epsilon,u}(z) \big)-F\big(\bar{X}(z) , Y_n^{\epsilon,u}(z) \big)\big\|_{\h}\;dz
\\&
\leq \frac{C_f}{\sqrt{\epsilon}h(\epsilon)} \|\chi\|_\h
\bigg(\int_{0}^{T}\big\| Y^{\epsilon,u}(z)-Y_n^{\epsilon,u}(z)\big\|^2_{\h}\;dz\bigg)^{1/2}(t-s)^{\frac{1-\theta}{2}},
\end{aligned}
\end{equation*}
where we also applied the Cauchy-Schwarz inequality to obtain the last line. As in the proof of Lemma \ref{Ibnds}, we can use a density argument to deduce that the last estimate holds for all $\chi\in B_\h$. Setting $s=0$
\begin{equation}\label{R1sobprebnd}
\begin{aligned}
\ex\sup_{\chi\in B_\h}&\bigg|\frac{1}{\sqrt{\epsilon}h(\epsilon)} \int_{0}^{t}\blangle F\big(\bar{X}(z), Y^{\epsilon,u}(z) \big)-F\big(\bar{X}(z), Y_n^{\epsilon,u}(z) \big), S_1(t-z)(-A_1)^{\frac{\theta}{2}}\chi\brangle_\h\; dz\bigg|^2\\&\leq \frac{C}{\epsilon h^2(\epsilon)}
\ex\int_{0}^{T}\big\| Y^{\epsilon,u}(z)-Y_n^{\epsilon,u}(z)\big\|^2_{\h}\;dz,
\end{aligned}
\end{equation}
\noindent while for $\theta=0$ we obtain
\begin{equation}\label{R1holderprebnd}
\begin{aligned}
\ex\sup_{\overset{s,t\in[0, T]}{ t\neq s}}\sup_{\chi\in B_\h}\frac{1}{|t-s|^{1/2}}&\bigg|\frac{1}{\sqrt{\epsilon}h(\epsilon)} \int_{s}^{t}\blangle F\big(\bar{X}(z), Y^{\epsilon,u}(z) \big)-F\big(\bar{X}(z), Y_n^{\epsilon,u}(z) \big), S_1(t-z)\chi\brangle_\h\; dz\bigg|\\&\leq \frac{C}{\sqrt{\epsilon}h(\epsilon)}
\bigg(\ex\int_{0}^{T}\big\| Y^{\epsilon,u}(z)-Y_n^{\epsilon,u}(z)\big\|^2_{\h}\;dz\bigg)^{1/2}.
\end{aligned}
\end{equation}

\noindent Next, recall that $Y_n$ solves \eqref{projy} and note that for fixed $\epsilon$ and all $z\in [0,T]$
\begin{equation*}
Y_n^{\epsilon,u}(z) \longrightarrow Y^{\epsilon,u}(z) \;,\text{as}\;\;n\to\infty\;\;\pr-a.s.
\end{equation*}
\noindent Moreover, $$\sup_{n\in\N}\ex\int_{0}^{T}\|  Y_n^{\epsilon,u}(z)- Y^{\epsilon,u}(z) \|^2_\h\leq 2 \ex\|Y^{\epsilon,u} \|^2_{L^2([0,T];\h)}$$
\noindent and the last expression is finite due to \eqref{yapriori}. An application of the Dominated Convergence theorem yields that for each fixed $\epsilon>0$
\begin{equation*}
\begin{aligned}
&  \frac{1}{\sqrt{\epsilon}h(\epsilon)} \lim_{n\to\infty}\bigg(\ex\int_{0}^{T}\big\| Y^{\epsilon,u}(z)-Y_n^{\epsilon,u}(z)\big\|^2_{\h}\;dz\bigg)^{1/2}
=0.
\end{aligned}
\end{equation*}
Combining the latter with \eqref{R1sobprebnd} and \eqref{R1holderprebnd} yields
\begin{equation*}
\begin{aligned}
&\lim_{n\to\infty}\ex\sup_{\overset{s,t\in[0, T]}{ t\neq s}}\sup_{\chi\in B_\h}\frac{1}{|t-s|^{1/2}}\bigg|\frac{1}{\sqrt{\epsilon}h(\epsilon)} \int_{s}^{t}\blangle F\big(\bar{X}(z), Y^{\epsilon,u}(z) \big)-F\big(\bar{X}(z), Y_n^{\epsilon,u}(z) \big), S_1(t-z)\chi\brangle_\h dz\bigg|\\&
=	\lim_{n\to\infty}\ex\sup_{\chi\in B_\h}\bigg|\frac{1}{\sqrt{\epsilon}h(\epsilon)} \int_{0}^{t}\blangle F\big(\bar{X}(z), Y^{\epsilon,u}(z) \big)-F\big(\bar{X}(z), Y_n^{\epsilon,u}(z) \big), S_1(t-z)(-A_1)^{\frac{\theta}{2}}\chi\brangle_\h dz\bigg|=0.
\end{aligned}
\end{equation*}
\noindent Thus, for all $\epsilon>0$ we can find  $n(\epsilon)\in\N$ large enough to satisfy
\begin{equation}\label{R1bnds}
\small
\begin{aligned}
&\ex\sup_{\overset{s,t\in[0, T]}{ t\neq s}}\sup_{\chi\in B_\h}\frac{1}{|t-s|^{1/2}}\bigg|\frac{1}{\sqrt{\epsilon}h(\epsilon)} \int_{s}^{t}\blangle F\big(\bar{X}(z), Y^{\epsilon,u}(z) \big)-F\big(\bar{X}(z), Y_{n(\epsilon)}^{\epsilon,u}(z) \big), S_1(t-z)(-A_1)^{\frac{\theta}{2}}\chi\brangle_\h dz\bigg|
\\&+\ex\sup_{\chi\in B_\h}\bigg|\frac{1}{\sqrt{\epsilon}h(\epsilon)} \int_{0}^{t}\blangle F\big(\bar{X}(z), Y^{\epsilon,u}(z) \big)-F\big(\bar{X}(z), Y_{n(\epsilon)}^{\epsilon,u}(z) \big), S_1(t-z)(-A_1)^{\frac{\theta}{2}}\chi\brangle_\h dz\bigg|\leq \frac{\epsilon}{3}\;.
\end{aligned}
\end{equation}
For the second term in \eqref{Remlem} we can use the first estimate in \eqref{psinorm} along with similar arguments to show that for each $\chi\in B_\h$
\begin{equation*}\label{R2prebnd}
\begin{aligned}
\bigg|\frac{1}{\sqrt{\epsilon}h(\epsilon)}&\int_{s}^{t}   \blangle \Psi^\epsilon_2\big(\bar{X}(z),Y_n^{\epsilon,u}(z)\big)\big[ P_nG\big( \bar{X}(z), Y^{\epsilon,u}(z) \big)- G\big( \bar{X}(z), Y_n^{\epsilon,u}(z) \big)\big], S_1(t-z)(-A_1)^{\frac{\theta}{2}}\chi\brangle_\h dz\bigg|\\&
\leq \frac{C}{\sqrt{\epsilon}h(\epsilon)}
\bigg(\int_{0}^{T} \big\| P_nG\big( \bar{X}(z), Y^{\epsilon,u}(z) \big)- G\big( \bar{X}(z), Y_n^{\epsilon,u}(z) \big)\big\|^2_{\h}\;dz\bigg)^{1/2}(t-s)^{\frac{1-\theta}{2}}.
\end{aligned}
\end{equation*}
Since $G$ is continuous in $y$, for each fixed $\epsilon$ and $z\in[0,T]$,  $$\big\| P_nG\big( \bar{X}(z), Y^{\epsilon,u}(z) \big)- G\big( \bar{X}(z), Y_n^{\epsilon,u}(z) \big)\big\|^2_{\h}\longrightarrow0\;\;,\;\text{as}\;\; n\to\infty\;\;\pr-a.s.$$
\noindent  From the linear growth of $G$ in both variables along with estimates and \eqref{xbarapriori} and \eqref{yapriori} we have
\begin{equation*}
\small
\begin{aligned}
\sup_{n\in\N}\ex\int_{0}^{T}&\big\| P_nG\big( \bar{X}(z), Y^{\epsilon,u}(z) \big)- G\big( \bar{X}(z), Y_n^{\epsilon,u}(z) \big)\big\|^2_{\h}dz\leq C_g\bigg( 1+\sup_{t\in[0,T]}\|\bar{X}(t)\|^2_\h+ \int_{0}^{T}\ex\|Y^{\epsilon,u}(z) \|^2_\h dz\bigg)<\infty.
\end{aligned}
\end{equation*}
Applying a dominated convergence argument as before we can show that, for all $\epsilon>0$, there exists  $n(\epsilon)\in\N$ large enough to satisfy
\begin{equation}\label{R2bnds}
\small
\begin{aligned}
&\ex\sup_{\overset{s,t\in[0, T]}{ t\neq s}}\sup_{\chi\in B_\h}\frac{1}{|t-s|^{1/2}}\bigg|\frac{1}{\sqrt{\epsilon}h(\epsilon)}\int_{s}^{t}   \blangle \Psi^\epsilon_2\big(\bar{X}(z),Y_{n(\epsilon)}^{\epsilon,u}(z)\big)\big[ P_{n(\epsilon)}G\big( \bar{X}(z), Y^{\epsilon,u}(z) \big)\\&- G\big( \bar{X}(z), Y_{n(\epsilon)}^{\epsilon,u}(z) \big)\big], S_1(t-z)\chi\brangle_\h dz\bigg|^2
+\ex\sup_{\chi\in B_\h}\bigg|\frac{1}{\sqrt{\epsilon}h(\epsilon)}\int_{0}^{t}   \blangle \Psi^\epsilon_2\big(\bar{X}(z),Y_{n(\epsilon)}^{\epsilon,u}(z)\big)\big[ P_{n(\epsilon)}G\big( \bar{X}(z), Y^{\epsilon,u}(z) \big)\\&- G\big( \bar{X}(z), Y_{n(\epsilon)}^{\epsilon,u}(z) \big)\big], S_1(t-z)(-A_1)^{\frac{\theta}{2}}\chi\brangle_\h dz\bigg|^2\leq\frac{\epsilon}{3}\;.
\end{aligned}
\end{equation}
\noindent It remains to estimate the last term in \eqref{Remlem}. Since the arguments are very similar to the ones above we will only sketch the proof.
\noindent In view of \eqref{Xitr} and the continuity of $D^2_y\Phi^\epsilon_\chi(x,y)$ in $y$
$$\blangle\Psi^{\epsilon,n}_3\big(\bar{X}(z),Y_n^{\epsilon,u}(z)\big),\chi \brangle_\h=\text{tr}\big[(P_n-I)D^2_y\Phi^\epsilon_\chi\big(\bar{X}(z),Y_n^{\epsilon,u}(z)\big)\big]\longrightarrow 0\;\;\text{as}\;\;n\to\infty$$
and this convergence is uniform over $\chi\in B_\h$.
In view of the estimate in \eqref{phitr}, which is uniform in $n$,
$$\sup_{n\in\N}\ex\int_{0}^{T}\big\|\Psi^{\epsilon,n}_3\big(\epsilon,\bar{X}(z),Y_n^{\epsilon,u}(z)\big)\big\|^2_\h dz\leq \frac{c}{c(\epsilon)}\bigg( 1+\sup_{z\in[0, T] }\|\bar{X}(z)\|^2_\h +\ex\|Y^{\epsilon,u} \|^2_{L^2([0,T];\h)}\bigg)$$
\noindent and for each fixed $\epsilon$ the right-hand is finite due to estimates \eqref{xbarapriori} and \eqref{yapriori}. Using the analyticity of $S_1$ along with the Dominated Convergence theorem as before we deduce that for each $\theta<1/2$ and $\epsilon>0$, there exists $n(\epsilon)\in\N$ large enough to satisfy
\begin{equation}\label{R3bnds}
\begin{aligned}
&\ex\sup_{\overset{s,t\in[0, T]}{ t\neq s}}\sup_{\chi\in B_\h}\frac{1}{|t-s|^{1/2}}\bigg|\frac{1}{2\sqrt{\epsilon}h(\epsilon)}  \int_{s}^{t}  \blangle \Psi^{\epsilon, n(\epsilon)}_3\big(\epsilon,\bar{X}(z),Y_{n(\epsilon)}^{\epsilon,u}(z)\big), S_1(t-z)\chi\brangle_\h dz\bigg|\\&
+\ex\sup_{\chi\in B_\h}\bigg|\frac{1}{2\sqrt{\epsilon}h(\epsilon)}  \int_{0}^{t}  \blangle \Psi^{\epsilon, n(\epsilon)}_3\big(\epsilon,\bar{X}(z),Y_{n(\epsilon)}^{\epsilon,u}(z)\big), S_1(t-z)(-A_1)^{\frac{\theta}{2}}\chi\brangle_\h dz\bigg|^2\leq \frac{\epsilon}{3}\;.
\end{aligned}
\end{equation}
The proof is complete upon combining \eqref{R1bnds}, \eqref{R2bnds}, \eqref{R3bnds}.
\end{proof}

Collecting the estimates we proved for $IV^{\epsilon,u}_{k}$, $k=1,\dots,6$ and $R^{\epsilon,u}$ we can finally prove the following:
\begin{lem}\label{IVbnds} Let $T<\infty$, $a>0$, $x_0,y_0\in H^{a}(0,L)$ and $IV^{\epsilon,u}$ as in \eqref{Itoeta3}. There exist $\epsilon_0>0$, $\theta<\frac{1}{2}\wedge a$, $\beta<\frac{1}{4}\wedge\frac{a}{2}$ and a constant $C>0$ independent of $\epsilon$ such that

\begin{equation}\label{IVSob}
\begin{aligned}
&	\sup_{\epsilon<\epsilon_0, u\in\mathcal{P}^T_N}\ex\bigg(\sup_{t\in[0, T] }\sup_{\chi\in B_\h}\big|IV^{\epsilon,u}(0,t,\theta,\chi)\big|^2\bigg)\leq C\big(1+\|x_0\|^2_{H^a}+\|y_0\|^2_{H^a}\big)
\end{aligned}
\end{equation}	
and
\begin{equation}\label{IVholder}
\begin{aligned}
&		\sup_{\epsilon<\epsilon_0, u\in\mathcal{P}^T_N}\ex\bigg(\sup_{\overset{s,t\in[0, T]}{ t\neq s}}\sup_{\chi\in B_\h}\frac{\big|IV^{\epsilon,u}(s,t,0,\chi)\big|}{  |t-s|^\beta  }\bigg)\leq C\big( 1+\|x_0\|_{H^a}+\|y_0\|_{H^a}\big).
\end{aligned}
\end{equation}
\end{lem}
\begin{proof} In view of \eqref{IV1Sob}, \eqref{IV2Sob}, \eqref{IV3Sob}, \eqref{IV4Sob}, \eqref{IV5Sob}, \eqref{IV6Sob} and \eqref{RSob} there exist $\epsilon_0>0$, $\theta<\frac{1}{2}\wedge a$ and, for each $\epsilon>0$, a $n(\epsilon)\in\N$  such that
\begin{equation}
\begin{aligned}
\sup_{\epsilon<\epsilon_0}\ex\bigg(\sup_{t\in[0, T] }\sup_{\chi\in B_\h}\big|IV^{\epsilon,u}(0,t,\theta,\chi)\big|^2\bigg)&\leq C \sum_{k=1}^{6} 	\sup_{\epsilon<\epsilon_0}\ex\bigg(\sup_{t\in[0, T] }\sup_{\chi\in B_\h}\big|IV_k^{\epsilon,u}(0,t,n(\epsilon),\theta,\chi)\big|^2\bigg)\\&+ C	\sup_{\epsilon<\epsilon_0}\ex\bigg(\sup_{t\in[0, T] }\sup_{\chi\in B_\h}\big|R^{\epsilon,u}(0,t,n(\epsilon),\theta,\chi)\big|^2\bigg)\\&
\leq C\big( 1+\|x_0\|^2_{H^a}+\|y_0\|^2_{H^a}\big),
\end{aligned}
\end{equation}	
\noindent which proves \eqref{IVSob}. Finally, in view of \eqref{IV1holder}, \eqref{IV2holder}, \eqref{IV3holder}, \eqref{IV4holder}, \eqref{IV5holder}, \eqref{IV6holder} and \eqref{Rholder} there exist $\epsilon_0>0$, $\beta<\frac{1}{4}\wedge\frac{a}{2}$ and, for each $\epsilon>0$, a $n(\epsilon)\in\N$  such that
\begin{equation*}
\begin{aligned}
\sup_{\epsilon<\epsilon_0}\ex\bigg(\sup_{\overset{s,t\in[0, T]}{ t\neq s}}\sup_{\chi\in B_\h}\frac{\big|IV^{\epsilon,u}(s,t,0,\chi)\big|}{  |t-s|^\beta  }\bigg)&\leq \sum_{k=1}^{6}\sup_{\epsilon<\epsilon_0}\ex\bigg(\sup_{\overset{s,t\in[0, T]}{ t\neq s}}\sup_{\chi\in B_\h}\frac{\big|IV_k^{\epsilon,u}(s,t,n(\epsilon),0,\chi)\big|}{  |t-s|^\beta  }\bigg)\\&
+\sup_{\epsilon<\epsilon_0}\ex\bigg(\sup_{\overset{s,t\in[0, T]}{ t\neq s}}\sup_{\chi\in B_\h}\frac{\big|R^{\epsilon,u}(s,t,n(\epsilon),0,\chi)\big|}{  |t-s|^\beta  }\bigg)\\&
\leq C\big( 1+\|x_0\|_{H^a}+\|y_0\|_{H^a}\big),
\end{aligned}
\end{equation*}
which proves \eqref{IVholder} and completes the argument.
\end{proof}

\subsection{Proof of Proposition \ref{etatightness}}\label{etaproof}
We can now combine the estimates of this section and prove the desired a priori estimates for $\eta^{\epsilon,u}$.

(i) Setting $s=0$ in the decomposition \eqref{etaspacetimedec}  (recall that $\eta^{\epsilon,u}(0)=0_\h$)
\begin{equation*}
\begin{aligned}
\|\eta^{\epsilon,u}(t)\|^2_{H^\theta}=\sup_{\chi\in B_\h}\big|\blangle \eta^{\epsilon,u}(t),(-A_1)^{\frac{\theta}{2}}\chi\brangle_\h\big|^2&\leq
\sup_{\chi\in B_\h}\big|I^{\epsilon,u}(0,t,\theta,\chi)\big|^2
+\sup_{\chi\in B_\h}\big|II^{\epsilon,u}(0,t,\theta,\chi)\big|^2\\&+\sup_{\chi\in B_\h}\big|III^{\epsilon,u}(0,t,\theta,\chi)\big|^2+\sup_{\chi\in B_\h}\big|IV^{\epsilon,u}(0,t,\theta,\chi)\big|^2.
\end{aligned}
\end{equation*}
\noindent
In view of  $\eqref{ISob}$,
\begin{equation*}
\begin{aligned}
\|\eta^{\epsilon,u}(t)\|^2_{H^\theta}&\leq C
\int_{0}^{t}(t-z)^{-\theta}\big\|\eta^{\epsilon,u}(z)\big\|^2_{H^\theta} dz
+\sup_{t\in[0, T] }\sup_{\chi\in B_\h}\big|II^{\epsilon,u}(0,t,\theta,\chi)\big|^2\\&+\sup_{t\in[0, T] }\sup_{\chi\in B_\h}\big|III^{\epsilon,u}(0,t,\theta,\chi)\big|^2+\sup_{t\in[0, T] }\sup_{\chi\in B_\h}\big|IV^{\epsilon,u}(                                                    0,t,\theta,\chi)\big|^2.
\end{aligned}
\end{equation*}
\noindent An application of Gr\"onwall's inequality then yields
\begin{equation*}
\begin{aligned}
\|\eta^{\epsilon,u}(t)\|^2_{H^\theta}\leq C_{T,\theta}\bigg(&
\sup_{t\in[0, T] }\sup_{\chi\in B_\h}\big|II^{\epsilon,u}(0,t,\theta,\chi)\big|^2+\sup_{t\in[0, T] }\sup_{\chi\in B_\h}\big|III^{\epsilon,u}(0,t,\theta,\chi)\big|^2\\&+\sup_{t\in[0, T] }\sup_{\chi\in B_\h}\big|IV^{\epsilon,u}(0,t,\theta,\chi)\big|^2\bigg).
\end{aligned}
\end{equation*}
Taking expectation and invoking \eqref{ISob},  \eqref{IISob},  \eqref{IIISob} and \eqref{IVSob} we obtain
\begin{equation*}
\begin{aligned}
&\ex\sup_{t\in[0, T] } \|\eta^{\epsilon,u}(t)\|^2_{H^\theta}\leq C\big( 1+\|x_0\|^2_{H^a}+\|y_0\|^2_{H^a}\big),
\end{aligned}
\end{equation*}
which holds for $\epsilon$ sufficiently small,  $\theta<(\frac{1}{2}-\nu)\wedge a$ and proves \eqref{etaSob}. \\

\noindent (ii) Setting $\theta=0$ in the decomposition \eqref{etaspacetimedec} we apply a reverse triangle inequality to obtain
\begin{equation*}
\begin{aligned}
\|\eta^{\epsilon,u}(t)-\eta^{\epsilon,u}(s) \|_\h &\leq \|(S_1(t-s)-I\big)\eta^{\epsilon,u}(s)\|_\h+ \sup_{\chi\in B_\h}\big|I^{\epsilon,u}(s,t,0,\chi)\big|+ \sup_{\chi\in B_\h}\big|II^{\epsilon,u}(s,t,0,\chi)\big|\\&
+ \sup_{\chi\in B_\h}\big|III^{\epsilon,u}(s,t,0,\chi)\big|+ \sup_{\chi\in B_\h}\big|IV^{\epsilon,u}(s,t,0,\chi)\big|\\&
\leq C (t-s)^{\theta/2}\|\eta^{\epsilon,u}(s)\|_{H^\theta}+ \sup_{\chi\in B_\h}\big|I^{\epsilon,u}(s,t,0,\chi)\big|+ \sup_{\chi\in B_\h}\big|II^{\epsilon,u}(s,t,0,\chi)\big|\\&
+ \sup_{\chi\in B_\h}\big|III^{\epsilon,u}(s,t,0,\chi)\big|+ \sup_{\chi\in B_\h}\big|IV^{\epsilon,u}(s,t,0,\chi)\big|,
\end{aligned}
\end{equation*}
where we used \eqref{sobcont}
to obtain the last inequality. Hence for any  $\beta<\theta/2<(\frac{1}{4}-\frac{\nu}{2})\wedge\frac{a}{2}$
we take expectation and apply  \eqref{Iholder}, \eqref{IIholder}, \eqref{IIIholder} and \eqref{IVholder} along with \eqref{etaSob} to deduce that
\begin{equation*}
\begin{aligned}
\ex\sup_{\overset{s,t\in[0, T]}{ t\neq s}}&\frac{\big\|\eta^{\epsilon,u}(t)-\eta^{\epsilon,u}(s)\big\|_\h}{|t-s|^\beta}\leq C \ex\sup_{t\in[0,T]}\big\|\eta^{\epsilon,u}(t)\big\|_{H^{\theta}}+\ex\sup_{\overset{s,t\in[0, T]}{ t\neq s}}\sup_{\chi\in B_\h}\frac{\big|I^{\epsilon,u}(s,t,0,\chi)\big|}{  |t-s|^\beta  }\\&+\ex\sup_{\overset{s,t\in[0, T]}{ t\neq s}}\sup_{\chi\in B_\h}\frac{\big|II^{\epsilon,u}(s,t,0,\chi)\big|}{  |t-s|^\beta  }+\ex\sup_{\overset{s,t\in[0, T]}{ t\neq s}}\sup_{\chi\in B_\h}\frac{\big|III^{\epsilon,u}(s,t,0,\chi)\big|}{  |t-s|^\beta  }\\&+\ex\sup_{\overset{s,t\in[0, T]}{ t\neq s}}\sup_{\chi\in B_\h}\frac{\big|IV^{\epsilon,u}(s,t,0,\chi)\big|}{  |t-s|^\beta  } \leq C\big( 1+\|x_0\|_{H^a}+\|y_0\|_{H^a}\big).
\end{aligned}
\end{equation*}
\noindent The proof is complete.

	\section{Tightness of the pairs $(\eta^{\epsilon,u}, P^{\epsilon, \Delta})$ and analysis of the limit}\label{Sec4}

Let $\eta^{\epsilon,u}$ denote the controlled moderate deviation processes defined in \eqref{etau1} and $P^{\epsilon, \Delta}$ the random occupation measures defined in \eqref{occupation1}.
In this section, we prove the first main result of this paper, Theorem \ref{viablim1}. To do so, we first show that the family $\{ (\eta^{\epsilon,u}, P^{\epsilon, \Delta}) , \epsilon>0, u\in\mathcal{P}_N^T\}$ is tight in Section \ref{tightsub} and then identify the limiting dynamics in Section \ref{limpnts}. We complete the proof of Theorem \ref{viablim1} in Section \ref{aveproof}.

Before we proceed to the main body of this section,  let us recall the notion of tightness for a family of probability measures and then state an extension of the classical theorem of Prokhorov which will be used in the sequel.

\begin{dfn}\label{D:TightnessDef}
	Let $\mathcal{E}$ be a Hausdorff topological space and $\Pi\subset\mathscr{P}(\mathcal{E})$ be a set of Borel probability measures on $\mathcal{E}$.
	\noindent (i) We say that a sequence $\{P_n\}\subset\Pi$ converges weakly to a measure $P\in\mathscr{P}(\mathcal{E})$ if for every $f\in C_b(\mathcal{E})$
	\begin{equation*}
	\lim_{n\to\infty} \int_{\mathcal{E}} f dP_n= \int_{\mathcal{E}} f dP.
	\end{equation*}
	(ii) We say that $\Pi$ is \textit{tight} if for each $\epsilon>0$ there exists a compact set $K_\epsilon\subset \mathcal{E}$ such that for all $P\in\Pi$,
	\begin{equation}\label{tight}
	P(\mathcal{E}\setminus K_\epsilon)<\epsilon.
	\end{equation}
\end{dfn}
\noindent The classical version of Prokhorov's theorem asserts that  the notions of tightness and relative weak sequential compactness on $\mathscr{P}(\mathcal{E})$ are equivalent, provided that $\mathcal{E}$ is a Polish space. The following generalization can be found e.g. in \cite{bogachev2007measure} (see Theorem 8.6.7).
\begin{thm}(Prokhorov)\label{prokhorov} Let $\mathcal{E}$ be a completely regular Hausdorff topological space and $\Pi\subset\mathscr{P}(\mathcal{E})$ be a tight family of Borel probability measures. Then  $\Pi$ has compact closure in the topology of weak convergence of measures. In addition, if for each $\epsilon>0$ the set $K_\epsilon$ in \eqref{tight} is metrizable, then every sequence in $\Pi$ contains a weakly convergent subsequence.
\end{thm}
\subsection{Tightness of $\{ (\eta^{\epsilon,u}, P^{\epsilon, \Delta}) , \epsilon\in(0,1), u\in\mathcal{P}_N^T\}$}\label{tightsub}
\begin{lem}\label{etarzela} Let $T<\infty, N>0$, $a>0$ and $(X^{\epsilon,u}, Y^{\epsilon,u})$ denote the mild solution of \eqref{controlledsystem} with initial conditions $x_0,y_0\in H^a(0,L)$. Then the family $\{ \eta^{\epsilon,u}, \epsilon\in(0,1), u\in\mathcal{P}_N^T\}$ is tight in $C\big([0,T];\h\big)$.
\end{lem}
\begin{proof}
	Let $M,\beta,\theta>0$. From an infinite-dimensional version of the Arzel\`a-Ascoli theorem, sets of the form
	\begin{equation*}
	\label{Arzela}
	\mathcal{K}_{M,\beta,\theta}=\bigg\{ X\in C([0,T];\h)   : \|X\|_{C^\beta([0,T];\h)} \leq M\;,\;    \sup_{t\in[0,T]}\|X(t)\|_{H^\theta}   \leq M        \bigg\}
	\end{equation*}
	are compact in $C([0,T];\h)$. Indeed, since the inclusion $H^\theta(0,L)\subset \h$ is compact, we see that $\mathcal{K}_{M,\beta,\theta}$ contain uniformly equicontinuous paths with values on compact subsets of $\h$. In view of Proposition \ref{etatightness} in Section \ref{Sec2}, there exist
	$\theta_0<\frac{1}{2}-\nu$ and $\beta_0<\frac{1}{4}-\frac{\nu}{2}$ such that
	\begin{equation*}
	\lim_{M\to\infty}\sup_{\epsilon\in(0,1), u\in\mathcal{P}_N^T}\pr\big[  \eta^{\epsilon,u}\notin  \mathcal{K}_{M,\beta_0,\theta_0}      \big]=0.
	\end{equation*}
	Equivalently, the probability laws of the processes $\eta^{\epsilon,u}$ are concentrated in compact subsets of $C([0,T];\h)$, uniformly in $\epsilon,u$. The proof is complete.
\end{proof}
\noindent In order to show that the laws of the random occupation measures $P^{\epsilon, \Delta}$ form a tight subset of $\mathscr{P}(\mathscr{P}(\h\times\h\times\h\times[0,T]))$ we need the following auxiliary lemma regarding the spatial regularity of the fast process $Y^{\epsilon,u}$.
\begin{lem} Let $T<\infty$. There exists $\theta>0$ and a constant $C>0$, independent of $\epsilon$, such that
	\begin{equation}\label{fastSob}
	\sup_{\epsilon>0, u\in\mathcal{P}^T_N}\ex\int_{0}^{T}\big\| Y^{\epsilon,u}(t)\big\|^2_{H^\theta}dt\leq C\big(1+\|x_0\|_\h^2+\|y_0\|_\h^2\big).
	\end{equation}
	
\end{lem}

\begin{proof} Recall that the mild solution of the controlled fast equation (see \eqref{controlledsystem}) is given by
	\begin{equation*}
	\begin{aligned}
	Y^{\epsilon,u}(t)=&S_2\bigg(\frac{t}{\delta}\bigg)y_0+\frac{1}{\delta}\int_{0}^{t}S_2\bigg(\frac{t-s}{\delta}\bigg)G\big(X^{\epsilon,u}(s),    Y^{\epsilon,u}(s)    \big) ds+\frac{h(\epsilon)}{\delta}\int_{0}^{t}S_2\bigg(\frac{t-s}{\delta}\bigg)u_2(s)ds\\&
	+\frac{1}{\sqrt{\delta}}\int_{0}^{t}S_2\bigg(\frac{t-s}{\delta}\bigg)dw_2(s).
	\end{aligned}
	\end{equation*}
	Using the analytic properties of the semigroup and the linear growth of $G$, we can estimate the first two terms by
	\begin{equation}\label{fastSob1}
	\int_{0}^{T}\bigg\|S_2\bigg(\frac{t}{\delta}\bigg)y_0\bigg\|^2_{H^{\theta}}dt\leq C \int_{0}^{T}\bigg(\frac{t}{\delta}\bigg)^{-\theta}e^{-\frac{\lambda t}{\delta}}\|y_0\|_\h^2dt\leq C_{\lambda,\theta}\delta \|y_0\|_\h^2
	\end{equation}
	and
	\begin{equation*}
	\begin{aligned}
	&\bigg\|\int_{0}^{t}S_2\bigg(\frac{t-s}{\delta}\bigg)G\big(X^{\epsilon,u}(s),    Y^{\epsilon,u}(s)    \big) ds\bigg\|_{H^{\theta}}\leq C\int_{0}^{t}\bigg(\frac{t-s}{\delta}\bigg)^{-\theta/2}e^{-\frac{\lambda (t-s)}{2\delta}}\big\| G\big(X^{\epsilon,u}(s),    Y^{\epsilon,u}(s)    \big)    \big\|_\h ds.
	\end{aligned}
	\end{equation*}
	\noindent Applying Young's inequality for convolutions in the form $\|f\star g\|_{2}\leq \|f\|_{1}\|g\|_2$ we obtain
	\begin{equation}\label{fastSob2}
	\begin{aligned}
	\ex\int_{0}^{T}\bigg\|\frac{1}{\delta}\int_{0}^{t}&S_2\bigg(\frac{t-s}{\delta}\bigg)G\big(X^{\epsilon,u}(s),    Y^{\epsilon,u}(s)    \big) ds\bigg\|^2_{H^{\theta}}dt\\&\leq C
	\bigg(\int_{0}^{\infty}t^{-\theta/2}e^{-\frac{\lambda t}{2}}dt\bigg)^2\ex\int_{0}^{T}\big(1+ \big\| X^{\epsilon,u} (t)\big\|^2_\h+ \big\| Y^{\epsilon,u} (t)\big\|^2_\h \big)dt
	\\&
	\leq C\big(1+\|x_0\|_\h^2+\|y_0\|_\h^2\big),
	\end{aligned}
	\end{equation}
	where the last inequality follows from the a priori bounds \eqref{xapriori}, \eqref{yapriori} in Section \ref{Sec2}.
	It remains to estimate the control and stochastic convolution terms. The first can be bounded by Young's inequality for convolutions and the $L^2$ bound on the controls as follows:
	\begin{equation}\label{fastSob3}
	\begin{aligned}
	\int_{0}^{T}\bigg\|\frac{h(\epsilon)}{\sqrt{\delta}}\int_{0}^{t}S_2\bigg(\frac{t-s}{\delta}\bigg)u(s)ds\bigg\|^2_{H^{\theta}}dt&\leq \frac{h^2(\epsilon)}{\delta}\bigg(\int_{0}^{T}(t/\delta)^{-\theta/2}e^{-\frac{\lambda t}{2\delta}}dt\bigg)^2 \int_{0}^{T}\|u(t)\|^2_\h dt\\&\leq N\frac{h^2(\epsilon)}{\delta}\delta^2\bigg(\int_{0}^{\infty}s^{-\theta/2}e^{-\frac{\lambda s}{2}}ds\bigg)^2 \\&\leq C \delta h^2(\epsilon)\longrightarrow 0\;,\;\text{as}\;\epsilon\to 0.
	\end{aligned}
	\end{equation}
 The last line above follows from the change of variables $s=t/\delta$ and the integral is finite provided that $\theta<2$. Finally, for the stochastic convolution term, we can proceed as in \cite{WSS} (see Lemma 4.6, (33) and set $\Sigma=I$) to show that
	\begin{equation}\label{fastSob4}
	\ex\int_{0}^{T}\bigg\|\frac{1}{\sqrt{\delta}}\int_{0}^{t}S_2\bigg(\frac{t-s}{\delta}\bigg)dw_2(s)\bigg\|^2_{H^{\theta}}    dt\leq C.
	\end{equation}
	\noindent The proof is complete upon combining \eqref{fastSob1}-\eqref{fastSob4}. \end{proof}

\noindent We can now argue that the family of occupation measures $P^{\epsilon, \Delta}$ is tight. The difference with the finite-dimensional case (see Proposition 3.1 in \cite{dupuis2011weak}) is that the controls take values on the infinite-dimensional space $\h$. Since the occupation measures are defined on $\h\times\h\times\h\times[0,T]$ with the WWNS topology and the weak topology is not globally metrizable, it follows that $\h\times\h\times\h\times[0,T]$ is not a Polish space (and consequently neither is $\mathscr{P}(\h\times\h\times\h\times[0,T])$ with the topology of weak convergence of measures). This is why we need Theorem \ref{prokhorov}.

\begin{lem}\label{Palaoglu} The family $\{  P^{\epsilon, \Delta}, \epsilon>0\}$ is tight in $ \mathscr{P}(\h\times\h\times\h\times[0,T])$ where $\h\times\h\times\h\times[0,T]$ is endowed with the WWNS topology.
\end{lem}

\begin{proof} Let $M>0$ and define
	\begin{equation*}
	\mathcal{K}_{M}=\big\{ (u_1,u_2,y)\in\h\times\h\times\h :\|u_1\|^2_\h+\|u_2\|^2_\h+\|y\|^2_{H^\theta} \leq M\big\}\times[0,T].
	\end{equation*}
	Since
	\begin{equation*}\label{Alaoglu}
	\begin{aligned}
	\mathcal{K}_{M}\subset
	\big\{ (u_1,u_2)\in\h\times\h :\|u_1\|^2_\h+\|u_2\|^2_\h \leq M\big\}\times\{ y\in\h : \|y\|^2_{H^\theta}  \leq M \big \}\times [0,T],
	\end{aligned}
	\end{equation*}
	we invoke the Banach-Alaoglu theorem along with the compact inclusion $H^{\theta}(0,L)\subset\h$ to deduce that 	$\mathcal{K}_{M}$ is compact in the WWNS topology. Next define
	\begin{equation*}\label{proho}
	\begin{aligned}
	\Pi_{i,j}=\bigcap_{L\geq i}\bigcup_{M\geq j}\bigg\{ P\in\mathscr{P}(\h\times\h\times\h\times[0,T])  :  P(	\mathcal{K}^c_{M}) <\frac{1}{L}       \bigg\}\;, i,j\in\N.
	\end{aligned}
	\end{equation*}
	By Definition \ref{D:TightnessDef} it follows that, for each $i,j,$  $\Pi_{i,j}$ is a tight family of measures. Since $\h$ is a separable Hilbert space and the weak topology on $B_\h$ is metrizable, the sets $\mathcal{K}_M$ are compact, metrizable. Thus, in light of Theorem \ref{prokhorov}, the sets $ \Pi_{i,j}$ are relatively compact and, in fact,  relatively sequentially compact in the topology of $ \mathscr{P}(\h\times\h\times\h\times[0,T])$. Now, an application of Chebyshev's inequality along with estimate \eqref{fastSob} yields
	\begin{equation*}
	\begin{aligned}
	\ex\big[ P^{\epsilon,\Delta}(\mathcal{K}_{M}^c) \big]&=\frac{1}{\Delta}\int_{0}^{T}\int_{t}^{t+\Delta}\pr[ (u_1(s),u_2(s),Y^{\epsilon,u}(s)) \in  \mathcal{K}_{M}^c]\;ds dt\\&\leq \frac{1}{M\Delta}\int_{0}^{T}\ex\int_{t}^{t+\Delta}\big(\|u_1(s)\|_\h^2+\|u_2(s)\|_\h^2+\|Y^{\epsilon,u}(s)\|^2_{H^\theta}\big) ds dt\\&
	\leq \frac{1}{M}\int_{0}^{T+\Delta}\big(\ex\|u_1(s)\|_\h^2+\ex\|u_2(s)\|_\h^2+\ex\|Y^{\epsilon,u}(s)\|^2_{H^\theta}\big)ds\leq \frac{C_N}{M}(1+\|x_0\|_\h^2+\|y_0\|_\h^2).
	\end{aligned}
	\end{equation*}
	Yet another application of Chebyshev's inequality implies that
	\begin{equation*}
	\begin{aligned}
	\pr\bigg[ P^{\epsilon,\Delta}(\mathcal{K}_{M}^c)\geq\frac{1}{L}\bigg] &\leq \frac{C_NL}{M}(1+\|x_0\|_\h^2+\|y_0\|_\h^2).
	\end{aligned}
	\end{equation*}
	Next, let $i\in\N,\rho>0$ and take $L\geq i$ and
	$$M\geq C_NL(1+\|x_0\|_\h^2+\|y_0\|_\h^2)/\rho\geq [C_Ni(1+\|x_0\|_\h^2+\|y_0\|_\h^2)/\rho]=:j(i,\rho), $$ where $[\cdot]$ indicates the floor function. It follows that
	\begin{equation*}
	\begin{aligned}
	\pr\big[ P^{\epsilon,\Delta}\notin \Pi_{i, j(i,\rho) }   \big]= \lim_{M\to\infty}\lim_{L\to\infty}\pr\bigg[ P^{\epsilon,\Delta}(\mathcal{K}_{M}^c)\geq\frac{1}{L}\bigg] \leq\rho,
	\end{aligned}
	\end{equation*}
	uniformly in $\epsilon,u$. Since $\rho$ is arbitrary the proof is complete.  \end{proof}
Finally, we state here, without proof, a result regarding the tail behavior of the random measures $P^{\epsilon,\Delta}$. The proof follows the same strategy as that of Proposition 3.1 in \cite{dupuis2012large} (see also Lemma 4.14 in \cite{WSS}).

\begin{lem}\label{UILem}
	Let $M,\theta>0, T<\infty$ and \begin{equation*}
	U_{M,\theta,T}:=\big\{(u_1,u_2,y,t):\|u_1\|_\h\geq M,\; \|u_2\|_\h\geq M,\; \|y\|_{H^\theta}\geq M, t\in[0,T]\big\}.
	\end{equation*}For all $T$ there exists $\theta$ such that the occupation measures $P^{\epsilon,\Delta}$  are uniformly integrable, in the sense that
	\begin{equation*}\label{UI}
	\hspace*{-0.3cm}
	\lim_{M\to\infty}\sup_{\epsilon>0}\ex\int_{U_{M,\theta,T}}\big( \|u_1\|_\h+\|u_2\|_\h+\|y\|_{H^\theta}\big)dP^{\epsilon, \Delta}(u_1,u_2, y,t)=0.
	\end{equation*}
\end{lem}

  \subsection{Identification of the limit points}\label{limpnts}

  Let $i=1,2$. In view of Lemmas \ref{etarzela} and \ref{Palaoglu} along with Prokhorov's theorem, each sequence of $\epsilon>0, u\in\mathcal{P}_N^T$ contains a subsequence $\epsilon_n, u_n$ such that $( \eta^{\epsilon_n,u_n}, P^{\epsilon_n,\Delta_n})$ converges in distribution to a random element $(\eta_i, P_i)$ in Regime $i$.
Returning to the decomposition \eqref{etaspacetimedec}, we can use very similar arguments to the ones found in Sections \ref{1,2,3}, \ref{4} and Lemma \ref{etarzela} to show that each one of the terms $I^{\epsilon,u}(0,t,0,\chi), II^{\epsilon,u}(0,t,0,\chi), III^{\epsilon,u}(0,t,0,\chi)$, $IV^{\epsilon,u}(0,t,0,\chi)$ are tight. Invoking Prokhorov's theorem once again, each of these terms have subsequential limits in distribution on $C([0,T];\h)$. The goal of this section is to identify these limits.

At this point we will use the Skorokhod representation theorem which allows us to assume that  the aforementioned sequences of random elements converge almost surely. The Skorokhod representation theorem involves the introduction of another probability space but this distinction is ignored in the notation.

In view of Lemma \ref{IIIbnds} we immediately see that the third term in \eqref{etaspacetimedec} converges to $0$ in distribution. Hence, it suffices to study the limits of $I^{\epsilon,u}$, $II^{\epsilon,u}$ and $IV^{\epsilon,u}$. This is done in Propositions \ref{linlim},  \ref{IIlim} and \ref{controlim} below. The proofs of these Propositions are based on a few preliminary lemmas which follow the general strategy of Lemmas 4.16, 4.17 in \cite{WSS}. Thus,  to avoid repetition, some intermediate steps in the proof of  Proposition \ref{linlim} as well as the proof of Proposition \ref{IIlim} will be omitted. Let us remark at this point that the averaging of $IV^{\epsilon,u}$ presents challenges that are absent from both the finite-dimensional MDP and the infinite-dimensional LDP. These are related to continuity properties of the operator-valued map $\Psi_2^0$ in \eqref{psi0}, which are here investigated with the aid of the first variation equation corresponding to the Markov process $Y^{x,y}$ \eqref{Yfrozen} (see Lemma \ref{varicontlem}). For this reason, we will present the proof of Proposition \ref{controlim} in full detail.

We start with $I^{\epsilon,u}$. Using Taylor approximation we can show that the limit of this term is linear in $\eta_i$.

\begin{lem}\label{Taylorlem} Let $T<\infty$. Under Hypothesis \ref{A2a} we have
	\begin{equation*}\label{Taylorform}
	\begin{aligned}
	\ex\sup_{t\in[0,T]} \bigg\|\frac{1}{\sqrt{\epsilon}h(\epsilon)}   &\int_{0}^{t}S_1(t-s)\big[ F\big(\bar{X}(s) +\sqrt{\epsilon}h(\epsilon)\eta^{\epsilon,u^\epsilon}(s), Y^{\epsilon,u^\epsilon}(s) \big)-F\big(\bar{X}(s), Y^{\epsilon,u^\epsilon}(s) \big)\big]ds\\&
	-\int_{0}^{t}S_1(t-s)D_xF\big(\bar{X}(s), Y^{\epsilon,u^\epsilon}(s) \big)\big( \eta^{\epsilon,u^\epsilon}(s)\big)ds\bigg\|_\h\longrightarrow 0 \;,\;\text{as}\;\epsilon\to 0.
	\end{aligned}
	\end{equation*}
\end{lem}
\begin{proof} Let $x,y,h\in\h$. A first-order Taylor expansion for G\^ateaux derivatives yields
	\begin{equation*}\label{taylor}
	F(x+h,y)=F(x,y)+D_xF(x,y)(h)+ \frac{1}{2}D^2_{x}F(x+\theta_0h,y)(h,h),
	\end{equation*}
	for some $\theta_0\in(0,1)$ (note that here we are considering $F:\h\times\h\rightarrow L^1(0,L) $). Letting $x=\bar{X}(s), y= Y^{\epsilon,u^\epsilon}(s)$ and $h=\sqrt{\epsilon}h(\epsilon)\eta^{\epsilon,u^\epsilon}(s)$, we integrate over $[0,t]$ to obtain
	\begin{equation*}\label{taylorapp}
	\begin{aligned}
	\frac{1}{\sqrt{\epsilon}h(\epsilon)} &\int_{0}^{t}S_1(t-s)\big[ F\big(\bar{X}(s) +\sqrt{\epsilon}h(\epsilon)\eta^{\epsilon,u^\epsilon}(s), Y^{\epsilon,u^\epsilon}(s) \big)-F\big(\bar{X}(s), Y^{\epsilon,u^\epsilon}(s) \big)\big]ds \\&
	=   \int_{0}^{t}S_1(t-s)D_xF\big(\bar{X}(s), Y^{\epsilon,u^\epsilon}(s) \big)\big( \eta^{\epsilon,u^\epsilon}(s)\big)ds\\&+ \frac{\sqrt{\epsilon}h(\epsilon)}{2}\int_{0}^{t}S_1(t-s)D^2_{x}F\big(\bar{X}(s) +\theta_0\sqrt{\epsilon}h(\epsilon)\eta^{\epsilon,u^\epsilon}(s),Y^{\epsilon,u^\epsilon}(s)\big)\big( \eta^{\epsilon,u^\epsilon}(s) ,  \eta^{\epsilon,u^\epsilon}(s)\big) ds,
	\end{aligned}
	\end{equation*}
	where we used the homogeneity of the G\^ateaux derivative to simplify the $\epsilon$-dependent coefficients. In view of the regularizing property \eqref{Lpsmoothing} (with $r=2,p=1$), along with \eqref{DxxF}, we obtain
	\begin{equation*}\label{L4}
	\begin{aligned}
	\sqrt{\epsilon}h(\epsilon)&\bigg\|\int_{0}^{t}S_1(t-s)D^2_{x}F\big(\bar{X}(s) +\theta_0\sqrt{\epsilon}h(\epsilon)\eta^{\epsilon,u^\epsilon}(s),Y^{\epsilon,u^\epsilon}(s)\big)\big( \eta^{\epsilon,u^\epsilon}(s) ,  \eta^{\epsilon,u^\epsilon}(s)\big) ds\bigg\|_\h \\&
	\leq c\sqrt{\epsilon}h(\epsilon)\int_{0}^{t}(t-s)^{-\frac{1}{4}}\|D^2_{x}F\big(\bar{X}(s) +\theta_0\sqrt{\epsilon}h(\epsilon)\eta^{\epsilon,u^\epsilon}(s),Y^{\epsilon,u^\epsilon}(s)\big)\big( \eta^{\epsilon,u^\epsilon}(s) ,  \eta^{\epsilon,u^\epsilon}(s)\big) \big\|_{L^1(0,L)}ds
	\\&\leq c \sqrt{\epsilon}h(\epsilon) \big\|   \partial_{\mathrm{x}\mathrm{x}}^2 f \big\|_{\infty}    \int_{0}^{t}(t-s)^{-\frac{1}{4}}\|\eta^{\epsilon,u^\epsilon}(s)\|^2_{\h}ds\leq C T^{3/4}\sqrt{\epsilon}h(\epsilon) \big\|   \partial_{\mathrm{x}\mathrm{x}}^2 f \big\|_{\infty}   \sup_{s\in[0, T] }\|\eta^{\epsilon,u^\epsilon}(s)\|^2_{\h}.
	\end{aligned}
	\end{equation*}
	\noindent Taking expectation, we use \eqref{etaSob} to deduce
	\begin{equation*}\label{Taylorem}
	\begin{aligned}
	\sqrt{\epsilon}h(\epsilon)\ex\sup_{t\in[0, T] }\bigg\|&\int_{0}^{t}S_1(t-s)D^2_{x}F\big(\bar{X}(s) +\theta_0\sqrt{\epsilon}h(\epsilon)\eta^{\epsilon,u^\epsilon}(s),Y^{\epsilon,u^\epsilon}(s)\big)\big( \eta^{\epsilon,u^\epsilon}(s) ,  \eta^{\epsilon,u^\epsilon}(s)\big) ds\bigg\|_\h \\&
	\leq C \sqrt{\epsilon}h(\epsilon) \ex   \sup_{s\in[0, T] }\|\eta^{\epsilon,u^\epsilon}(s)\|^2_{\h}
	\leq C \sqrt{\epsilon}h(\epsilon)\big(1+\|x_0\|^2_{H^a}+\|y_0\|^2_{H^a}\big)\longrightarrow 0
	\end{aligned}
	\end{equation*}
	\noindent as $\epsilon\to 0$. The proof is complete.
\end{proof}

\begin{lem}\label{limcontlem}
Let $\Delta$ as in \eqref{Delta1} and $T<\infty$. Under Hypothesis \ref{A2a} we have
\begin{equation*}\label{contprop}
\begin{aligned}
\ex\sup_{t\in[0,T]}\bigg\|\frac{1}{\Delta}&\int_{0}^{t}\int_{s}^{s+\Delta} S_1(t-s)D_xF\big(\bar{X}(s), Y^{\epsilon,u^\epsilon}(r)\big)\big( \eta^{\epsilon,u^\epsilon}(s)\big)drds\\&-  \frac{1}{\Delta}\int_{0}^{t}\int_{s}^{s+\Delta} S_1(t-s)D_xF\big(\bar{X}(r), Y^{\epsilon,u^\epsilon}(r)\big)\big( \eta^{\epsilon,u^\epsilon}(s)\big)drds\bigg\|_\h\longrightarrow 0 \;,\;\text{as}\;\epsilon\to 0.
\end{aligned}
\end{equation*}

\end{lem}

\begin{proof} In view of the regularizing property \eqref{Lpsmoothing},
\begin{equation*}
\begin{aligned}
\bigg\|\frac{1}{\Delta}&\int_{0}^{t}\int_{s}^{s+\Delta} S_1(t-s)\big[D_xF\big(\bar{X}(s), Y^{\epsilon,u^\epsilon}(r)\big)-D_xF\big(\bar{X}(r), Y^{\epsilon,u^\epsilon}(r)\big)\big]\big( \eta^{\epsilon,u^\epsilon}(s)\big)drds\bigg\|_\h\\&
\leq \frac{C}{\Delta}\int_{0}^{t}\int_{s}^{s+\Delta}(t-s)^{-\frac{1}{4}}\big\|\big[     D_xF\big(\bar{X}(s), Y^{\epsilon,u^\epsilon}(r)\big)-D_xF\big(\bar{X}(r), Y^{\epsilon,u^\epsilon}(r)\big)\big]\big( \eta^{\epsilon,u^\epsilon}(s)\big)\big\|_{L^1(0,L)} drds.
\end{aligned}
\end{equation*}
\noindent Next, let $r\in[s, s+\Delta]$. An application of the Cauchy-Schwarz and mean value inequalities yields
\begin{equation*}
\begin{aligned}
\big\|\big[     D_xF\big(\bar{X}(s)&, Y^{\epsilon,u^\epsilon}(r)\big)-D_xF\big(\bar{X}(r), Y^{\epsilon,u^\epsilon}(r)\big)\big]\big( \eta^{\epsilon,u^\epsilon}(s)\big) \big\|_{L^1(0,L)}
\\&\leq \big\|\eta^{\epsilon,u^\epsilon}(s)\|_\h
\bigg(\int_{0}^{L}\big| \partial_{\mathrm{x}}f\big(\xi,\bar{X}(s,\xi),Y^{\epsilon,u^\epsilon}(r,\xi)\big)-\partial_{\mathrm{x}}f\big(\xi,\bar{X}(r,\xi),Y^{\epsilon,u^\epsilon}(r,\xi)\big)\big|^2\ d\xi\bigg)^\frac{1}{2}
\\&\leq \big\|\partial^2_{\mathrm{x}\mathrm{x}}f\|_{\infty}\sup_{t\in[0,T]}\big\|\eta^{\epsilon,u^\epsilon}(t)\|_\h \|\bar{X}(s)-\bar{X}(r) \|_\h\;.
\end{aligned}
\end{equation*}
\noindent	In view of the Schauder estimate \eqref{Schauderbar} we obtain
\begin{equation*}
\begin{aligned}
\big\|\big[     D_xF\big(\bar{X}(s), Y^{\epsilon,u^\epsilon}(r)\big)-D_xF\big(\bar{X}(r), Y^{\epsilon,u^\epsilon}(r)\big)\big]\big( \eta^{\epsilon,u^\epsilon}(s)\big) \big\|_{L^1(0,L)} \leq C_f \sup_{t\in[0,T]}\big\|\eta^{\epsilon,u^\epsilon}(t)\|_\h \Delta^\theta (1+\|x_0\|_{H^a}),
\end{aligned}
\end{equation*}
\noindent 	where $\theta<\frac{1}{4}\wedge \frac{a}{2}$. Thus,
\begin{equation*}
\begin{aligned}
\bigg\|\frac{1}{\Delta}\int_{0}^{t}\int_{s}^{s+\Delta} S_1(t-s)&\big[D_xF\big(\bar{X}(s), Y^{\epsilon,u^\epsilon}(r)\big)-D_xF\big(\bar{X}(r), Y^{\epsilon,u^\epsilon}(r)\big)\big]\big( \eta^{\epsilon,u^\epsilon}(s)\big)drds\bigg\|_\h\\&\leq
C\Delta^{\theta}(1+\|x_0\|_{H^a})\sup_{t\in[0,T]}\big\|\eta^{\epsilon,u^\epsilon}(t)\|_\h \int_{0}^{t}(t-s)^{-\frac{1}{4}}ds
\\&\leq 	CT^{\frac{3}{4}}\Delta^{\theta}(1+\|x_0\|_{H^a})\sup_{t\in[0,T]}\big\|\eta^{\epsilon,u^\epsilon}(t)\|_\h.
\end{aligned}
\end{equation*}
\noindent In view of \eqref{etaSob} it follows that
\begin{equation*}
\begin{aligned}
\ex\sup_{t\in[0,T]}\bigg\|\frac{1}{\Delta}\int_{0}^{t}\int_{s}^{s+\Delta} S_1(t-s)&\big[D_xF\big(\bar{X}(s), Y^{\epsilon,u^\epsilon}(r)\big)-D_xF\big(\bar{X}(r), Y^{\epsilon,u^\epsilon}(r)\big)\big]\big( \eta^{\epsilon,u^\epsilon}(s)\big)drds\bigg\|_\h\\&\leq 	C_T\Delta^{\theta}(1+\|x_0\|_{H^a})(1+\|x_0\|_{H^a}+\|y_0\|_{H^a}).
\end{aligned}
\end{equation*}
The proof is complete upon taking $\Delta\rightarrow 0$.
\end{proof}

\begin{lem}\label{lincont2lem}
Let $i=1,2$, $T<\infty$ and assume that the pair $(\eta^{\epsilon,u^\epsilon}, P^{\epsilon,\Delta})$  converges in distribution, in Regime $i$, to $(\eta_i,P_i)$ in $C([0,T];\h)\times\mathscr{P}(\h\times\h\times\h\times[0,T])$. Then the following limit is valid with probability $1$:
\begin{equation*}
\begin{aligned}
\sup_{t\in[0,T]}\bigg\|
\frac{1}{\Delta}\int_{0}^{t}&\int_{s}^{s+\Delta} S_1(t-s)D_xF\big(\bar{X}(r), Y^{\epsilon,u^\epsilon}(r)\big)\big( \eta^{\epsilon,u^\epsilon}(s)\big)drds
\\&  -  \int_{\h\times\h\times\h\times[0,t] }S_1(t-s) D_xF\big(\bar{X}(s),y       \big)  \eta_i(s)dP^{\epsilon,\Delta}(u_1,u_2,y,s)        \bigg\|_\h\longrightarrow 0 \;,\;\text{as}\;\epsilon\to 0.
\end{aligned}
\end{equation*}
\end{lem}
\begin{proof}
Recall that for each fixed $x,y\in\h$, $D_xF(x,y)\in\mathscr{L}(\h)$ with \begin{equation}\label{unifbnd}\sup_{x,y\in\h}\big\|D_xF(x,y)\big\|_{\mathscr{L}(\h)}\leq\|\partial_{\mathrm{x}}f\|_{\infty}<\infty.\end{equation} By virtue of the Skorokhod representation theorem it follows that  $\pr$-a.s.
\begin{equation*}
\begin{aligned}
\sup_{t\in[0,T]}\bigg\|&\frac{1}{\Delta}\int_{0}^{t}\int_{s}^{s+\Delta} S_1(t-s)D_xF\big(\bar{X}(r), Y^{\epsilon,u^\epsilon}(r)\big)\big( \eta^{\epsilon,u^\epsilon}(s)- \eta_i(s)\big)drds\bigg\|_\h\\&
\leq \frac{C}{\Delta}\Delta T\sup_{x,y\in\h}\big\|D_xF(x,y)\big\|_{\mathscr{L}(\h)}	\sup_{s\in[0,T]}\big\|\eta^{\epsilon,u^\epsilon}(s)- \eta_i(s)\|_\h\longrightarrow 0\;,\;\text{as}\;\epsilon\to 0.
\end{aligned}
\end{equation*}
Hence, it suffices to study the term
$$\frac{1}{\Delta}\int_{0}^{t}\int_{s}^{s+\Delta} S_1(t-s)D_xF\big(\bar{X}(r), Y^{\epsilon,u^\epsilon}(r)\big)\big(  \eta_i(s)\big)drds.$$
\noindent The rest of the proof is omitted as the arguments are identical to the ones used in the proof of Lemma 4.16 in \cite{WSS}.
\end{proof}

\begin{lem} \label{linocculem}
Let $i=1,2$, $T<\infty$ and assume that the pair $(\eta^{\epsilon,u^\epsilon}, P^{\epsilon,\Delta}) $ converges in distribution, in Regime $i$, to $(\eta_i,P_i)$ in $C([0,T];\h)\times\mathscr{P}(\h\times\h\times\h\times[0,T])$. Then the following limit is valid with probability $1$:
\begin{equation*}
\begin{aligned}
\sup_{t\in[0,T]}\bigg\|&
\int_{\h\times\h\times\h\times[0, t] }S_1(t-s) D_xF\big(\bar{X}(s),y      \big)  \eta_i(s)dP^{\epsilon,\Delta}(u_1,u_2,y,s) \\&-      \int_{\h\times\h\times\h\times[0, t] }S_1(t-s) D_xF\big(\bar{X}(s),y       \big)  \eta_i(s)dP_i(u_1,u_2,y,s) \bigg\|_{\h}\longrightarrow 0 \;,\;\text{as}\;\epsilon\to 0.
\end{aligned}
\end{equation*}
\end{lem}
\begin{proof}
The argument is identical to the proof of Lemma 4.15 in \cite{WSS}. In fact, the present setting is even simpler since  the family  $\big\{D_xF\big(x,y        \big)\big\}_{x,y\in\h}\subset\mathscr{L}(\h)$  is uniformly bounded in the operator norm topology (see \eqref{unifbnd}).
\end{proof}

\noindent Combining Lemmas \ref{Taylorlem}, \ref{limcontlem}, \ref{lincont2lem} and \eqref{linocculem} we obtain the following:
\begin{prop}\label{linlim}
Let $i=1,2$, $T<\infty$ and assume that the pair $(\eta^{\epsilon,u^\epsilon}, P^{\epsilon,\Delta}) $ converges in distribution, in Regime $i$, to $(\eta_i,P_i)$ in $C([0,T];\h)\times\mathscr{P}(\h\times\h\times\h\times[0,T])$. Then the following limit is valid with probability $1$:
\begin{equation*}
\begin{aligned}
\lim_{\epsilon\to 0}\sup_{t\in[0, T]}\bigg\|\frac{1}{\sqrt{\epsilon}h(\epsilon)}&\int_{0}^{t}S_1(t-s)\big[F\big(\bar{X}(s) +\sqrt{\epsilon}h(\epsilon)\eta^{\epsilon,u^\epsilon}(s), Y^{\epsilon,u^\epsilon}(s) \big)-F\big(\bar{X}(s), Y^{\epsilon,u^\epsilon}(s) \big)\big]ds \\&
-   \int_{\h\times\h\times[0, t] }S_1(t-s) D_xF\big(\bar{X}(s),y        \big)  \eta_i(s)dP_i(u_1,u_2,y,s)\bigg\|_\h=0.
\end{aligned}
\end{equation*}
\end{prop}

\noindent  Regarding the averaging of the term $II^{\epsilon,u}$, first note that $X^{\epsilon,u}=\bar{X}+\sqrt\epsilon h(\epsilon)\eta^{\epsilon,u}$ and by the Skorokhod representation theorem $\eta^{\epsilon,u}\rightarrow \eta_i$ in $C([0,T];\h)$ with probability $1$. Using the latter along with the uniform integrability of the occupation measures (see Lemma \ref{UILem}) and the fact that, for each $t>0, x,y\in\h$, the operator $u\mapsto S_1(t)\Sigma(x,y)u$ is compact, we can follow the proofs of lemmas 4.15, 4.16 of \cite{WSS} verbatim to show Proposition \ref{IIlim} below.
\begin{prop}\label{IIlim} Let $i=1,2$, $T<\infty$ and assume that the pair $(\eta^{\epsilon,u^\epsilon}, P^{\epsilon,\Delta})$ converges in distribution, in Regime $i$, to $(\eta_i,P_i)$ in $C([0,T];\h)\times\mathscr{P}(\h\times\h\times\h\times[0,T])$. Then  the following limit is valid with probability $1$:
\begin{equation}
\hspace*{-0.7cm}
\begin{aligned}
\lim_{\epsilon\to 0}\sup_{t\in[0, T]}\bigg\|\int_{0}^{t}&S_1(t-s)\Sigma\big(\bar{X}(s), Y^{\epsilon,u}(s)\big)u^\epsilon_1(s)ds\\& -\ \int_{\h\times\h\times\h\times[0,t] }S_1(t-s)\Sigma\big(\bar{X}(s), y\big)u_1 dP_i(u_1,u_2,y,s)\bigg\|_{\h}=0.
\end{aligned}
\end{equation}
\end{prop}

\noindent It remains to study the limiting behavior of the term $IV^{\epsilon,u}$ in \eqref{etaspacetimedec}. To this end, let us  set $\theta=0,s=0$ in \eqref{Itoeta3}. In view of this decomposition, along with Lemmas \ref{IV1bnds}- \ref{IVbnds}, we see that for all $\epsilon>0$ there exists $n=n(\epsilon)>0$ and $\epsilon_0>0$ such that for all $\epsilon<\epsilon_0$
\begin{equation}\label{IVlim}
\begin{aligned}
\ex\sup_{t\in[0,T]}\sup_{\chi\in B_\h}&\bigg|IV^{\epsilon,u}(0,t,0,\chi)- \frac{\sqrt{\delta}}{\sqrt{\epsilon}}\int_{0}^{t}\langle S_1(t-z)\Psi^\epsilon_2\big(\bar{X}(z),Y_{n(\epsilon)}^{\epsilon,u}(z)\big)  u_{2,n(\epsilon)}(z) ,\chi\rangle_\h dz\bigg|\\&=
\ex\sup_{t\in[0,T]}\sup_{\chi\in B_\h}\big|IV^{\epsilon,u}(0,t,0,\chi)- IV_5^{\epsilon,u}(0,t,n(\epsilon),0, \chi)\big|<\epsilon.
\end{aligned}
\end{equation}
Thus, it suffices to study the term
\begin{equation*}\label{Poissonterm}
\frac{\sqrt{\delta}}{\sqrt{\epsilon}}\int_{0}^{t}S_1(t-z)\Psi^\epsilon_2\big(\bar{X}(z),Y_n^{\epsilon,u}(z)\big)  u_{2,n}(z) dz.
\end{equation*}
\noindent In fact, since for all $T>0$ we have $\|u_{2,n}-u_2\|_{L^2([0,T];\h)}\rightarrow 0$, $\pr$-a.s. and $\|\Psi^\epsilon_2(x,y)\big\|_{\mathscr{L}(\h)}\leq C/\ell$ uniformly in $x,y$  (see \eqref{psinorm}) we can directly work with
\begin{equation}\label{Poissontermnon}
\gamma_i\int_{0}^{t}S_1(t-z)\Psi^\epsilon_2\big(\bar{X}(z),Y_n^{\epsilon,u}(z)\big)  u_2(z) dz.
\end{equation}
where $\gamma_i=\lim_{\epsilon\to0}\sqrt{\delta/\epsilon}$ in Regime $i$. \noindent First, we need to find the limit of the operator-valued map $\Psi^\epsilon_2$ as $\epsilon\to 0$. In view of \eqref{Riesz} and estimates \eqref{psinorm} we have that, for all $x,\chi,v\in \h$ and $y\in Dom(A_2)$,
\begin{equation*}
\begin{aligned}
\blangle\Psi^\epsilon_2\big(x,y\big)v,    &    \chi\brangle_{\h}=\blangle D_y\Phi^\epsilon_\chi\big(x,y\big),v\brangle_\h,
\end{aligned}
\end{equation*}
where $D_y\Phi_\chi^\epsilon$ is the partial Fr\'echet derivative of the solution of the Kolmogorov equation \eqref{Kolmeq}. Recall that the latter is explicitly given by \eqref{Feynman}. Hence we can write
\begin{equation}\label{feynmander1}
\begin{aligned}
\blangle\Psi^\epsilon_2\big(x,y\big)v, \chi\brangle_{\h}&=
\int_{0}^{\infty}e^{-c(\epsilon)t}D_yP^x_t\big( \langle F\big(x,y\big)-\bar{F}(x),\chi  \rangle_\h     \big)(v)dt \\&=\int_{0}^{\infty}e^{-c(\epsilon)t}D_y\ex\big[ \langle F\big(x,Y^{x,y}(t)\big)-\bar{F}(x),\chi  \rangle_\h\big](v) dt,
\end{aligned}\end{equation}
\noindent where $P_t^x$ denotes the transition semigroup corresponding to the fast process $Y^{x,y}$ (see \eqref{Yfrozen}, \eqref{markovsemi}).
Now, for each fixed $x\in\h$, the map $$ \h\ni y\longmapsto \langle F(x,y),\chi\rangle_\h\in\R$$ is Fr\'echet differentiable with $$D_y \langle F(x,y),\chi\rangle_\h(v)=\langle D_yF(x,y)\chi,v\rangle_\h,$$ along the direction of any $v\in\h$. Therefore, we can differentiate under the sign of expectation and use the chain rule for Fr\'echet differentials to obtain
\begin{equation}\label{chainrule}
D_y\ex\big[ \blangle F\big(x,Y^{x,y}(t)\big)-\bar{F}(x),\chi  \brangle_\h\big](v)
=\ex \blangle  D_yF\big(x,Y^{x,y}(t)\big)\chi, D_yY^{x,y}(t)v\brangle_\h.
\end{equation}
In view of the latter, \eqref{feynmander1} yields
\begin{equation}\label{Feynmander}
\begin{aligned}
\blangle\Psi^\epsilon_2\big(x,y\big)v, &\chi\brangle_{\h}
=\int_{0}^{\infty}e^{-c(\epsilon)t}\ex \blangle  D_yF(x,Y^{x,y}(t))\chi, D_yY^{x,y}(t)v\brangle_\h dt.
\end{aligned}
\end{equation}
\noindent Under Hypothesis \ref{A2a}, the following lemma addresses the limiting behavior of $\Psi^\epsilon_2$ in \eqref{Poissontermnon} as the correction term in the Kolmogorov equation vanishes.

\begin{lem}\label{epsilondeplem} Let $T<\infty$ and define a map
$$ \h\times\h\ni(x,y)\longmapsto \Psi^0_2\big(x,y\big)\in \mathscr{L}\big(\h\big)$$
by
\begin{equation}\label{psi0}
\blangle\Psi^0_2\big(x,y\big)v, \chi\brangle_{\h}:=\int_{0}^{\infty}\ex\blangle D_yF(x,Y^{x,y}(t))\chi, D_yY^{x,y}(t)v\brangle_\h dt \;,\;\chi,v\in\h.
\end{equation}
\noindent The following limit is valid $\pr$-almost surely:
\begin{equation*}
\label{epsilondep}
\begin{aligned}
&\lim_{\epsilon\to 0 }\sup_{t\in[0,T]}\bigg\|\int_{0}^{t}S_1(t-z)\Psi^\epsilon_2\big(\bar{X}(z),Y_n^{\epsilon,u}(z)\big)  u_2(z) dz-\int_{0}^{t}S_1(t-z)\Psi^0_2\big(\bar{X}(z),Y_n^{\epsilon,u}(z)\big)  u_2(z) dz\bigg\|_{\h}=0.\end{aligned}
\end{equation*}

\end{lem}

\begin{proof} Let $\chi\in \h$ and $v\in\h$. Under our dissipativity assumptions, the $y$-Fr\'echet derivative
of $Y^{x,y}$ at the point $y$ and along the direction $v$ satisfies
\begin{equation}\label{yder}
\sup_{x,y\in\h}\big\|D_yY^{x,y}(t)v\big\|_\h\leq e^{-\ell t}\|v\|_\h\;,\;\pr-\text{a.s.}\;,
\end{equation}
\noindent where $\ell=\frac{\lambda-L_g}{2}$ (see 3.7 in \cite{cerrai2009averaging}). Hence,
\begin{equation}\label{dombnd}
\begin{aligned}
\sup_{\epsilon>0}\big|\blangle\Psi^\epsilon_2\big(x,y\big)v, \chi\brangle_{\h}\big|
&\leq\sup_{\epsilon>0}\int_{0}^{\infty}e^{-c(\epsilon)t}\ex  \big\| D_yF(x,Y^{x,y}(t))\chi\big\|_\h \big\|D_yY^{x,y}(t)v\big\|_\h dt\\&\leq \|\partial_{\mathrm{y}}f\|_\infty\|\chi\|_\h\|v\|_\h\sup_{\epsilon>0}\int_{0}^{\infty}e^{-c(\epsilon)t}e^{-\ell t}dt\\& \leq C_f\|\chi\|_\h\|v\|_\h \int_{0}^{\infty}e^{-\ell t}dt<\infty.
\end{aligned}
\end{equation}
\noindent An application of the Dominated Convergence theorem yields that for each fixed $x,y\in\h$
\begin{equation*}
\begin{aligned}
\lim_{\epsilon\to 0}\blangle\Psi^\epsilon_2\big(x,y\big)v, \chi\brangle_{\h}
&=\int_{0}^{\infty}\lim_{\epsilon\to 0}e^{-c(\epsilon)t}\ex\blangle D_yF(x,Y^{x,y}(t))\chi, D_yY^{x,y}(t)v\brangle_\h dt\\&= \int_{0}^{\infty}\ex\blangle D_yF(x,Y^{x,y}(t))\chi, D_yY^{x,y}(t)v\brangle_\h dt\\&= \blangle\Psi^0_2\big(x,y\big)v, \chi\brangle_{\h}.
\end{aligned}
\end{equation*}
\noindent In fact, estimate \eqref{dombnd} is uniform in $x,y$ and $\chi, v\in B_\h$ hence we obtain
$$ \sup_{x,y\in\h}\big\| \Psi^\epsilon_2\big(x,y\big)- \Psi^0_2\big(x, y\big)\big\|_{\mathscr{L}(\h)}\longrightarrow 0\;,\;\text{as}\;\epsilon\to 0.$$
The proof is complete. \end{proof}

\noindent To proceed in finding the averaging limit of $$\int_{0}^{t}S_1(t-z)\Psi^{0}_2\big(\bar{X}(z),Y_n^{\epsilon,u}(z)\big)  u_2(z) dz,$$
we need to establish uniform continuity properties of the map
$(x,y)\mapsto \Psi_2^{0}(x,y)$.
In view of \eqref{Feynmander}, this is related to the continuity of the map
$$ x\longmapsto D_yP^x_t\big[ \langle F(x,\cdot)-\bar{F}(x),\chi  \rangle_\h     \big](y)(v)=D_y\ex \langle F(x,Y^{x,y}(t))-\bar{F}(x),\chi  \rangle_\h(v), $$
for each fixed $t>0,y, v\in\h$. This is done in the next two lemmas. Note that, in order to obtain continuity properties of $D_yY^{x,y}$ with respect to $x,y,$ we need to assume the stronger dissipativity from Hypothesis \ref{A2c}.

\begin{lem}\label{varicontlem}
Let $t>0$, $v,y_1,y_2, x_1, x_2\in\h$ and $
\omega=\frac{\lambda-3L_g}{2}>0$
as in Hypothesis \ref{A2c}. Under Hypotheses \ref{A2b} and \ref{A2c}
there exists $C>0$ independent of $t$, such that
\begin{equation}\label{Zinfty}
(i)\quad\quad\quad \sup_{x,y\in\h}\big\|D_yY^{x,y} (t)v\big\|_{ L^\infty(0,L)}\leq C (t\wedge 1)^{-\frac{1}{4}}e^{-\ell t} \|v\|_\h.
\end{equation}
Moreover, for each $t\geq 0, v\in\h$ the maps $x,y\mapsto D_yY^{x,y} (t)$ are Lipschitz continuous with
\begin{equation}\label{varicont}
(ii)\quad\quad \big\|D_yY^{x_1,y} (t)v-D_yY^{x_2,y} (t)v\big\|_{ \h}\leq C (1+t)e^{-\omega t}\|v\|_\h\|x_1-x_2\|_\h
\end{equation}
and
\begin{equation}\label{varicont2}
(iii)\quad\quad \big\|D_yY^{x,y_1} (t)v-D_yY^{x,y_2} (t)v\big\|_{ \h}\leq C(1+t)e^{-\omega t}\|v\|_\h\|y_1-y_2\|_\h.
\end{equation}
\end{lem}

\begin{proof}
(i) For $x\in\h$, the first-order derivative $D_yY^{x,y} (t)v$ at the point $y\in \h$ and along the direction $v\in\h$ solves the \textit{first variation equation}
\begin{equation}\label{1stvareq}
\left\{
\begin{aligned}
&\partial_t{Z^v_{x,y}(t)}= A_2 Z^v_{x,y}(t)+ D_yG(x,y\big)Z^v_{x,y}(t)\;\;,\; t>0 \\&
Z^v_{x,y}(0)=v\in\h.
\end{aligned}\right.
\end{equation}
\noindent Under our dissipativity assumptions it follows that for all $p\geq 1$, $Z^v_{x,y}(t)\in L^p(0,L) $, $\pr$-a.s. and for $p=2$ we have
\begin{equation}\label{varbound}
\sup_{x,y\in\h}\big\|Z^v_{x,y}(t)\big\|_{\h}\leq C e^{-\ell t}\|v\|_{\h}\;,
\end{equation}
for all $t>0$, where $\ell=\frac{\lambda-L_g}{2}>0$ (see eg (3.7) in \cite{cerrai2009averaging}). For a proof of \eqref{varbound} we refer the reader to \cite{cerrai2001second}, Prop. 4.2.1. In order to prove \eqref{Zinfty} we use the mild formulation of \eqref{1stvareq} along with \eqref{varbound} and  the ultracontractivity of $S_2$ (see \eqref{Lpsmoothing}) to obtain
\begin{equation*}
\begin{aligned}
\big\|Z^v_{x,y}(t)\|_{L^\infty(0,L)}&\leq \|S_2(t)v\|_{L^\infty(0,L)}+  \int_{0}^{t}\big\|S_2(t-s)D_yG(x,y)	Z^v_{x,y}(s)\big\|_{L^\infty(0,L)}ds\\&\leq C t^{-\frac{1}{4}}\|v\|_{\h}+ C\int_{0}^{t}(t-s)^{-\frac{1}{4}}\big\|D_yG(x,y)	Z^v_{x,y}(s)\big\|_{\h}ds\\&
\leq C t^{-\frac{1}{4}}\|v\|_{\h}+ CL_g\int_{0}^{t}(t-s)^{-\frac{1}{4}}e^{-\ell s}\|v\|_\h ds.
\end{aligned}
\end{equation*}
\noindent Hence, for $t\leq1$ we have
\begin{equation}\label{Zinftysmall}
\begin{aligned}
\big\|Z^v_{x,y}(t)\|_{L^\infty(0,L)}&\leq Ct^{-\frac{1}{4}}\|v\|_\h.
\end{aligned}
\end{equation}
\noindent As for $t> 1$ we use the latter along with the linearity of \eqref{1stvareq} to deduce that
\begin{equation}\label{Zinftylarge}
\begin{aligned}
\big\|Z^v_{x,y}(t)\|_{L^\infty(0,L)}&=\|Z_{x,y}^{Z^v_{x,y}(t-1)}(1)\|_{L^\infty(0,L)}\leq C 1^{-\frac{1}{4}}\big\|Z^v_{x,y}(t-1)\big\|_\h\leq C e^{-\ell(t-1)}\|v\|_\h,
\end{aligned}
\end{equation}
\noindent where we invoked \eqref{varbound} once more to obtain the last inequality. Combining \eqref{Zinftysmall} and \eqref{Zinftylarge}, we get that \eqref{Zinfty} holds.

\noindent (ii) From the mild formulation of \eqref{1stvareq} we have
\begin{equation*}
\begin{aligned}
Z^v_{x_1,y}(t)-Z^v_{x_2,y}(t)&=\int_{0}^{t}S_2(t-s)\big[D_yG(x_1,y)	Z^v_{x_1,y}(s)-D_yG(x_2,y)Z^v_{x_2,y}(s)\big]ds\\&= \int_{0}^{t}S_2(t-s)D_yG(x_1,y)\big[	Z^v_{x_1,y}(s)-	Z^v_{x_2,y}(s)\big]ds\\&+ \int_{0}^{t}S_2(t-s)\big[D_yG(x_1,y)-D_yG(x_2,y)\big]Z^v_{x_2,y}(s)ds.
\end{aligned}
\end{equation*}
\noindent Using \eqref{Zinfty} on the second term we estimate
\begin{equation*}
\begin{aligned}
\big\|Z^v_{x_1,y}(t)-Z^v_{x_2,y}(t)\big\|_{\h}&\leq L_g \int_{0}^{t}e^{-\lambda(t-s)}\big\|Z^v_{x_1,y}(s)-Z^v_{x_2,y}(s)\big\|_{\h} ds\\&+  \int_{0}^{t}e^{-\lambda(t-s)}\|D_yG(x_1,y)-D_yG(x_2,y)\|_{\mathscr{L}(L^\infty(0,L);\h)} \big\|Z^v_{x,y}(s)\|_{L^\infty(0,L)} ds\\&
\leq L_g \int_{0}^{t}e^{-\lambda(t-s)}\big\|Z^v_{x_1,y}(s)-Z^v_{x_2,y}(s)\big\|_{\h} ds\\&+  Ce^{-\lambda t}\|v\|_\h\|D_yG(x_1,y)-D_yG(x_2,y)\|_{\mathscr{L}(L^\infty(0,L);\h)} \int_{0}^{t}(s\wedge 1)^{-\frac{1}{4}} e^{(\lambda-\ell)s}ds.
\end{aligned}
\end{equation*}
\noindent An application of the mean value inequality then yields
\begin{equation*}
\begin{aligned}
\big\|Z^v_{x_1,y}(t)-Z^v_{x_2,y}(t)\big\|_{\h}&\leq L_g e^{-\lambda t}\int_{0}^{t}e^{\lambda s}\big\|Z^v_{x_1,y}(s)-Z^v_{x_2,y}(s)\big\|_{\h}ds\\&+  C_ge^{-\lambda t}\|x_1-x_2\|_\h \|v\|_{\h} \int_{0}^{t}(s\wedge 1)^{-\frac{1}{4}}e^{(\lambda-\ell)s} ds
\end{aligned}
\end{equation*}
and $\lambda-\ell=\frac{\lambda+L_g}{2}>0$. Hence
\begin{equation}
\begin{aligned}
e^{\lambda t}\big\|Z^v_{x_1,y}(t)-Z^v_{x_2,y}(t)\big\|_{\h}&\leq L_g \int_{0}^{t}e^{\lambda s}\big\|Z^v_{x_1,y}(s)-Z^v_{x_2,y}(s)\big\|_{\h}ds\\&+  C\|x_1-x_2\|_\h \|v\|_{\h} e^{(\lambda-\ell)t}\big[ t^{\frac{3}{4}}\mathds{1}_{(0,1)}(t)+(1+t)\mathds{1}_{[1,\infty)}(t)\big]
\end{aligned}
\end{equation}
and the second term on the right-hand side is increasing in $t$. Invoking Gr\"onwall's inequality we obtain
\begin{equation*}
\begin{aligned}
e^{\lambda t}\big\|Z^v_{x_1,y}(t)-Z^v_{x_2,y}(t)\big\|_{\h}\leq C \big(1+ t\big)e^{(L_g+\lambda-\ell)t}\|x_1-x_2\|_\h \|v\|_{\h}
\end{aligned}
\end{equation*}
and $L_g-\ell=L_g-\frac{\lambda-L_g}{2}=-\omega$ is negative in view of \eqref{extradiss}. The proof of \eqref{varicont} is complete.\\ \\
\noindent (iii) Similarly, we can write
\begin{equation*}
\begin{aligned}
Z^v_{x,y_1}(t)-Z^v_{x,y_2}(t) &=\int_{0}^{t}S_2(t-s)\big[D_yG(x,y_1)	Z^v_{x,y_1}(s)-D_yG(x,y_2)Z^v_{x,y_2}(s)\big]ds\\&= \int_{0}^{t}S_2(t-s)D_yG(x,y_1)\big[	Z^v_{x,y_1}(s)-	Z^v_{x,y_2}(s)\big]ds\\&+ \int_{0}^{t}S_2(t-s)\big[D_yG(x,y_1)-D_yG(x,y_2)\big]Z^v_{x,y_2}(s)ds.
\end{aligned}
\end{equation*}
Using an identical argument as in (i), the result follows by Gr\"onwall's inequality.
\end{proof}

\begin{lem}\label{semicont} Let $t>0,\chi, x_1,x_2,y_1,y_2,v\in\h $ and $c(t):=1+t+(t\wedge 1)^{-\frac{1}{4}}$ . Under Hypotheses \ref{A2a}-\ref{A2c} and for all $x,y\in\h$ we have \\
\noindent $(i)$
\begin{equation*}
\begin{aligned}
&\big|\ex \big[D_y\blangle F\big(x_1,Y^{x_1,y}(t)\big),\chi  \brangle_\h(v) - D_y\blangle F(x_2,Y^{x_2,y}(t)),\chi  \brangle_\h(v) \big]\big|\leq C\|\chi\|_{\h}\|v\|_{\h}\|x_1-x_2\|_\h c(t)e^{-\omega t}\;,
\end{aligned}
\end{equation*}

\begin{equation*}
(ii)\quad\quad	\big|\ex \big[D_y\blangle F\big(x,Y^{x,y_1}(t)\big),\chi  \brangle_\h(v) - D_y\blangle F(x,Y^{x,y_2}(t)),\chi  \brangle_\h(v) \big]\big|\leq C \|\chi\|_{\h}\|v\|_{\h}\|y_1-y_2\|_\h c(t)e^{-\omega t}\;,
\end{equation*}
with $\omega$ as in \eqref{extradiss}.
\end{lem}
\begin{proof}
$(i)$ Let $Z^v_{x,y}(t):= D_yY^{x,y}(t)v$ as in the previous lemma. In view of \eqref{chainrule},
\begin{equation*}
\begin{aligned}
\ex \big[D_y\blangle &F\big(x_1,Y^{x_1,y}(t)\big),\chi  \brangle_\h(v) - D_y\blangle F(x_2,Y^{x_2,y}(t)),\chi  \brangle_\h (v)\big]
\\&	=\ex \blangle  D_yF\big(x_1,Y^{x_1,y}(t)\big)\chi, Z^v_{x_1,y}(t)-Z^v_{x_2,y}(t)\brangle_\h \\& + \ex \blangle  D_yF\big(x_1,Y^{x_1,y}(t)\big)\chi-D_yF\big(x_2,Y^{x_2,y}(t)\big)\chi, Z^v_{x_2,y}(t)\brangle_\h=: I_1+I_2.
\end{aligned}
\end{equation*}
\noindent From \eqref{varicont} we obtain
\begin{equation}\label{I1}
\begin{aligned}
\big|I_1\big|&\leq  \big\|   D_yF\big(x,Y^{x_1,y}(t)\big)\chi\big\|_{ L^2(\Omega\times(0,L))}\|Z^v_{x_1,y}(t)-Z^v_{x_2,y}(t)\big\|_{L^2(\Omega\times(0,L))}
\\& \leq C (1+t)e^{-\omega t}\|\partial_{\mathrm{y}}f\big\|_{\infty} \|\chi\|_{\h}\|v\|_\h\|x_1-x_2\|_\h.
\end{aligned}
\end{equation}
\noindent As for $I_2$, we apply  \eqref{Zinfty} along with the mean value inequality to deduce that
\begin{equation}\label{I2}
\begin{aligned}
\big| I_2\big|&\leq \ex\bigg[     \big\|Z^v_{x_2,y}(t)\big\|_{L^{\infty}(0,L)}\big\|D_yF\big(x_1,Y^{x_1,y}(t)\big)\chi-D_yF\big(x_2,Y^{x_2,y}(t)\big)\chi\|_{L^{1}(0,L)}    \bigg] \\&\leq C   (t\wedge 1)^{-\frac{1}{4}}e^{-\ell t}\|v\|_{\h}\|\chi\|_\h\big(   \|\partial^2_{\mathrm{x}\mathrm{y}}f\big\|_{\infty}\|x_1-x_2\|_\h+ \|\partial^2_{\mathrm{y}\mathrm{y}}f\big\|_{\infty}\ex \big\|Y^{x_1,y} (t)-Y^{x_2,y} (t)\big\|_\h\big)\\&\leq C_f(t\wedge 1)^{-\frac{1}{4}}e^{-\ell t}\|v\|_{\h}\|\chi\|_\h\big(\|x_1-x_2\|_\h+\ex\sup_{x,y\in\h}\big\|D_xY^{x,y}(t)\big\|_{\mathscr{L}(\h)}\|x_1-x_2\|_\h\big)\\&\leq C_f  (t\wedge 1)^{-\frac{1}{4}}e^{-\ell t}\|v\|_{\h}\|\chi\|_\h\|x_1-x_2\|_\h\big( 1+e^{-\ell t}        \big),		\end{aligned}
\end{equation}
where we invoked (3.9) in \cite{cerrai2009averaging} to obtain the last line.
Combining the latter with \eqref{I1} concludes the argument. Finally, $(ii)$ follows from a similar argument along with estimate \eqref{varicont2}.\end{proof}
\begin{cor}\label{Psicont} Let $x,x_1,x_2,y,y_1,y_2\in\h$. There exists $C>0$ such that \\ (i) The  $\mathscr{L}(\h)$-valued map $x\mapsto\Psi^0_2(x, y)$ is $C$-Lipschitz continuous uniformly in $y$ i.e.
	\begin{equation*}
	\big\| \Psi^0_2(x_1,y)- \Psi^0_2(x_2,y)\big\|_{\mathscr{L}(\h)}\leq C\big\|x_1-x_2\big\|_\h\;.
	\end{equation*}
	(ii) The  $\mathscr{L}(\h)$-valued map $y\mapsto\Psi^0_2(x, y)$ is $C$-Lipschitz continuous uniformly in $x$ i.e.
	\begin{equation}\label{Psiylip}
	\big\| \Psi^0_2(x,y_1)- \Psi^0_2( x,y_2)\big\|_{\mathscr{L}(\h)}\leq C\big\|y_1-y_2\big\|_\h\;.
	\end{equation}
\end{cor}
\begin{proof}(i)
	\noindent From \eqref{psi0} and Lemma \ref{semicont}(i) it follows that
	\begin{equation*}
	\begin{aligned}
	&\sup_{v,\chi\in B_\h}\big|\blangle  \Psi^0_2( x_1,y)v- \Psi^0_2(x_2,y)v, \chi\brangle_{\h}\big| \\&\leq\int_{0}^{\infty} \sup_{v,\chi\in B_\h}\big| D_yP^{x_1}_t\big[ \langle F(x_1,y)-\bar{F}(x_1),\chi  \rangle_\h     \big](v)-D_yP^{x_2}_t\big[ \langle F(x_2,y)-\bar{F}(x_2),\chi  \rangle_\h     \big](v)\big| dt\\&
	\leq   C\|x_1-x_2\|_\h \int_{0}^{\infty}c(t)e^{-\omega t}dt =C\|x_1-x_2\|_\h\int_{0}^{\infty}\big[1+t+(t\wedge 1)^{-\frac{1}{4}}\big]e^{-\omega t}dt ,
	\end{aligned}
	\end{equation*}
	and the last integral is finite. As for $(ii)$, the estimate follows from an identical argument along with Lemma \ref{semicont}$(ii)$.
\end{proof}

\noindent The next lemma is analogous to Lemma \ref{limcontlem} that was proved for $I^{\epsilon,u}$.

\begin{lem}\label{Kolmocontlem} For $\Delta>0$ as in (\ref{Delta1}) and $T<\infty$ we have
	\begin{equation*}\label{yderave}
	\begin{aligned}
	\sup_{n\in\N}\sup_{t\in[0, T] }&\bigg\|\frac{1}{\Delta}\int_{0}^{t}\int_{s}^{s+\Delta}S_1(t-s)\Psi^0_2(\bar{X}(s),Y_n^{\epsilon,u}(r)\big)  u_2(r) drds\\&-		\frac{1}{\Delta}\int_{0}^{t}\int_{s}^{s+\Delta}S_1(t-s)\Psi^0_2(\bar{X}(r),Y_n^{\epsilon,u}(r)\big)  u_2(r) drds\bigg\|_{\h}\longrightarrow 0 \;,\;\text{as}\;\epsilon\to 0\;\;, \pr-\text{a.s.}
	\end{aligned}
	\end{equation*}
\end{lem}
\begin{proof}
	The proof is a direct application of Corollary \ref{Psicont}. In particular, we have
	\begin{equation*}
	\begin{aligned}
	\bigg\|\int_{0}^{t}\int_{s}^{s+\Delta}&S_1(t-s)\Psi^0_2\big(\bar{X}(s),Y_n^{\epsilon,u}(r)\big)   -\Psi^0_2\big(\bar{X}(r),Y_n^{\epsilon,u}(r)\big) \big] u_2(r) drds\bigg\|_{\h}\\&
	\leq C \int_{0}^{t}\int_{s}^{s+\Delta}\big\| \Psi^0_2\big(\bar{X}(s),Y_n^{\epsilon,u}(r)\big) -	\Psi^0_2\big(\bar{X}(r),Y_n^{\epsilon,u}(r)\big)\big\|_{\mathscr{L}(\h)}\big\|u_2(r)\|_\h drds\\&
	\leq C\int_{0}^{t}\int_{s}^{s+\Delta}\big\| \bar{X}(s)-\bar{X}(r)\big\|_\h \big\|u_2(r)\|_\h dr ds\\&\leq C\big[ \bar{X}\big]_{C^\theta([0,T+1])} \int_{0}^{t}\int_{s}^{s+\Delta}|s-r|^\theta\big\|u_2(r)\|_\h dr ds
	\end{aligned}
	\end{equation*}
		\begin{equation*}
	\begin{aligned}
	&
	\leq C(1+\|x_0\|_{H^a})\Delta^\theta \int_{0}^{t}\int_{s}^{s+\Delta}\big\|u_2(r)\|_\h dr ds\\&
	\leq C(1+\|x_0\|_{H^a})\Delta^{\theta+1}\int_{0}^{T+\Delta}\big\|u_2(s)\|_\h ds
	\leq  C_{T,N}(1+\|x_0\|_{H^a})\Delta^{\theta+1},
	\end{aligned}
	\end{equation*}
	where $\theta<\frac{1}{4}\wedge\frac{a}{2}$ and we used  \eqref{Schauderbar} to obtain the third inequality and the Cauchy-Schwarz inequality, along with fact that $u\in\mathcal{P}_N^T$, to obtain the last line.

   Therefore,
	\begin{equation*}
	\begin{aligned}
	\frac{1}{\Delta}\sup_{n\in\N,t\in[0,T]}&\bigg\|\int_{0}^{t}\int_{s}^{s+\Delta}S_1(t-z)\big[\Psi^0_2\big(\bar{X}(s),Y_n^{\epsilon,u}(r)\big)  u_2(r) -	\Psi^0_2\big(\bar{X}(r),Y_n^{\epsilon,u}(r)\big) \big] u_2(r) drds\bigg\|_{\h}\\&\leq C \Delta^\theta(1+\|x_0\|_{H^a}).
	\end{aligned}
	\end{equation*}
	 The proof is complete upon taking $\epsilon\to 0$.\end{proof}
\noindent For $n\in\N$  and $\Delta$ as in Definition \ref{Delta1}, define the \textit{projected} occupation measures
\begin{equation*}\label{projoccu}
\begin{aligned}
&P_n^{\epsilon,\Delta}(\Gamma_1\times\Gamma_2\times \Gamma_3\times \Gamma_4)= P^{\epsilon,\Delta}(\Gamma_1\times \Gamma_2\times P_n^{-1}\big(\Gamma_3\big)\times \Gamma_4 )\\&
=\frac{1}{\Delta}\int_{\Gamma_4}\int_{t}^{t+\Delta}\mathds{1}_{\Gamma_1}\big(u_1(s)\big)\mathds{1}_{\Gamma_2}\big(u_2(s)\big) \mathds{1}_{\Gamma_3}\big(Y_n^{\epsilon,u}(s)\big)dsdt,
\end{aligned}
\end{equation*}
$\Gamma_1\times\Gamma_2\times \Gamma_3\times \Gamma_4\in\mathcal{B}\big(  \h\times\h\times\h\times[0,T] \big)$ i.e. $P_n^{\epsilon,\Delta}$ is the push-forward of $P^{\epsilon,\Delta}$ induced by the $n$-dimensional orthogonal projection $P_n$ on the third marginal. It is straightforward to verify that $P_n^{\epsilon,\Delta}$ inherit the tightness and uniform integrability properties from the occupation measures $P^{\epsilon,\Delta}$ (see Lemmas \ref{Palaoglu} and \ref{UILem}). Moreover, for each $\epsilon>0$ there exists $n=n(\epsilon)>0$ large enough so that, after passing to subsequences, $P_n^{\epsilon,\Delta}$  and $P^{\epsilon,\Delta}$ share the same limit in distribution (denoted by $P_i$) as $\epsilon\to 0$ in the topology of weak convergence of measures on $\h\times\h\times\h\times[0,T]$.

Indeed, the class of Lipschitz-continuous functions $f\in C_b(\h\times\h\times\h\times [0,T])$ characterizes weak convergence of measures (see \cite{dupuis2011weak}, Remark A.3.5.) and for any such $f$ we fix $\epsilon>0$ and apply the dominated convergence theorem to obtain
\begin{equation*}
\begin{aligned}
\bigg|\int_{\h\times\h\times\h\times[0,T]}&f\big(u_1,u_2,y,t\big)dP_{n}^{\epsilon,\Delta}(u_1,u_2,y,t)-\int_{\h\times\h\times\h\times[0,T]}f\big(u_1,u_2,y,t\big)dP^{\epsilon,\Delta}(u_1,u_2,y,t)\bigg|\\&
=\bigg|\frac{1}{\Delta}\int_{0}^{T}\int_{t}^{t+\Delta}f\big(u_1^\epsilon(s),u_2^\epsilon(s),Y_{n}^{\epsilon,u^\epsilon}(s),t\big)-f\big(u_1^\epsilon(s),u_2^\epsilon(s),Y^{\epsilon,u^\epsilon}(s),t\big)dsdt\bigg|
\\&\leq \frac{1}{\Delta}\int_{0}^{T}\int_{t}^{t+\Delta}\big\| P_{n}Y^{\epsilon,u^\epsilon}(s)-Y^{\epsilon, u^\epsilon}(s)\big\|_\h dsdt\longrightarrow 0\;\;\text{as}\;n\to \infty.
\end{aligned}
\end{equation*}
\noindent Using the latter, along with Lemma \ref{Kolmocontlem}, we can now prove the following asymptotics:
\begin{lem}\label{controlocculem}Let $i=1,2$, $T>0$ and assume that the pair $(\eta^{\epsilon,u^\epsilon}, P^{\epsilon,\Delta})$ converges in distribution, in Regime $i$, to $(\eta_i,P_i)$ in $C([0,T];\h)\times\mathscr{P}(\h\times\h\times\h\times[0,T])$. Then  there exists $n=n(\epsilon)>0$ large enough, such that the following limits hold with probability $1$:
	\begin{equation}\label{Psiave1}
	\begin{aligned}
	\sup_{t\in[0,T]}\bigg\|&\int_{0}^{t}S_1(t-s)\Psi^0_2\big(\bar{X}(s),Y_n^{\epsilon,u}(s)\big)  u_2^\epsilon(s) ds\\&-\int_{\h\times\h\times\h\times[0,t] }S_1(t-s)\Psi^0_2\big(\bar{X}(s), y\big)u_2 dP^{\epsilon,\Delta}_n(u_1,u_2,y,s)\bigg\|_{\h}\longrightarrow 0\;,\;\text{as}\;\epsilon\to 0
	\end{aligned}
	\end{equation}	
	\noindent and
	\begin{equation}\label{Psiave2}
	\begin{aligned}
	\sup_{t\in[0,T]}\bigg\|&\int_{\h\times\h\times\h\times[0,t] }S_1(t-s)\Psi^0_2\big(\bar{X}(s), y\big)u_2 dP^{\epsilon,\Delta}_n(u_1,u_2,y,s) \\&-\int_{\h\times\h\times\h\times[0,t] }S_1(t-s)\Psi^0_2\big(\bar{X}(s), y\big)u_2 dP_i(u_1,u_2,y,s)\bigg\|_{\h}\longrightarrow 0\;,\;\text{as}\;\epsilon\to 0.
	\end{aligned}
	\end{equation}	
\end{lem}
\begin{proof} We start with \eqref{Psiave1}. Notice that
	\begin{equation*}
	\begin{aligned}
	&\int_{\h\times\h\times\h\times[0,t] }S_1(t-s)\Psi^0_2\big(\bar{X}(s), y\big)u_2 dP^{\epsilon,\Delta}_n(u_1,u_2,y,s)\\&=\int_{0}^{t}\int_{s}^{s+\Delta}S_1(t-s)\Psi^0_2\big(\bar{X}(s), Y_n^{\epsilon,u}(r)\big)u_2(r)dr ds.
	\end{aligned}
	\end{equation*}
	In view of Lemma \ref{Kolmocontlem} it is enough to study the term
	\begin{equation*}
	\begin{aligned}
	\int_{0}^{t}\int_{s}^{s+\Delta}S_1(t-s)\Psi^0_2\big(\bar{X}(r), Y_n^{\epsilon,u}(r)\big)u_2(r)dr ds.
	\end{aligned}
	\end{equation*}
	Changing the order of integration, the latter is equal to
	\begin{equation*}
	\begin{aligned}
	&\int_{0}^{\Delta}\int_{0}^{r}S_1(t-s)\Psi^0_2\big(\bar{X}(r), Y_n^{\epsilon,u}(r)\big)u_2(r)dsdr \\&
	+ \int_{\Delta}^{t}\int_{r-\Delta}^{r}S_1(t-s)\Psi^0_2\big(\bar{X}(r), Y_n^{\epsilon,u}(r)\big)u_2(r)dsdr \\&
	+	\int_{t}^{t+\Delta}\int_{r-\Delta}^{t}S_1(t-s)\Psi^0_2\big(\bar{X}(r), Y_n^{\epsilon,u}(r)\big)u_2(r) dsdr.
	\end{aligned}
	\end{equation*}
	The first and third terms in this expression converge to zero as  $\epsilon\to0$, so we only need to focus on the second term. In view of \eqref{sobcont},
	
	\begin{equation*}
	\begin{aligned}
	\bigg\|\int_{\Delta}^{t}\int_{r}^{r-\Delta}&S_1(t-s)\Psi^0_2\big(\bar{X}(r), Y_n^{\epsilon,u}(r)\big)u_2(r)dr ds-\int_{\Delta}^{t}S_1(t-r)\Psi^0_2\big(\bar{X}(r),Y_n^{\epsilon,u}(r)\big)  u_2(r) dr\bigg\|_\h\\&
	\leq \int_{\Delta}^{t}\bigg\|\frac{1}{\Delta}\int_{0}^{\Delta}S_1(s)ds-I\bigg\|_{\mathscr{L}(H^\theta;\h)}\big\|S_1(t-r)\Psi^0_2\big(\bar{X}(r),Y_n^{\epsilon,u}(r)\big)  u_2(r)\|_{H^\theta}dr\\&
	\leq \frac{C}{\Delta}\int_{\Delta}^{t}\bigg(\int_{0}^{\Delta}s^{\theta/2}ds\bigg) \big\|S_1(t-r)\Psi^0_2\big(\bar{X}(r),Y_n^{\epsilon,u}(r)\big)  u_2(r)\|_{H^\theta}dr.
	\end{aligned}
	\end{equation*}
	Finally, we invoke Lemma \ref{sigmacont}(ii) to conclude that
	\begin{equation*}
	\begin{aligned}
	\bigg\|\int_{\Delta}^{t}\int_{r}^{r-\Delta}&S_1(t-s)\Psi^0_2\big(\bar{X}(r), Y_n^{\epsilon,u}(r)\big)u_2(r)dr ds-\int_{\Delta}^{t}S_1(t-r)\Psi^0_2\big(\bar{X}(r),Y_n^{\epsilon,u}(r)\big)  u_2^\epsilon(r) dr\bigg\|_\h\\&
	\leq C_\theta\Delta^{\theta/2}\int_{\Delta}^{t} (t-r)^{-\rho}\big\|\Psi^0_2\big(\bar{X}(r),Y_n^{\epsilon,u}(r)\big)\big\|_{\mathscr{L}(\h)} \|u_2^\epsilon(r)\|_{\h}dr\\&
	\leq  C_\theta\Delta^{\theta/2}N\int_{\Delta}^{t}(t-r)^{-2\rho}dr,
	\end{aligned}
	\end{equation*}
	where $\rho>\theta+1/2$ and we used
	the Cauchy-Schwarz inequality to obtain the last line. Since $\theta$ can be chosen to be arbitrarily small, \eqref{Psiave1} follows.
	
	It remains to prove \eqref{Psiave2}. To this end, let $P_m^i$ denote orthogonal projection to an $m$-dimensional eigenspace of $A_1$. From a slight modification of Lemma \ref{sigmacont}(ii) we have
	\begin{equation}\label{compactrem}
	\begin{aligned}
	\big\|(I-P_m^1)S_1(t)\Psi^0_2\big(x, y\big)\big\|^2_{\mathscr{L}(\h)}&\leq C\|\Psi_2^{0}(x,y)\big\|_{\mathscr{L}(\h)}(t-s)^\rho e^{-\frac{\lambda t}{2}}\sum_{j=m+1}^{\infty} a^{-\rho}_{2,j}\\&\leq C(t-s)^\rho e^{-\frac{\lambda t}{2}}\sum_{j=m+1}^{\infty} a^{-\rho}_{2,j},
	\end{aligned}
	\end{equation}
	for some $\rho>1/2$. The last term on the right-hand side is the tail of a convergent sum. Thus, for fixed $t>0$, the operator   $u\mapsto S_1(t)\Psi^0_2\big(x, y\big)u$ is a uniform limit of finite-dimensional operators, hence a compact operator. As such, it is continuous from the weak topology of $\h$ to the norm topology of $\h$ and for each $k\in\N$ the real-valued map
	\begin{equation*}
	(s,y,u_2)\longmapsto\blangle S_1(t-s)\Psi^0_2\big(\bar{X}(s), y\big)u_2, e_{1,k}       \brangle_\h
	\end{equation*}
	is continuous in the WWNS topology on $\h\times\h\times\h\times[0,T]$. Appealing to the Skorokhod representation theorem once again, there exists $n(\epsilon)\in\N$ such that $P_{n(\epsilon)}^{\epsilon,\Delta}$ converges weakly to $P_i$ as $\epsilon\to0$ with probability $1$. Combining this with the uniform integrability of $P_{n(\epsilon)}^{\epsilon,\Delta}$ (see Lemma \ref{UILem}), we have that for each $m\in\N$,
	\begin{equation*}
	\begin{aligned}
	\bigg\|&\int_{\h\times\h\times\h\times[0,t] }P^1_mS_1(t-s)\Psi^0_2\big(\bar{X}(s), y\big)u_2 dP^{\epsilon,\Delta}_n(u_1,u_2,y,s) \\&-\int_{\h\times\h\times\h\times[0,t] }P^1_mS_1(t-s)\Psi^0_2\big(\bar{X}(s), y\big)u_2 dP_i(u_1,u_2,y,s)\bigg\|^2_{\h}\\&
	=\sum_{k=1}^{m}\bigg(   \int_{\h\times\h\times\h\times[0,t] }\blangle S_1(t-s)\Psi^0_2\big(\bar{X}(s), y\big)u_2, e_{1,k}\brangle_\h dP^{\epsilon,\Delta}_n(u_1,u_2,y,s) \\&-\int_{\h\times\h\times\h\times[0,t] }\blangle S_1(t-s)\Psi^0_2\big(\bar{X}(s), y\big)u_2, e_{1,k}\brangle_\h  dP_i(u_1,u_2,y,s)    \bigg)^2\longrightarrow 0
	\end{aligned}
	\end{equation*}
	as $\epsilon\to 0$. Finally, we use \eqref{compactrem}, \eqref{psinorm} to show that the remainders
	\begin{equation*}
	\bigg\|\int_{\h\times\h\times\h\times[0,t] }(I-P^1_m)S_1(t-s)\Psi^0_2\big(\bar{X}(s), y\big)u_2 dP^{\epsilon,\Delta}_n(u_1,u_2,y,s)\bigg\|^2_\h
	\end{equation*}
	are uniformly bounded in $\epsilon,t,n$ and small as $m\to\infty$. The proof is complete. 	\end{proof}

\noindent To conclude this section, we combine Lemma \ref{epsilondeplem}, Lemma \ref{Kolmocontlem} and Lemma \ref{controlocculem} to obtain the following, regarding the limiting behavior of the term $IV^{\epsilon,u}$ in \eqref{etaspacetimedec}:

 \begin{prop}\label{controlim} Let $i=1,2,\gamma_i$ as in \eqref{gammai} and $T<\infty$. Assume that the pair $(\eta^{\epsilon,u^\epsilon}, P^{\epsilon,\Delta})$ converges in distribution, in Regime $i$, to $(\eta_i,P_i)$ in $C([0,T];\h)\times\mathscr{P}(\h\times\h\times\h\times[0,T])$. Then  there exists $n=n(\epsilon)>0$ such that the following limit is valid with probability $1$:
	\begin{equation*}
	\begin{aligned}
	\lim_{\epsilon\to 0}\sup_{t\in[0, T] }\bigg\|\frac{\sqrt{\delta}}{\sqrt{\epsilon}}\int_{0}^{t}&S_1(t-s)\Psi^\epsilon_2\big(\bar{X}(s),Y_n^{\epsilon,u}(s)\big)  u_2^\epsilon(s) dz\\&-\gamma_i\ \int_{\h\times\h\times\h\times[0,t] }S_1(t-s)\Psi^0_2\big(\bar{X}(s), y\big)u_2 dP_i(u_1,u_2,y,s)\bigg\|_{\h}= 0.
	\end{aligned}
	\end{equation*}
\end{prop}

\subsection{Proof of Theorem \ref{viablim1}}\label{aveproof}

Let $i=1,2$. In this section we will show that the limiting pair $(\eta_i, P_i)$ in Regime $i$ is, with probability $1$, a viable pair in $\mathcal{V}_{(\Xi_i, \mu^{\bar{X}})}$. In particular, we shall show that  $(\eta_i, P_i)$ satisfies (i), (ii) and (iii) in Definition \eqref{viable}.

First, note that Propositions \ref{linlim},  \ref{IIlim}, \ref{controlim} from Section \ref{limpnts}, along with \eqref{IVlim}, imply that any sequence in $\{(\eta^{\epsilon,u}, P^{\epsilon, \Delta}:\epsilon\in(0,1), u\in\mathcal{P}_N^T\}$ has a subsequence that converges in distribution to a pair $(\eta_i, P_i)$. This pair satisfies the integral equation
\begin{equation*}
\small
\begin{aligned}
\eta_i(t)&= \int_{\h\times\h\times\h\times[0,t]}S_1(t-s)\bigg[  D_xF\big(\bar{X}(s),y\big)\eta_i(t)+ \Sigma(\bar{X}(s),y)u_1+\gamma_i\Psi^0_2\big(\bar{X}(s), y\big)u_2\bigg] dP_i(u_1,u_2,y,s)\\&=
\int_{\h\times\h\times\h\times[0,t]}S_1(t-s)\Xi_i\big( \eta_i(s),\bar{X}(s),y, u_1,u_2\big) dP_i(u_1,u_2,y,s)
\end{aligned}
\end{equation*}
\noindent with probability 1. Hence, $(\eta_i, P_i)$ satisfies \eqref{viable3}. As for \eqref{viable1}, the weak convergence of $P^{\epsilon,\Delta}$ to $P_i$ along with the uniform integrability of $P^{\epsilon,\Delta}$ (Lemma \ref{UILem}) imply the square integrability of the measures $P_i$.

Regarding \eqref{viableleb}, note that this property holds at the prelimit level. Since the map $t\mapsto  P_i(\h \times \h\times\h \times[0, t])$ is continuous and $P_i(\h \times \h\times\h \times\{ t\})=0$ the result follows as in the finite-dimensional case (see \cite{dupuis2012large}).

Finally, we verify the decomposition \eqref{viable2}. For this it suffices to show that the third and fourth marginals of $P_i$ are given by the product $d\mu^{\bar{X}(t)}\times dt$ of the local invariant measure and Lebesgue measure. Indeed, we shall show that for any $f\in C_b(\h)$,
\begin{equation*}
\int_{\h\times\h\times\h\times [0,T]}f(y)dP_i(u_1,u_2,y,t)=\int_{0}^{T}\int_{\h}f(y)d\mu^{\bar{X}(t)}(y)dt.
\end{equation*}

To this end, let $\tilde{Y}^{\epsilon}_u$ denote the uncontrolled fast process depending on the controlled slow process $X^{\epsilon,u}$, i.e. $\tilde{Y}^{\epsilon}_u$ solves

\begin{equation*}
d\tilde{Y}^{\epsilon}_u(t)= \frac{1}{\delta}\big[A_2 \tilde{Y}^{\epsilon}_u(t)+ G\big( X^{\epsilon,u}(t),\tilde{Y}^{\epsilon}_u(t)\big)\big]dt +\frac{1}{\sqrt{\delta}}\; dw_2(t)\;,
\tilde{Y}_u^\epsilon(0)=y_0.
\end{equation*}

\noindent The following lemma, whose proof is deferred to the end of this section, shows that the process $\tilde{Y}^{\epsilon}_u(t)$ is close to the controlled fast process $Y^{\epsilon,u}$ in an appropriate ergodic sense.

\begin{lem}\label{fastergodiclem} Let $T<\infty, u\in\mathcal{P}_N^T$ and $\Delta=\Delta(\epsilon)>0$ as in Definition \ref{Delta1}. Then
	\begin{equation}\label{fastergodic}
	\frac{1}{\Delta}\ex\int_{0}^{T}\big\| Y^{\epsilon,u}(t)-\tilde{Y}_u^{\epsilon}(t)\big\|^2_\h dt\leq C_{T,g}\frac{\delta h^2(\epsilon)}{\Delta}\longrightarrow 0\;,\text{as}\;\epsilon\to 0.
	\end{equation}
\end{lem}

\noindent Similarly, for $s\geq t$, we can define the two parameter process $Y^{\epsilon, X^{\epsilon,u}(t)}(s;t)$ solving
\begin{equation*}
\begin{aligned}
&dY^{\epsilon, X^{\epsilon,u}(t)}(s;t)= \frac{1}{\delta}\big[A_2Y^{\epsilon, X^{\epsilon,u}(t)}(s;t)+ G\big( X^{\epsilon,u}(t), Y^{\epsilon, X^{\epsilon,u}(t)}(s;t)\big)\big]ds +\frac{1}{\sqrt{\delta}}\; dw_2(s)\;,
\\&Y^{\epsilon,X^{\epsilon,u}(t)}(t;t)=Y^{\epsilon}(t)
\end{aligned}
\end{equation*}
and show that for any $t>0$ there exists $\epsilon_0(t)>0$ such that for all $\epsilon<\epsilon_0$ we have

\begin{equation}\label{fastergodic1}
\frac{1}{\Delta}\ex\int_{t}^{t+\Delta}\big\| \tilde{Y}_u^{\epsilon}(t) dt-Y^{\epsilon,X^{\epsilon,u}(t)}(s;t)\big\|^2_\h ds\leq C_{t,\epsilon},
\end{equation}
with $\Delta$ as in \eqref{Delta1} and for each fixed $t>0$,  $C_{t,\epsilon}\rightarrow 0$ as $\epsilon\to0$. This shows that, in small time intervals, we can consider the effect of $X^{\epsilon,u}$ as frozen.

In view of \eqref{fastergodic} and \eqref{fastergodic1} we can now apply Lemma 4.19 from \cite{WSS} to show that, for any $f\in C_b(\h)$,
\begin{equation*}
\int_{\h\times\h\times\h\times [0,T]}f(y)dP_i(u_1,u_2,y,t)=\int_{0}^{T}\int_{\h}f(y)d\mu^{\bar{X}(t)}(y)dt.
\end{equation*}
This completes the proof of the decomposition \eqref{viable2}. Let us now conclude this section with the proof of Lemma \ref{fastergodiclem}.
\begin{proof}[Proof of Lemma \ref{fastergodiclem}]
	Let $\Gamma^{\epsilon,u}= Y^{\epsilon,u}-\tilde{Y}_u^{\epsilon}$. This process has weakly differentiable paths and solves the equation
	\begin{equation*}
	\partial_t\Gamma^{\epsilon,u}(t)= \frac{1}{\delta}\big[ A_2
	\Gamma^{\epsilon,u}(t)+ G(X^{\epsilon,u}(t), \tilde{Y}_u^{\epsilon}(t))-G(X^{\epsilon,u}(t), Y^{\epsilon}(t))\big]+\frac{h(\epsilon)}{\sqrt{\delta}}u_2(t),\;	\Gamma^{\epsilon,u}(0)=0_\h\;.
	\end{equation*}
	As in Lemma \ref{Yprebnd} we have
	\begin{equation*}
	\begin{aligned}
	\frac{1}{2}\partial_t\|\Gamma^{\epsilon,u}(t)\|^2_{\h}&\leq \frac{L_g-\lambda}{\delta}\|\Gamma^{\epsilon,u}(t)\|^2_{\h}+\frac{h(\epsilon)}{\sqrt{\delta}}\|\Gamma^{\epsilon,u}(t)\|_\h\|u_2(t)\|_\h\\&
	\leq \frac{L_g-\lambda}{2\delta}\|\Gamma^{\epsilon,u}(t)\|^2_{\h}+\frac{h^2(\epsilon)}{c_g}\|u_2(t)\|^2_\h.
	\end{aligned}
	\end{equation*}
	Integrating yields
	\begin{equation*}
	\begin{aligned}
	\frac{1}{2}\sup_{t\in[0,T]}\|\Gamma^{\epsilon,u}(t)\|^2_{\h}+\frac{\lambda-L_g}{2\delta}\int_{0}^{T}\|\Gamma^{\epsilon,u}(t)\|^2_{\h}dt\leq \frac{h^2(\epsilon)}{c_g}\int_{0}^{T}\|u_2(t)\|^2_\h dt\leq \frac{Nh^2(\epsilon)}{c_g}\;.
	\end{aligned}
	\end{equation*}
	The latter completes the proof, since it implies $	\int_{0}^{T}\|\Gamma^{\epsilon,u}(t)\|^2_{\h}dt\leq C_{g,N}\delta h^2(\epsilon).$
\end{proof}

	\section{Proof of the Moderate Deviation Principle}\label{MDPsec}

This section is devoted to the proof of Theorem \ref{MDP1}. Recall from Section \ref{weakconv} that the  MDP for the family $\{X^\epsilon\;,\epsilon>0\}$ of slow processes is equivalent to an LDP for the family $\{\eta^\epsilon\;,\epsilon>0\}$ with speed $h^2(\epsilon)$.

In Section \ref{LPup} we use the variational representation \eqref{varrep} to show that, in Regime $i=1,2$,  $\{\eta^\epsilon\;,\epsilon>0\}$ satisfies the Laplace Principle upper bound with rate function
\begin{equation}
\label{ratefun}
\small
\begin{aligned}
&\mathcal{S}_i(\phi):=\inf_{(\phi,P)\in\mathcal{V}_{(\Xi_i, \mu^{\bar{X}})} }\bigg[\frac{1}{2}\int_{\h\times\h\times\h\times[0,T]}\big( \|u_1 \|^2_\h+ \|u_2 \|^2_\h\big)\;dP(u_1,u_2,y,t)\bigg] \;\;,\phi\in C\big([0,T];\h\big),
\end{aligned}
\end{equation}
where $\Xi_i$ is given in \eqref{xidef1} and the infimum runs over the family $\mathcal{V}_{(\Xi_i, \mu^{\bar{X}})}$ of viable pairs (Definition \ref{viable}). The upper bound is a straightforward consequence of Theorem \ref{viablim1} and the Portmanteau lemma.

The Laplace Principle lower bound in Regime $i$ is proved in Section \ref{LPlo}. The situation for the  lower bound is more complicated, as we have to construct nearly optimal controls that achieve the bound. To do so, we take advantage of the affine structure  of the limiting dynamics, captured by $\Xi_i$, to express the rate function in an explicit, non-variational form (\ref{nonvar}). This allows us to construct  nearly optimal controls which, in principle, depend on the fast process in feedback form, but have sufficient regularity properties for the averaging principle to hold.

Finally, we verify in Section \ref{levelsets} that the rate function has compact sublevel sets. This guarantees that the LDP is equivalent to the LP and completes the analysis.

Note that throughout Section \ref{LPlo} we switch from Hypothesis \ref{A3a} to the stronger Hypothesis \ref{A3a'}. The reasons for this will become clear below.

\subsection{Laplace Principle upper bound}\label{LPup}

We aim to prove that for $T<\infty$ and any bounded, continuous $\Lambda: C\big([0,T];\h\big)\rightarrow \R$,
\begin{equation}
\label{Laplupper}
\begin{aligned}
&\limsup_{\epsilon\to 0}\frac{1}{h^2(\epsilon)}\log\ex\big[ e^{-h^2(\epsilon)\Lambda(\eta^\epsilon)}\big]\leq -\inf_{\phi\in C([0,T];
	\h)}\big[ \mathcal{S}_i(\phi)+\Lambda(\phi) \big]\;, i=1,2.
\end{aligned}
\end{equation}

\noindent It suffices to verify the above limit along any convergent subsequence in $\epsilon$. Such a subsequence exists since, for $\epsilon$ small enough,
$$\bigg|\frac{1}{h^2(\epsilon)}\log\ex\big[ e^{-h^2(\epsilon)\Lambda(\eta^\epsilon)}\big]\bigg|\leq \sup_{\phi\in C([0,T];\h)}\big|\Lambda(\phi)\big|. $$

\noindent Next let $\rho>0$. In view of the variational representation \eqref{varrep}, it follows that for each $\epsilon>0$ there exists a family of controls $\{(u_1^\epsilon, u_2^\epsilon) \}_{\epsilon>0}\subset\mathcal{P}^T(\h\oplus\h)$ such that
\begin{equation}\label{preup}
\frac{1}{h^2(\epsilon)}\log\ex\big[ e^{-h^2(\epsilon)\Lambda(\eta^\epsilon)}\big]\leq -\ex\bigg[\frac{1}{2} \int_{0}^{T}  \big(\|u_1^\epsilon(t) \|^2_\h+ \|u_2^\epsilon(t) \|^2_\h\big)\;dt+\Lambda\big(    \eta^{\epsilon, u^\epsilon} \big)\bigg]+\rho.
\end{equation}
\noindent In fact, we can assume without loss of generality that $\{(u_1^\epsilon,u_2^\epsilon)\}_{\epsilon>0}\subset\mathcal{P}_N^T(\h\oplus\h)$ for $N=N(\rho)$ large enough (see \cite{budhiraja2008large} and \cite{budhiraja2000variational},p.22). Using this family of controls and the associated controlled moderate deviations processes $\eta^{\epsilon,u^\epsilon}$ we can define occupation measures $P^{\epsilon,\Delta}$ and, from Theorem \ref{viablim1}, the family $\{  (\eta^{\epsilon,u^\epsilon},P^{\epsilon,\Delta}),\epsilon,\Delta>0\}$ is tight. From the same theorem, any sequence of $\epsilon$,  contains  a further subsequence for which $(\eta^{\epsilon,u^\epsilon},P^{\epsilon,\Delta})$ converges in distribution, in Regime $i$, to a viable pair $(\eta_i, P_i)\in\mathcal{V}_{(\Xi_i, \mu^{\bar{X}})}$. Taking limits along this subsequence in \eqref{preup} yields
\begin{equation*}
\small
\begin{aligned}
\limsup_{\epsilon\to 0}\frac{1}{h^2(\epsilon)}\log\;\ex\big[& e^{-h^2(\epsilon)\Lambda(\eta^\epsilon)}\big]\leq \limsup_{\epsilon\to 0}-\ex\bigg[\frac{1}{2} \int_{0}^{T} \frac{1}{\Delta}\int_{t}^{t+\Delta} \big(\|u_1^\epsilon(s) \|^2_\h+ \|u_2^\epsilon(t) \|^2_\h\big)\;ds dt+\Lambda\big(    \eta^{\epsilon, u^\epsilon} \big)\bigg]+\rho\\&
= -\liminf_{\epsilon\to 0}\ex\bigg[\frac{1}{2} \int_{\h\times\h\times\h\times[0,T]} \big(\|u_1 \|^2_\h+ \|u_2 \|^2_\h\big)\;dP^{\epsilon,\Delta}(u_1,u_2,y,t)+\Lambda\big(    \eta^{\epsilon, u^\epsilon} \big)\bigg]+\rho.
\end{aligned}
\end{equation*}
\noindent Since the map $$ \mathscr{P}\big(\h\times\h\times\h\times[0,T]\big)\ni \nu\longmapsto \int_{\h\times\h\times\h\times[0,T]}\big(\|u_1\|^2_\h+\|u_2\|^2_\h\big) d\nu(u_1,u_2,y,t)\in\R$$ is nonnegative and lower semi-continuous, we  use the Portmanteau lemma to obtain
\begin{equation*}
\begin{aligned}
\limsup_{\epsilon\to 0}\frac{1}{h^2(\epsilon)}\log&\ex\big[ e^{-h^2(\epsilon)\Lambda(\eta^\epsilon)}\big]\leq
-\ex\bigg[\frac{1}{2}\int_{\h\times\h\times\h\times[0,T]}\big(\|u_1\|^2_\h+\|u_2\|^2_\h\big)\;dP_i(u_1,u_2,y,t)+\Lambda(\eta_i)\bigg]+\rho\\&\leq
-\inf_{(\phi,P)\in\mathcal{V}_{(\Xi_i, \mu^{\bar{X}})} }\bigg[ \frac{1}{2}\int_{\h\times\h\times\h\times[0,T]}\big(\|u_1\|^2_\h+\|u_2\|^2_\h\big)\;dP(u_1,u_2,y,t)+\Lambda(\phi)\bigg]+\rho.
\end{aligned}
\end{equation*}
	Since $\rho>0$ is arbitrary, the proof of \eqref{Laplupper} is complete.

\subsection{Laplace Principle lower bound}\label{LPlo}

Assume Hypotheses \ref{A3a'} and \ref{A3b}. We aim to prove that for $T<\infty$ and any bounded, continuous $\Lambda: C\big([0,T];\h\big)\rightarrow \R$
\begin{equation}
\label{Laplower}
\liminf_{\epsilon\to 0}\frac{1}{h^2(\epsilon)}\log\ex\big[ e^{-h^2(\epsilon)\Lambda(\eta^\epsilon)}\big]\geq-\inf_{\phi\in C([0,T];
	\h)}\big[ \mathcal{S}_i(\phi)+\Lambda(\phi) \big]\;, i=1,2.
\end{equation}
From our definition of viable pairs and Theorem \ref{aveproof} we see that the third marginal of the invariant measure $P$ does not depend on the control variables $u_1,u_2$ and is in fact given by the local invariant measure $\mu^x$. This decoupling is further exploited in the following lemma, which allows to rewrite the rate function $\mathcal{S}_i$ (see \eqref{ratefun}) in a convenient ordinary control formulation.
\begin{lem}\label{ordformlem}
	With $i=1,2$ and $\Xi_i, \mu^x$ as in Theorem \ref{MDP1}, let \begin{equation*}
	\begin{aligned}
	\mathscr{A}^r_{i,\psi,T}=&\bigg\{ P: [0,T]\longrightarrow\mathscr{P}(\h\times\h\times\h) : P_t(B_1\times B_2\times B_3)=\int_{B_3}\nu(B_1\times B_2|y,t)d\mu^{\bar{X}(t)}(y)\;,  \\&\int_{0}^{T}\int_{\h\times\h\times\h}\big(\|u_1\|^2_\h+\|u_2\|^2_\h+\|y\|^2_{H^\theta}\big) dP_s(u_1,u_2,y)ds<\infty\;\;\text{for some}\;\theta>0,\\&
	\psi(t)=\int_{0}^{t}\int_{\h\times\h\times\h}S_1(t-s)\Xi_i\big(\bar{X}(s), \psi(s), y, u_1,u_2\big) dP_s(u_1,u_2,y)ds\bigg\}
	\end{aligned}
	\end{equation*}
	and
	\begin{equation*}
	\begin{aligned}
	\mathscr{A}^o_{i,\psi,T}=&\bigg\{ (u_1,u_2): [0,T]\times\h\longrightarrow\h \times\h: \\& \int_{0}^{T}\int_{\h}\big(\|u_1(s,y)\|^2_\h+\|u_2(s,y)\|^2_\h+\|y\|^2_{H^\theta}\big) d\mu^{\bar{X}(s)}(y)ds<\infty\;\;\text{for some}\;\theta>0,\\&
	\psi(t)=\int_{0}^{t}\int_{\h}S_1(t-s)\Xi_i\big(\bar{X}(s), \psi(s), y, u_1(s,y), u_2(s,y)\big) d\mu^{\bar{X}(s)}(y)ds\bigg\}
	\end{aligned}
	\end{equation*}
	\noindent (the superscripts $r,o$ refer to the relaxed and ordinary control formulations respectively). For $\psi\in C\big([0,T];\h\big)$ we have
	\begin{equation}\label{ordform}
	\begin{aligned}
	\mathcal{S}_i(\psi)&=\inf_{P\in 	\mathscr{A}^r_{i,\psi,T} }\bigg[\frac{1}{2}\int_{0}^{T}\int_{\h\times\h\times\h} \big(\|u_1\|^2_\h+\|u_2\|^2_\h\big)\;dP_s(u_1,u_2,y)ds\bigg] \\&=\inf_{(u_1,u_2)\in 	\mathscr{A}^o_{i,\psi,T} }\bigg[\frac{1}{2}\int_{0}^{T}\int_{\h} \big(\|u_1(s,y)\|^2_\h+\|u_2(s,y)\|^2_\h\big)\;d\mu^{\bar{X}(s)}(y)ds\bigg].
	\end{aligned}
	\end{equation}
\end{lem}
\noindent This result is standard and a proof can be found e.g. in \cite{WSS}, Section 5.2.
\noindent Proceeding to the main proof, let $\rho>0$ and $\psi\in C\big([0,T];\h\big)$ such that
\begin{equation}\label{approxmin}
\mathcal{S}_i(\psi)+\Lambda(\psi)\leq\inf_{\phi\in C([0,T];\h)}\big[ \mathcal{S}_i(\phi)+\Lambda(\phi)\big]+\rho<\infty.
\end{equation}
\noindent For each $(u_1,u_2)\in\mathscr{A}^o_{i,\psi, T}$,
\begin{equation*}
\begin{aligned}
&\psi(t)= \int_{0}^{t}\int_{\h}S_1(t-s)\Xi_i\big(\bar{X}(s), \psi(s), y, u_1(s,y),  u_2(s,y)\big) d\mu^{\bar{X}(s)}(y)ds\\&
= \int_{0}^{t}\int_{\h}S_1(t-s) D_xF\big(\bar{X}(s),y\big)\psi(s) d\mu^{\bar{X}(s)}(y)ds+ \int_{0}^{t}\int_{\h}S_1(t-s)\Sigma\big(\bar{X}(s),y\big) u_1(s,y) d\mu^{\bar{X}(s)}(y)ds
\\&+  \gamma_i\int_{0}^{t}\int_{\h}S_1(t-s) \Psi^0_2\big(\bar{X}(s),y\big)u_2(s,y) d\mu^{\bar{X}(s)}(y)ds.
\end{aligned}
\end{equation*}

\noindent Hence, $\psi$ is the mild solution of the semilinear evolution equation 	
\begin{equation}\label{limequation}
	\left \{\begin{aligned}
&\partial_t\psi(t)=  A_1\psi(t)+\overline{D_xF}\big(\bar{X}(t)\big)\psi(t)+\int_{\h}\big[\Sigma\big(\bar{X}(t),y\big)u_1(t,y)+ \gamma_i\Psi^0_2\big(\bar{X}(t),y\big)u_2(t,y)\big]d\mu^{\bar{X}(t)}(y) \\&\psi(0)=0_\h\;,
\end{aligned}\right.
\end{equation}
where \begin{equation}\label{DFaved}
\overline{D_xF}\big(\bar{X}(t)\big):=\int_{\h}D_xF\big(\bar{X}(t),y\big)d\mu^{\bar{X}(t)}(y).
\end{equation}
\noindent In view of Hypotheses \ref{A2a} and \ref{A3a'}, the maps $$t\longmapsto \int_{\h}D_xF\big(\bar{X}(t),y\big)\psi(t)d\mu^{\bar{X}(t)}(y)\;,\int_{\h}\big[\Sigma\big(\bar{X}(t),y\big)u_1(t,y)+ \gamma_i\Psi^0_2\big(\bar{X}(t),y\big)u_2(t,y)\big]d\mu^{\bar{X}(t)}(y)$$
belong to $L^2([0,T];\h)$. From standard theory of deterministic parabolic equations it follows that $\psi$ is a weak solution of \eqref{limequation} in the sense that $\psi\in H_{0}^1([0,T];\h)\cap L^2([0,T];Dom(A_1))$.

\noindent The next step is to show that $\mathcal{S}_i$ has a non-variational form. To this end, let $x\in\h$ and define $\widetilde{Q}_i(x):L^2( \h,\mu^x;\h)\oplus L^2( \h,\mu^x;\h)\rightarrow\h$ with
\begin{equation*}
\widetilde{Q}_i(x)(u_1,u_2):=\int_{\h}\big[\Sigma(x,y)u_1(y)+\gamma_i\Psi^{0}_2(x,y)u_2(y)\big]d\mu^{x}(y)\;,\; i=1,2.
\end{equation*}
Note that $\widetilde{Q}_i^*(x):\h\rightarrow L^2( \h,\mu^x;\h)\oplus L^2( \h,\mu^x;\h)$ is given by \begin{equation}
\label{Qtildestar}
\widetilde{Q}_i^*(x)v:=\big( \Sigma^*(x,y)v,  \gamma_i\Psi^{0*}_2(x,y) v  \big).
\end{equation}
Next, define $Q_i(x)\in\mathscr{L}(\h)$ by
\begin{equation}\label{Qi1}
Q_i(x):=\widetilde{Q}_i(x)\widetilde{Q}_i^*(x)=\int_{\h}\big[ \Sigma(x,y)\Sigma^*(x,y)+\gamma^2_i\Psi^0_2(x,y)\Psi^{0*}_2(x,y) \big]d\mu^x(y).
\end{equation}
We can now prove the following:

\begin{prop}\label{optimalprop}
	Under Hypothesis \ref{A3a'} the following hold:\\
	\noindent (i)  For $i=1,2$ and each $x\in\h$, $Q_i(x)$ has a bounded inverse that satisfies
	\begin{equation}\label{Qinvnorm}
	\sup_{x\in\h}\|Q^{-1}_i(x)\|_{\mathscr{L}(\h)}\leq c^{-2}_1.
	\end{equation}
	Furthermore,
	$\widetilde{Q}_i(x)$ has a bounded right inverse given by
	\begin{equation}\label{Qtildepseudo}
	\widetilde{Q}^{+}_i(x)=\widetilde{Q}^{*}_i(x)Q^{-1}_i(x).
	\end{equation}
	\noindent (ii) For $i=1,2$ and $T<\infty,$ $\mathcal{S}_i(\psi)<\infty$ if and only if $\psi\in H_{0}^1([0,T];\h)\cap L^2([0,T];Dom(A_1))$. Moreover, the infimum in \eqref{ordform} is attained and letting
		\begin{equation}\label{v1opt}
	v^i_1(t,y)=\Sigma^*\big(\bar{X}(t),y\big)Q^{-1}_i\big(\bar{X}(t)\big)\bigg(\partial_t\psi(t)-A_1\psi(t)-\overline{D_xF}\big(\bar{X}(t)\big)\psi(t)\bigg),
	\end{equation}
	\begin{equation}\label{v2opt}v^i_2(t,y)=\gamma_i\Psi^{0*}_2\big(\bar{X}(t), y\big)Q^{-1}_i\big(\bar{X}(t)\big)\bigg(\partial_t\psi(t)-A_1\psi(t)-\overline{D_xF}\big(\bar{X}(t)\big)\psi(t)\bigg)	\end{equation}
	we have
	\begin{equation*}
	(v_1^{i}, v_2^{i}    )\in\mathrm{argmin}_{(u_1,u_2)\in\mathscr{A}^o_{i,\psi,T}}\bigg\{ \int_{0}^{T} \int_{\h}\big(\|u_1(t,y)\|^2_\h +\|u_2(t,y)\|^2_\h   \big) d\mu^{\bar{X}(t)}(y)  dt    \bigg\}.
	\end{equation*}
	Hence, the rate function in Regime $i$ takes the non-variational form
	\begin{equation*}
	\hspace*{-0.8cm}
	\begin{aligned}
	& \mathcal{S}_i(\psi)=\frac{1}{2}\int_{0}^{T}\bigg\|Q_i\big(\bar{X}(t)\big)^{-\frac{1}{2}}\big[ \partial_t\psi(t)- A_1\psi(t)-\overline{D_xF}\big(\bar{X}(t)\big)\psi(t)\big]\bigg\|^2_\h dt,
	\end{aligned}
	\end{equation*}
	for $\psi\in  H_0^1([0,T];\h)\cap L^2([0,T];Dom(A_1))$ and $\mathcal{S}_i=\infty$ otherwise.
\end{prop}
\begin{proof}
	(i) Let $u\in\h$. By definition, $Q_i(x)$ is self-adjoint and from Hypothesis \ref{A3a'} we have
	\begin{equation*}
	\begin{aligned}
	\langle Q_i(x)u,u\rangle_\h &= \|\widetilde{Q}^*_i(x)u\|^2_{L^2( \h,\mu^x;\h)\oplus L^2( \h,\mu^x;\h)}\\&=\int_{\h}\|\Sigma^{*}(x,y)u\|^2_{\h}d\mu^x(y)+\gamma_i^2\int_{\h}\|\Psi^{0*}_2(x,y)u\|^2_{\h}d\mu^x(y)\\&\geq c_1^2\|u\|^2_\h\mu^{x}(\h)=c_1^2\|u\|^2_\h\;.
	\end{aligned}
	\end{equation*}
	Thus, $Q_i(x)$ is injective and
	\begin{equation*}
	\|\widetilde{Q}_i(x)u\|_\h\geq c^2_1\|u\|_\h,
	\end{equation*}
	which implies that $\widetilde{Q}_i(x)$ has a closed range in $\h$. It follows that $Q_i(x)(\h)=\overline{Q_i(x)(\h)}=\ker( Q^*_i(x) )^{\perp}=\ker(Q_i(x) )^{\perp}=\{0_\h\}^{\perp}=\h$. By virtue of the inverse mapping theorem we deduce that $Q^{-1}_i(x)\in\mathscr{L}(\h)$ and \eqref{Qinvnorm} follows. Lastly, it is straightforward to check that  $\widetilde{Q}_i^{+}(x)$ is a right inverse of $\widetilde{Q}_i(x)$ and in view of \eqref{Qtildestar} and \eqref{Qinvnorm}, $\widetilde{Q}_i^{+}(x)\in\mathscr{L}( \h; L^2( \h,\mu^x;\h)\oplus L^2( \h,\mu^x;\h)      )$.
	
	\noindent (ii)  Letting $\psi\in C([0,T];\h)$ such that $\mathcal{S}_i(\psi)<\infty$ it follows that $\mathscr{A}^o_{i,\psi,T}\neq\varnothing$. From our previous discussion, there exists $(u_1,u_2)\in\mathscr{A}^o_{i,\psi,T}$ such that $\psi$ is the strong solution of $\eqref{limequation}$. Hence $\psi\in H_{0}^1([0,T];\h)\cap L^2([0,T];Dom(A_1))$ and for $t\in[0,T]$ we have
	$$(u_1(t,\cdot),u_2(t,\cdot))\in    \widetilde{Q}_i(\bar{X}(t))^{-1}\bigg(\partial_t\psi(t)-A_1\psi(t)-\overline{D_xF}\big(\bar{X}(t)\big)\psi(t)\bigg)   \subset       L^2( \h,\mu^{\bar{X}(t)};\h)\oplus L^2( \h,\mu^{\bar{X}(t)};\h).$$
	Since
	$	\widetilde{Q}^{+}_i(\bar{X}(t))\big(\partial_t\psi(t)-A_1\psi(t)-\overline{D_xF}\big(\bar{X}(t)\big)\psi(t)\big)$ is an element of $$\widetilde{Q}_i(\bar{X}(t))^{-1}\bigg(\partial_t\psi(t)-A_1\psi(t)-\overline{D_xF}\big(\bar{X}(t)\big)\psi(t)\bigg)$$ with minimal $   L^2( \h,\mu^{\bar{X}(t)};\h)\oplus L^2( \h,\mu^{\bar{X}(t)};\h)$-norm it follows that
	\begin{equation}\label{LB}
	\begin{aligned}
	\int_{0}^{T} &\int_{\h}\big(\|u_1(t,y)\|^2_\h +\|u_2(t,y)\|^2_\h   \big) d\mu^{\bar{X}(t)}(y)dt =\int_{0}^{T}\|( u_1(t,\cdot), u_2(t,\cdot)) \|^2_{ L^2( \h,\mu^{\bar{X}(t)};\h)\oplus L^2( \h,\mu^{\bar{X}(t)};\h)} dt\\&
	\geq \int_{0}^{T}    \big\|\widetilde{Q}^{+}_i(\bar{X}(t))\big(\partial_t\psi(t)-A_1\psi(t)-\overline{D_xF}\big(\bar{X}(t)\big)\psi(t)\big)\big\|^2_{ L^2( \h,\mu^X;\h)\oplus L^2( \h,\mu^X;\h)}dt\\&
	=\int_{0}^{T}    \big\|\widetilde{Q}^{*}_i(\bar{X}(t))Q^{-1}_i(\bar{X}(t))\big(\partial_t\psi(t)-A_1\psi(t)-\overline{D_xF}\big(\bar{X}(t)\big)\psi(t)\big)\big\|^2_{ L^2( \h,\mu^{\bar{X}(t)};\h)\oplus L^2( \h,\mu^{\bar{X}(t)};\h)}dt\\&
	=\int_{0}^{T}  \int_{\h}	\big\|\Sigma^{*}(\bar{X}(t),y)Q^{-1}_i(\bar{X}(t))\big(\partial_t\psi(t)-A_1\psi(t)-\overline{D_xF}\big(\bar{X}(t)\big)\psi(t)\big)\big\|^2_\h d\mu^{\bar{X}(t)}(y)dt\\&
	+\int_{0}^{T}  \int_{\h}	\big\|\gamma_i\Psi_2^{0*}(\bar{X}(t),y)Q^{-1}_i(\bar{X}(t))\big(\partial_t\psi(t)-A_1\psi(t)-\overline{D_xF}\big(\bar{X}(t)\big)\psi(t)\big)\big\|^2_\h d\mu^{\bar{X}(t)}(y)dt\\&
	=\int_{0}^{T}\bigg\langle \partial_t\psi(t)- A_1\psi(t)-\overline{D_xF}\big(\bar{X}(t)\big)\psi(t), Q_i^{-1}\big(\bar{X}(t)\big)\big[ \partial_t\psi(t)- A_1\psi(t)-\overline{D_xF}\big(\bar{X}(t)\big)\psi(t)\big]\bigg\rangle_\h dt.
	\end{aligned}
	\end{equation}
	\noindent Now, in view of \eqref{ordform},
	\begin{equation*}
	\begin{aligned}
	\mathcal{S}_i(\psi)&=\frac{1}{2}\inf_{(u_1,u_2)\in\mathscr{A}^o_{i,\psi,T}}\int_{0}^{T} \int_{\h}\big(\|u_1(t,y)\|^2_\h +\|u_2(t,y)\|^2_\h   \big) d\mu^{\bar{X}(t)}(y)dt\\&\geq\frac{1}{2}\int_{0}^{T}\big\|Q_i\big(\bar{X}(t)\big)^{-\frac{1}{2}}\big[ \partial_t\psi(t)- A_1\psi(t)-\overline{D_xF}\big(\bar{X}(t)\big)\psi(t)\big]\big\|^2_\h dt.
	\end{aligned}
	\end{equation*}
	From \eqref{Qtildepseudo} and \eqref{Qtildestar} we see that
	\begin{equation*}
	(v^i_1, v^i_2)=\widetilde{Q}^{+}_i(\bar{X}(t))\big[\partial_t\psi(t)-A_1\psi(t)-\overline{D_xF}\big(\bar{X}(t)\big)\psi(t)\big]
	\end{equation*}
	and since the $\mathscr{L}(\h)$-valued maps $ Q^{-1}_i, \Sigma^*, \Psi^{0*}_{2}$ are bounded uniformly in $x$ and $y$ (see \eqref{Qinvnorm}, \eqref{sigma'} and \eqref{psinorm} respectively) we conclude that  $(v^i_1, v^i_2)\in\mathscr{A}^o_{i,\psi,T}$ and achieves the lower bound in \eqref{LB}. The proof is complete.	\end{proof}

\noindent We are now ready to prove regularity properties for the pair $(v^i_1, v^i_2)$.
\begin{lem} For $i=1,2,$ $T<\infty$ and $(v^i_1, v^i_2)$ as in \eqref{v1opt}, \eqref{v2opt} there exists $\kappa_i\in L^2[0,T]$ such that  :\\
	\noindent (i) For each $t\in[0,T]$,
	\begin{equation*}
	\sup_{y\in\h}\|v^i_1(t,y)\|_\h+\sup_{y\in\h}\|v^i_2(t,y)\|_\h\leq \kappa_i(t).
	\end{equation*}
	\noindent (ii) For each $t\in[0,T]$ and $y_1, y_2\in\h$,
	\begin{equation*}
	\|v^i_1(t,y_1)-v^i_1(t,y_2)\|_\h +\|v^i_2(t,y_1)-v^i_2(t,y_2)\|_\h\leq \kappa_i(t)\|y_1-y_2\|_\h\;.
	\end{equation*}
\end{lem}
\begin{proof}
	(i) From Hypothesis \ref{A3a'} and \eqref{psinorm},
	\begin{equation*}
	\begin{aligned}
	\|v^i_1(t,y)\|_\h+\|v^i_2(t,y)\|_\h&\leq\big\|\Sigma^*\big(\bar{X}(t),y\big)Q^{-1}_i\big(\bar{X}(t)\big)\big(\partial_t\psi(t)-A_1\psi(t)-\overline{D_xF}\big(\bar{X}(t)\big)\psi(t)\big)\big\|_\h\\&
	+\big\|\gamma_i\Psi^{0*}_2\big(\bar{X}(t), y\big)Q^{-1}_i\big(\bar{X}(t)\big)\big(\partial_t\psi(t)-A_1\psi(t)-\overline{D_xF}\big(\bar{X}(t)\big)\psi(t)\big)\big\|_\h\\&
	\leq C_i\|Q^{-1}_i\big(\bar{X}(t)\big)\|_{\mathscr{L}(\h)}\|\partial_t\psi(t)-A_1\psi(t)-\overline{D_xF}\big(\bar{X}(t)\big)\psi(t)\big\|_\h\\&
	\leq C_i c_2^{-2}\big\|\partial_t\psi(t)-A_1\psi(t)-\overline{D_xF}\big(\bar{X}(t)\big)\psi(t)\big\|_\h\;,
	\end{aligned}
	\end{equation*}
	where the last line follows from \eqref{Qinvnorm}. Since $\psi_i\in H_{0}^1([0,T];\h)\cap L^2([0,T];Dom(A_1))$ and, in view of Hypothesis \ref{A2a}, $\sup_{t\in[0,T]}\|\overline{D_xF}\big(\bar{X}(t)\big)\|_{\mathscr{L}(\h)}<\infty$ we deduce that
	\begin{equation*}
	\small
	\int_{0}^{T}\big\|\partial_t\psi(t)-A_1\psi(t)-\overline{D_xF}\big(\bar{X}(t)\big)\psi(t)\big\|^2_\h dt\leq C\big(\|\psi\|^2_{C([0,T];\h)}+\|\psi\|_{L^2([0,T];Dom(A_1))}+\|\psi\|_{H^1_0([0,T];\h)}\big)<\infty.
	\end{equation*}
	The argument is complete upon setting
	\begin{equation}\label{kappa}
	\kappa_i(t):=\|\partial_t\psi(t)-A_1\psi(t)-\overline{D_xF}\big(\bar{X}(t)\big)\psi(t)\big\|_\h\;.
	\end{equation}
	(ii) With $\kappa_i$ as in \eqref{kappa},
	\begin{equation*}
	\hspace*{-0.4cm}
	\begin{aligned}
	\|v^i_1(t,y_1)-v^i_1(t,y_2)\|_\h& +\|v^i_2(t,y_1)-v^i_2(t,y_2)\|_\h\\&\leq \|Q^{-1}_i\big(\bar{X}(t)\big)\|_{\mathscr{L}(\h)}\|\kappa_i(t)\|_\h\|\Sigma\big(\bar{X}(t),y_1\big)-\Sigma\big(\bar{X}(t),y_2\big)   \|_{\mathscr{L}(\h)}\\&+\gamma_i\|\Psi_2^{0*}\big(\bar{X}(t),y_1\big)-\Psi_2^{0*}\big(\bar{X}(t),y_2\big)   \|_{\mathscr{L}(\h)}.
	\end{aligned}
	\end{equation*}
	In light of Hypothesis \ref{A3b} and \eqref{Psiylip} it follows that
	\begin{equation*}
	\begin{aligned}
	&\|v^i_1(t,y_1)-v^i_1(t,y_2)\|_\h +\|v^i_2(t,y_1)-v^i_2(t,y_2)\|_\h
	\leq C_i\|\kappa_i(t)\|_\h\|y_1-y_2\|_\h\;.
	\end{aligned}
	\end{equation*}
	The proof is complete.\end{proof}
Appealing to a mollification argument (see e.g. \cite{dupuis2011weak}, Section 6.5 as well as \cite{WSS}, Theorem 5.6) we can also assume, without loss of generality, that $v^i_1,v^i_2$ are continuous in time. Having established these regularity properties we can now use the optimal pair
$(v^i_1, v^i_2)$ to construct a pair of stochastic controls in feedback form that approximate the lower bound \eqref{approxmin}.  To this end, let $$v^{i,\epsilon}(t):=\big(v^i_1([t/\Delta]\Delta, \widetilde{Y}^{\epsilon, \bar{X}}(t)), v^i_2([t/\Delta]\Delta, \widetilde{Y}^{\epsilon, \bar{X}}(t)))\;,t\in[0,T]\;,i=1,2$$
where $[\cdot]$ denotes the floor function, $\Delta=\Delta(\epsilon)$ is such that
$\Delta/\delta\rightarrow \infty$ as $\epsilon\to0$
and $\widetilde{Y}^{\epsilon, \bar{X}}$ solves the evolution equation
\begin{equation*}
d\widetilde{Y}^{\epsilon, \bar{X}}(t)=\frac{1}{\delta}\big[ A_2\widetilde{Y}^{\epsilon, \bar{X}}(t)+ G\big(\bar{X}([t/\Delta]\Delta), \widetilde{Y}^{\epsilon, \bar{X}}(t)\big)\big]dt+\frac{1}{\sqrt{\delta}}dw_2(t)\;, \widetilde{Y}^{\epsilon, \bar{X}}(0)=y_0\in\h\;.
\end{equation*}
An application of Lemma 5.7 in \cite{WSS} yields
\begin{equation}\label{lbave1}
\small
\begin{aligned}
\lim_{\epsilon\to 0}\frac{1}{2}\ex\bigg[\int_{0}^{T} \|v^{i,\epsilon}(t)\|^2_{\h\oplus\h} dt\bigg]&=\frac{1}{2} \int_{0}^{T}\int_\h \big( \|v^i_1(t,y) \|^2_\h+\|v^i_2(t,y) \|^2_\h\big)\;d\mu^{\bar{X}(t)}(y)dt
=\mathcal{S}_i(\psi),
\end{aligned}
\end{equation}
\noindent where the last equality follows from Proposition \ref{optimalprop}(ii). Next consider, in Regime $i$, the family of moderate deviations processes $\eta^{\epsilon, v^{i,\epsilon}}$ controlled by $v^{i,\epsilon}$. Repeating the arguments of Section \ref{Sec4} it follows that \begin{equation}\label{lbave2}
\eta^{\epsilon, v^{i,\epsilon}}\longrightarrow\psi \;\text{as}\; \epsilon\to 0\; \text{in distribution in}\; C([0,T];\h).
\end{equation} To verify the latter, the only additional step is to show that the control terms $II^{\epsilon,v^{i,\epsilon}}, IV^{\epsilon,v^{i,\epsilon}}$ converge to the averaging limit. In particular, we can apply the arguments of Lemma 5.8 in \cite{WSS} to show that, as $\epsilon\to 0$,
\begin{equation*}
\small
\begin{aligned}
\int_{0}^{t}S_1(t-s)\Sigma\big(\bar{X}(s), \widetilde{Y}^{\epsilon,\bar{X}}(s)\big)v^i_1([s/\Delta]\Delta, \widetilde{Y}^{\epsilon, \bar{X}}(s))ds\rightarrow \int_{0}^{t}\int_{\h}S_1(t-s)\Sigma\big(\bar{X}(s),y\big) v^{i}_1(s,y) d\mu^{\bar{X}(s)}(y)ds
\end{aligned}
\end{equation*}
and
\begin{equation*}
\small
\begin{aligned}
\frac{\sqrt\delta}{\sqrt\epsilon} \int_{0}^{t}S_1(t-s)\Psi_2^{0}\big(\bar{X}(s), \widetilde{Y}^{\epsilon,\bar{X}}(s)\big)v^i_2([s/\Delta]\Delta, \widetilde{Y}^{\epsilon, \bar{X}}(s))ds\rightarrow \gamma_i\int_{0}^{t}\int_{\h}S_1(t-s)\Psi_2^{0}\big(\bar{X}(s),y\big) v^{i}_2(s,y) d\mu^{\bar{X}(s)}(y)ds
\end{aligned}
\end{equation*}
in $L^1(\Omega;C([0,T];\h))$.

	 \noindent In view of \eqref{lbave1} and \eqref{lbave2} along with the variational representation \eqref{varrep}, the Laplace Principle lower bound follows. Indeed, for any bounded, continuous $\Lambda: C([0,T];\h)\rightarrow\R$
\begin{equation*}
\begin{aligned}
\limsup_{\epsilon\to 0}-\frac{1}{h^2(\epsilon)}\log\;\ex\big[ e^{-h^2(\epsilon)\Lambda(\eta^\epsilon)}\big]&=\limsup_{\epsilon\to 0}\inf_{u\in\mathcal{P}^T(\h\oplus\h)}\ex\bigg[ \frac{1}{2} \int_{0}^{T}\|u(t) \|^2_{\h\oplus\h}dt+ \Lambda\big(\eta^{\epsilon,u}  \big)  \bigg] \\&\leq
\limsup_{\epsilon\to 0} \ex\bigg[ \frac{1}{2} \int_{0}^{T} \|v^{\epsilon}_i(t) \|^2_{\h\oplus\h}\;dt+\Lambda\big(\eta^{\epsilon,v_i^\epsilon}  \big)  \bigg] \\&=   \frac{1}{2} \int_{0}^{T}\int_\h \big( \|v^i_1(t,y) \|^2_\h+\|v^i_2(t,y) \|^2_\h\big)\;d\mu^{\bar{X}(t)}(y)dt+ \Lambda(\psi) \\&
=\mathcal{S}_i(\psi)+ \Lambda(\psi)\leq \inf_{\phi\in C([0,T];\h)}\big[ \mathcal{S}_i(\phi)+\Lambda(\phi)\big]+\rho.
\end{aligned}
\end{equation*}
\noindent where the equality on the last line follows from the optimality of $v^{i}_1, v^{i}_2$ and the last inequality is due to the fact that $\psi_i$ was chosen to satisfy \eqref{approxmin}. Since $\rho$ is arbitrary, the result follows.

\subsection{Compactness of the sublevel sets}\label{levelsets}
In this section we show that $\mathcal{S}_i,$ $i=1,2$ (see \eqref{ratefun}) is a \textit{good} rate function, i.e. for each $M>0$ the sublevel set $$\mathcal{Z}_i(M)=\{  \psi\in C([0,T];\h) : \mathcal{S}_i(\psi)\leq M      \}$$
is compact. To this end, consider a sequence  of viable pairs $\{(\psi_n, P_n)\}_{n\in\N}\subset\mathcal{V}_{(\Xi_i, \mu^{\bar{X}})}$ such that \begin{equation*}\label{level1}\int_{\h\times\h\times\h\times[0,T]}\big( \|u_1 \|^2_\h+ \|u_2 \|^2_\h+\|y\|^2_{H^\theta}\big)\;dP_n(u_1,u_2,y,t)\leq M.\end{equation*}
Now for each $n\in\N$, $\psi_n\in H_{0}^1([0,T];\h)\cap L^2([0,T];Dom(A_1))$ is the strong solution of \eqref{limequation}. Since the last marginal of $P_n$ is Lebesgue measure we can work with the mild solution of \eqref{limequation} to prove estimates similar to those of Lemma \ref{etatightness} that are uniform in $n\in\N$. By an Arzel\`a-Ascoli argument we conclude that $\{\psi_n\}_{n\in\N}\subset C([0,T];\h)$ is relatively compact. Moreover, we can use Prokhorov's theorem exactly as we did in Lemma \ref{Palaoglu} to show that the sequence of (deterministic) measures $\{P_n\}_{n\in\N}\subset\mathscr{P}(\h\times\h\times\h\times [0,T])$ is weakly relatively sequentially compact.

Next, we claim that the limit $(\psi, P)$ of any convergent sequence of $\{(\psi_n, P_n)\}$ is also a viable pair. To this end, note that the Portmanteau lemma immediately implies that
\begin{equation*}
\int_{\h\times\h\times\h\times[0,T]}\big( \|u_1 \|^2_\h+ \|u_2 \|^2_\h+\|y\|^2_{H^\theta}\big)\;dP(u_1,u_2,y,t)<\infty\;;
\end{equation*}
hence \eqref{viable1} holds. For each $n\in\N$ we have
\begin{equation*}
\psi_n(t)=\int_{\h\times\h\times\h\times[0,t]}S_1(t-s)\Xi_i\big( \psi_n(s),\bar{X}(s), y, u_1,u_2\big) dP_n(u_1,u_2,y,s)
\end{equation*}
and we can show that $P_n$ are uniformly integrable as in Lemma \ref{UILem}. Since $\Xi_i$ is affine in $\psi$, $u$ and $(\psi_n, P_n)$ converges to $(\psi, P)$, the latter will also satisfy \eqref{viable3}.
Proving that $(\psi, P)$ satisfies \eqref{viable2} is straightforward since, at the prelimit level, we have
$$ dP_n(u_1,u_2,y,t)=d\nu_n(u_1,u_2 |y,t)d\mu^{\bar{X}(t)}(y)dt,$$
where $\nu_n$ is a sequence of stochastic kernels. Finally, $P$ satisfies $\eqref{viableleb}$ since, for each $n$, the last marginal of $P_n$ is Lebesgue measure and $P(\h\times\h\times\h\times[0,t])=t$. Therefore,  $(\psi, P)$ is indeed in $\mathcal{V}_{(\Xi_i, \mu^{\bar{X}})}$.

At this point we have established that for $i=1,2$ and $M>0$ the sublevel set $\mathcal{Z}_i(M)$ is relatively compact. To show compactness it remains to prove that it is closed. This will be done by showing that $\mathcal{S}_i$ is lower-semicontinuous. Indeed, let $\{(\psi_n, P_n)\}$ be a sequence of viable pairs converging to a pair $(\psi, P)$. Assuming that $\liminf_{n\to\infty} \mathcal{S}_i(\psi_n)=M<\infty$ we can pass to a subsequence that satisfies
\begin{equation}\label{level2}\int_{\h\times\h\times\h\times[0,T]}\big( \|u_1 \|^2_\h+ \|u_2 \|^2_\h+\|y\|^2_{H^\theta}\big)\;dP_n(u_1,u_2,y,t)\leq M'\end{equation}
and
\begin{equation*} \mathcal{S}_i(\psi_n)\geq \int_{\h\times\h\times\h\times[0,T]}\big( \|u_1 \|^2_\h+ \|u_2 \|^2_\h\big)\;dP_n(u_1,u_2,y,t)-\frac{1}{n}\;.
\end{equation*}
From \eqref{level2} and our previous discussion, $\{(\psi_n, P_n)\}$ has a subsequence that converges to a viable pair $\{(\psi', P')\}$ and by uniqueness of the limit $(\psi', P')=(\psi, P)$. It follows that
\begin{equation*}
\begin{aligned}
\liminf_{n\to\infty} \mathcal{S}_i(\psi_n)& \geq 	\liminf_{n\to\infty}\int_{\h\times\h\times\h\times[0,T]}\big( \|u_1 \|^2_\h+ \|u_2 \|^2_\h\big)\;dP_n(u_1,u_2,y,t)\\&\geq \int_{\h\times\h\times\h\times[0,T]}\big( \|u_1 \|^2_\h+ \|u_2 \|^2_\h\big)\;dP(u_1,u_2,y,t)\\&
\geq \inf_{(\psi,P)\in\mathcal{V}_{(\Xi_i, \mu^{\bar{X}})}}\int_{\h\times\h\times\h\times[0,T]}\big( \|u_1 \|^2_\h+ \|u_2 \|^2_\h\big)\;dP(u_1,u_2,y,t) = \mathcal{S}_i(\psi);
\end{aligned}
\end{equation*}
hence $\mathcal{S}_i$ is lower semicontinuous. The proof is complete.

\appendix
\section{}\label{AppA}
\noindent In this section we collect a few preliminary estimates concerning the regularity properties of stochastic convolutions that are used throughout the paper. Some of them are well known when $\delta=1$. In the context of the present work, these estimates depend on the fast scale parameter $\delta$. The reason we present them here is to showcase this dependence when $\delta$ is close to $0$. Finally, we provide the proof of estimate \eqref{deribar} in Lemma \ref{L:LimAverEq}.

\noindent For $i=1,2, \delta>0,$ $t\geq 0$ and an operator-valued map $B_i:[0,\infty)\rightarrow\mathscr{L}(\h)$ we define the re-scaled stochastic convolution $w^\delta_{A_i}$ by
\begin{equation}\label{wepsilon}
w_{A_i}^{\delta}(t):= \frac{1}{\sqrt{\delta}}\int_{0}^{t} S_i\bigg(\frac{t-z}{\delta}\bigg)B_i(z)dw_i(z).
\end{equation}

\noindent We consider $B_2$ to be constant in $s$ equal to identity. To study the space-time regularity of $w^\delta_{A_i}$, we use the stochastic factorization formula
\begin{equation}\label{stofact}
\begin{aligned}
&w^\delta_{A_i}(t)=\frac{\sin(a\pi)}{\sqrt{\delta}\pi}\int_{0}^{t}  (t-z)^{a-1} S_i\bigg(\frac{t-z}{\delta}\bigg) M^{\delta}_{a}(0,z,z;i)dz,\; a\in(0,1/2),	\end{aligned}
\end{equation}
where,  for any $t_1\leq t_2\leq t_3$, we define
\begin{equation}\label{stofactxi}
\begin{aligned}
M^\delta_{a}(t_1,t_2,t_3;i):=\int_{t_1}^{t_2} (t_3-\zeta)^{-a}S_i\bigg(\frac{t_3-\zeta}{\delta}\bigg)B_i(\zeta)dw_i(\zeta).
\end{aligned}
\end{equation}

\noindent
 The stochastic convolution $w^\delta_{A_i}$ is a well-defined $\h$-valued process and has a version with continuous paths (see \cite{da2014stochastic}, Theorem 5.11). Before we proceed to the main estimates we need the following auxiliary lemma:

\begin{lem}\label{sigmacont}
	Let $i=1,2,$ $0\leq s<t, \theta\in\R$ and $B_i:[0,\infty)\rightarrow\mathscr{L}(\h)$ be an operator-valued map. Furthermore, let $B_i^*(s)$ denote the $\h$-adjoint of the bounded linear operator $B_i(s)$. Under Hypotheses \ref{A1a} and \ref{A1b} the following hold:
	\\ (i) For $\rho\in(1/2,1)$ and $u\in\h$ there exists a constant $C_i>0$ such that  \begin{equation}\label{sigmaconttheta}
	\quad\quad\big\|S_i(t-s) (-A_i)^{\frac{\theta}{2}}B_i(s)u   \big\|_\h\leq C_i (t-s)^{-(\rho+\theta)/2}\big\|B_i^*(s)\big\|_{\mathscr{L}(L^\infty(0,L);\h)}\|u\|_\h\;.
	\end{equation}
	\\ (ii) Let $P^i_n\in\mathscr{L}(\h)$  denote the orthogonal projection to the $n$-dimensional subspace of $\h$ spanned by $\{e_{i,k}, k=1,\dots, n\}$. For $ \rho>\theta+\frac{1}{2}$ there exists a constant $C_i>0$ such that
	\begin{equation}\label{HSbnd}
	\begin{aligned}
	\sup_{n\in\N}\big\|(-A_i)^{\frac{\theta}{2}}S_i(t-s)&B_i(s)P^i_n\big\|^2_{\mathscr{L}_2(\h)}\leq C_i \|B_i^*(s)\|^2_{\mathscr{L}(L^{\infty}(0,L);\h)} (t-s)^{-\rho}.
	\end{aligned}
	\end{equation}
\end{lem}
\noindent These estimates are obtained by expanding with respect to the orthonormal basis $\{e_{i,k}, k\in\N\}$ and using Hypothesis \ref{A1b}, along with the fact that the eigenvalues of the elliptic operator $-A_i$ satisfy $a_{i,k}\sim k^2$, for each $k\in\N$. Such arguments can be found e.g. in Lemma 4.2 and Lemma 4.3 of \cite{WSS}.

\noindent In view of the strict dissipativity of $A_2$ (see Hypothesis \ref{A1c}), we can prove that the Hilbert-Schmidt norm of the fast semigroup $S_2$ decays exponentially for large enough $t$. In particular, we set $\theta=0, P^i_n=I, B\equiv I$ in \eqref{HSbnd} and then invoke \eqref{S2decay} to show that, for all $\rho\in(\frac{1}{2},1),$
\noindent
\begin{equation}\label{hsnorm}
\big\| S_2(t)\big\|_{\mathscr{L}_2(\h)}\leq C (t\wedge 1)^{-\frac{\rho}{2}}e^{-\lambda t},\; t> 0.
\end{equation}

	\noindent The next lemma provides temporal continuity estimates for the stochastic convolution $w^\delta_{A_2}$. As seen below, the estimate for the mean $C([0,T];\h)$ norm is singular of order $\delta^{-{\frac{1}{2}}^-}$ as $\epsilon\to 0$.

\begin{lem}\label{stoconvb}
	Let $T<\infty$, $\delta>0$ and $w^{\delta}_{A_2}$ be as in \eqref{wepsilon}. \\ (i) Let $p\geq 1$.  There exists $C>0$ independent of $\delta$ such that
	$$ \sup_{\delta>0,t\geq 0}\ex\big[  \|w_{A_2}^{\delta}(t)\|^{2p}_\h\big]\leq C.
	$$ (ii) For all $\rho\in(1/2,1)$ there exists $C_T>0$ independent of $\delta$ such that
	$$\ex \sup_{t\in[0, T]} \|w_{A_2}^{\delta}(t)\|^2_\h\leq C_T\delta^{\rho-1}.$$

\end{lem}

\begin{proof} $(i)$  An application of the Burkholder-Davis-Gundy inequality, along with the substitution $z\mapsto t- \delta \zeta$, yields
	\begin{equation*}
	\begin{aligned}
	\ex\|w_{A_2}^{\delta}(t)\|^{2p}_{\h}&\leq\frac{1}{\delta^p}\ex\sup_{s\in[0,t]} \bigg\|\int_{0}^{s} S_2\bigg(\frac{t-z}{\delta}\bigg)dw_2(z)\bigg\|^{2p}_\h\\& \leq \frac{C}{\delta^p}\bigg(\int_{0}^{t} \bigg\|S_2\bigg(\frac{t-z}{\delta}\bigg)\bigg\|^{2}_{\mathscr{L}_2(\h)}dz\bigg)^p  =C\bigg(\int_{0}^{t/\delta} \big\|S_2(\zeta)\big\|^2_{\mathscr{L}_2(\h)}d\zeta\bigg)^{p}.
	\end{aligned}
	\end{equation*}
	\noindent In view of \eqref{hsnorm} it follows that
	\begin{equation*}
	\begin{aligned}
	\ex\|w_{A_2}^{\delta}(t)\|^{2p}_{\h}&\leq C\int_{0}^{\infty} (1+\zeta^{-\rho})e^{-2\lambda \zeta} d\zeta=C(2\lambda)^{-1}+(2\lambda)^{\rho-1} \Gamma(1-\rho)<\infty,
	\end{aligned}
	\end{equation*}
	\noindent where $\rho<1$ and $\Gamma$ denotes the Gamma function.\\
	
	\noindent $(ii)$ Appealing to the stochastic factorization formula we have
	\begin{equation*}
	\begin{aligned}
	\|w_{A_2}^{\delta}(t)\|_{\h}&\leq\frac{\sin(a\pi)}{\sqrt{\delta}\pi}\int_{0}^{t}  (t-z)^{a-1} \bigg\|S_2\bigg(\frac{t-z}{\delta}\bigg) M^\delta_{a}(0,z,z;2)\bigg\|_\h dz\\& \leq  \frac{C_a}{\sqrt{\delta}}\int_{0}^{t}  (t-z)^{a-1}e^{-\frac{\lambda(t-z)}{\delta}} \big\| M^\delta_{a}(0,z,z;2)\big\|_\h dz.
	\end{aligned}
	\end{equation*}
	\noindent An application of H\"older's inequality for $q>1/a>2$ then yields
	\begin{equation*}
	\begin{aligned}
	\|w_{A_2}^{\delta}(t)\|_{\h}&\leq \frac{C}{\sqrt{\delta}} \bigg(\int_{0}^{T}(t-z)^{p(a-1)}dz\bigg)^{\frac{1}{p}}\bigg(\int_{0}^{T}\big\|M^\delta_{a}(0,z,z;2)\big\|^q_\h dz\bigg)^{\frac{1}{q}}\\&\leq \frac{CT^{a-\frac{1}{q}}}{\sqrt{\delta}}\bigg(\int_{0}^{T}\sup_{s\in[0,z]}\big\|M^\delta_{a}(0,s,z;2)\big\|^q_\h dz\bigg)^{\frac{1}{q}}.
	\end{aligned}
	\end{equation*}
	\noindent Thus, we apply Jensen's inequality to obtain
	\begin{equation*}
	\begin{aligned}
	\ex\sup_{t\in[0,T]}\|w_{A_2}^{\delta}(t)\|^2_{\h}&\leq \frac{C_{T}}{\delta} \bigg(\int_{0}^{T}\ex\sup_{s\in[0,z]}\big\| M^{\delta}_a(0,s,z;2)\big\|^q_\h dz\bigg)^{\frac{2}{q}}\\& \leq  \frac{C_T}{\delta}\bigg(\int_{0}^{T}\bigg(\int_{0}^{z}(z-\zeta)^{-2a}\ex \bigg\|S_2\bigg(\frac{z-\zeta}{\delta}\bigg)\bigg\|^2_{\mathscr{L}_2(\h)}d\zeta\bigg)^{\frac{q}{2}} dz\bigg)^{\frac{2}{q}}\\&
	\leq C \delta^{\rho-1}\bigg(\int_{0}^{T}\bigg(\int_{0}^{z}(z-\zeta)^{-2a-\rho}d\zeta\bigg)^{\frac{q}{2}} dz\bigg)^{\frac{2}{q}},
	\end{aligned}
	\end{equation*}
	\noindent where the second line follows from the Burkholder-Davis-Gundy inequality and the third from \eqref{hsnorm}. The last integral is finite, provided that we choose $a<(1-\rho)/2<1/4$. The proof is complete.	\end{proof}

\noindent Next, we provide estimates of spatial Sobolev   regularity and temporal H\"older regularity for  $w_{A_2}^{\delta}$. Both estimates are singular as $\epsilon\to 0$.
\begin{lem}
	Let $T<\infty$ and $\delta\in(0,1)$. \\
	\noindent$(i)$ For any $a,\theta<1/2$ and $\rho\in(\theta+1/2,1-2a)$ we have
	\begin{equation}\label{stcosob}
	\begin{aligned}
	\ex\sup_{t\in[0, T] }\|w_{A_2}^{\delta}(t)\|_{H^\theta}& \leq C_T \delta^\frac{\rho-1}{2}.
	\end{aligned}
	\end{equation}
	\noindent (ii) There exists $\beta<1/4$ such that for any $\rho\in(1/2, 1/2+2\beta)$
	\begin{equation}\label{Holderconv}
	\begin{aligned}
	\ex\big[ w_{A_2}^{\delta}\big]_{C^{\beta}([0,T];\h)}&\leq C_T \delta^{\frac{\rho-1}{2}}.
	\end{aligned}
	\end{equation}
	
\end{lem}

\begin{proof}
	$(i)$ Using the stochastic factorization formula and H\"older's inequality with $q>1/a>2$, as in the proof of Lemma \ref{stoconvb}(ii), we obtain
	\begin{equation*}
	\|w_{A_2}^{\delta}(t)\|_{H^\theta}\leq \frac{C_aT^{a-\frac{1}{q}}}{\sqrt{\delta}}\bigg(\int_{0}^{T}\sup_{s\in[0,z]}\big\|(-A_2)^\frac{\theta}{2}M^\delta_{a}(0,s,z;2)\big\|^q_\h dz\bigg)^{\frac{1}{q}}.
	\end{equation*} Assuming momentarily that the integrand in \eqref{stofactxi} is in $Dom((-A_2)^\frac{\theta}{2})$, we can interchange stochastic integral and unbounded operator and then apply Jensen's inequality followed by the Burkholder-Davis-Gundy inequality to obtain
	\begin{equation*}
	\begin{aligned}
	\ex\sup_{t\in[0, T] }\|w_{A_2}^{\delta}(s)\|_{H^\theta}&\leq \frac{C_T}{\sqrt\delta}\bigg(\int_{0}^{T}\bigg(\int_{0}^{z}(z-\zeta)^{-2a} \bigg\|(-A_2)^\frac{\theta}{2}S_2\bigg(\frac{z-\zeta}{\delta}\bigg)\bigg\|^2_{\mathscr{L}_2(\h)}d\zeta\bigg)^{\frac{q}{2}} dz\bigg)^{\frac{1}{q}}\\& \leq C_T
	\delta^{\frac{\rho-1}{2}}\bigg(\int_{0}^{T}\bigg(\int_{0}^{z}(z-\zeta)^{-2a-\rho}d\zeta\bigg)^{\frac{q}{2}} dz\bigg)^{\frac{1}{q}},
	\end{aligned}
	\end{equation*}
	\noindent where $\rho>\theta+1/2$ and the last line follows from Lemma \ref{sigmacont}(ii). The last integral is finite provided that $\theta<\frac{1}{2}-2a$ and $\theta+\frac{1}{2}<\rho<1-2a$.

	\noindent $(ii)$ Let $0\leq s<t\leq T$. From the stochastic factorization formula \eqref{stofact} it follows that
	\begin{equation*}\label{stconvdec}
	\begin{aligned}
	\frac{\sqrt\delta\pi}{\sin(a\pi)}\big(w_{A_2}^{\delta}(t)- w_{A_2}^{\delta}(s)\big)&=\int_{s}^{t}  (t-z)^{a-1} S_2\bigg(\frac{t-z}{\delta}\bigg) M^\delta_{a}(s,z,z;2) dz\\&+\bigg[S_2\bigg(\frac{t-s}{\delta}\bigg)-I\bigg]w_{A_2}^{\delta}(s)
	=:J^\delta_1(s,t)+J^\delta_2(s,t).
	\end{aligned}
	\end{equation*}
	
	\noindent  For the first term we apply H\"older's inequality with $q>1/a>2$ to obtain
	\begin{equation*}
	\begin{aligned}
	\big\|J_1^\delta(s,t)\big\|_\h&\leq \frac{1}{\sqrt{\delta}} \bigg(\int_{s}^{t}(t-z)^{p(a-1)}dz\bigg)^{\frac{1}{p}}\bigg(\int_{0}^{T}\big\|M^\delta_{a}(s,z,z;2)\big\|^q_\h dz\bigg)^{\frac{1}{q}}\\&
	\leq  \frac{C_a}{\sqrt{\delta}}(t-s)^{a-\frac{1}{q}}\bigg(\int_{0}^{T}\big\|M^\delta_{a}(s,z,z;2)\big\|^q_\h dz\bigg)^{\frac{1}{q}}.
	\end{aligned}
	\end{equation*}
	\noindent Recalling \eqref{stofactxi}, we see that $M^\delta_{a}(s,z,z;2)=M^\delta_{a}(0,z,z;2)-M^\delta_{a}(0,s,z;2)$. Therefore,
	\begin{equation*}
	\begin{aligned}
	\big\|J_1^\delta(s,t)\big\|_\h&
	\leq \frac{C_{a,q}}{\sqrt{\delta}}(t-s)^{a-\frac{1}{q}}\bigg(\int_{0}^{T}\sup_{s\in[0,z]}\big\|M^\delta_{a}(s,z,z)\big\|^q_\h dz\bigg)^{\frac{1}{q}}\;.
	\end{aligned}
	\end{equation*}

	\noindent Proceeding as in the proof of Lemma \ref{stoconvb}, we deduce that
	\begin{equation}\label{J1}
	\begin{aligned}
	\ex\sup_{s\neq t\in[0,T]}\frac{\big\|J_1^\delta(s,t)\big\|_\h}{|t-s|^{a-\frac{1}{q}}}&
	\leq C\delta^{\frac{\rho-1}{2}} \bigg(\int_{0}^{T}\bigg(\int_{0}^{z}(z-\zeta)^{-2a-\rho}d\zeta\bigg)^{\frac{q}{2}} dz\bigg)^{\frac{1}{q}}\;.
	\end{aligned}
	\end{equation}
	\noindent Note that $q$ is arbitrarily large and the last integral is finite, provided that $2\alpha<1-\rho<1/2$. \\
	\noindent As for $J_2^\delta$, we invoke \eqref{sobcont} to obtain
	\begin{equation*}
	\begin{aligned}
	\big\|J_2^\delta(s,t)\big\|_\h&
	\leq C \bigg\|S_2\bigg(\frac{t-s}{\delta}\bigg)-I\bigg\|_{\mathscr{L}(H^\theta;\h)}\|w_{A_2}^{\delta}(s)\|_{H^\theta}\\&
	\leq C \delta^{-\theta/2}(t-s)^{\theta/2}\|w_{A_2}^{\delta}(s)\|_{H^\theta}\;,
	\end{aligned}
	\end{equation*}
	where $\theta\in(0,1/2)$. In view of \eqref{stcosob}, we have
	\begin{equation*}
	\begin{aligned}
	&	\ex\sup_{s\neq t\in[0,T]}\frac{\big\|J_2^\delta(s,t)\big\|_\h}{|t-s|^{\theta/2}}&
	\leq C \delta^{-\theta/2}\ex\sup_{s\in [0,T]}\|w_{A_2}^{\delta}(s)\|_{H^\theta}
	\leq C\delta^\frac{\rho'-1-\theta}{2},
	\end{aligned}
	\end{equation*}
	where $\rho'\in(1/2+\theta,1-2a')$ and $a'<1/2$ can be arbitrarily small. Choosing $\rho\in(1/2, 1/2+\theta)$ and $\theta=\rho'-\rho<1/2-2a'$ it follows that
	\begin{equation}\label{J2}
	\begin{aligned}
	&	\ex\sup_{s\neq t\in[0,T]}\frac{\big\|J_2^\delta(s,t)\big\|_\h}{|t-s|^{\theta/2}}&
	\leq C  \delta^{-\theta/2}\ex\sup_{s\in [0,T]}\|w_{A_2}^{\delta}(s)\|_{H^\theta}
	\leq C \delta^\frac{\rho-1}{2}.
	\end{aligned}
	\end{equation}
	
	\noindent The proof is complete upon combining \eqref{J1} and \eqref{J2}. \end{proof}

\noindent We conclude this appendix with the proof of estimate \eqref{deribar} of Lemma \ref{L:LimAverEq}.

\begin{proof}[Proof of Lemma \ref{L:LimAverEq} (iii)]
	
		\noindent From the mild formulation of \eqref{x-aved} we have
	\begin{equation*}
	\begin{aligned}
	\bar{X}(t)&=
	S_1(t)x_0+\int_{0}^{t}	S_1(t-s)\bar{F}\big(\bar{X}(t) \big)ds+\int_{0}^{t}	S_1(t-s)\big[\bar{F}\big(\bar{X}(s) \big)-\bar{F}\big(\bar{X}(t) \big)\big] ds.
	\end{aligned}
	\end{equation*}
	\noindent Using this decomposition along with \eqref{sobcont} and the Lipschitz continuity of $\bar{F}$  we obtain
	\begin{equation*}
\small
	\begin{aligned}
	\big\|A_1\bar{X}(t)\big\|_\h&\leq 	\big\|A_1S_1(t)x_0	\big\|_\h+ \bigg\|\int_{0}^{t}	A_1S_1(t-s)\bar{F}\big(\bar{X}(t) \big) ds\bigg\|_\h	+\int_{0}^{t}\big\|	A_1S_1(t-s)\big[\bar{F}\big(\bar{X}(s) \big)-\bar{F}\big(\bar{X}(t) \big)\big] \big\|_\h ds\\&\leq C t^{\frac{a}{2}-1}\|x_0\|_{H^a}+\big\|\big(S_1(t)-I\big)\bar{F}\big(\bar{X}(t)\big) \big\|_\h +
	C_f\int_{0}^{t}	(t-s)^{-1}\big\|\bar{X}(s) -\bar{X}(t) \big\|_\h ds\\&\leq C t^{\frac{a}{2}-1}\|x_0\|_{H^a}+ c_T\bigg( 1+ L_f\sup_{t\in[0, T]}\big\| \bar{X}(t)\|_\h\bigg)
	+C \big[\bar{X}\big]_{C^{\theta}([0,T];\h)}                  \int_{0}^{t}	(t-s)^{-1+\theta}\ ds\\& \leq C t^{\frac{a}{2}-1}\|x_0\|_{H^a}+ C\big( 1+\big\|x_0\|_{H^a}\big)+ C_{f,\theta} (1+\|x_0\|_{H^a})     T^{\theta}   \\&
	\leq  C\bigg( t^{\frac{a}{2}-1}\|x_0\|_{H^a} + 1+\big\|x_0\|_{H^a}\bigg),
	\end{aligned}
	\end{equation*}
	\noindent where we used \eqref{xbarapriori} and \eqref{Schauderbar} to obtain the last inequality.
\end{proof}

	\section{}\label{AppB}
\noindent Here we give the proof of Lemma \ref{IVitodeclem}.
\begin{proof}
	\noindent By virtue of the It\^o formula and \eqref{Xider} we have
	\begin{equation}\label{Ito1}
	\begin{aligned}
	&\Theta\big(t,\bar{X}(t),Y_n^{\epsilon,u}(t)\big)-\Theta\big(s,\bar{X}(s),Y_n^{\epsilon,u}(s)\big)= \int_{s}^{t}\blangle \Psi^\epsilon\big(\bar{X}(z),Y_n^{\epsilon,u}(z)\big), S_1(t-z)(-A_1)^{1+\frac{\theta}{2}}\chi\brangle_\h dz\\&
	+\int_{s}^{t}    \blangle \Psi^\epsilon_1\big(\bar{X}(z),Y_n^{\epsilon,u}(z)\big)\big[ A_1 \bar{X}(z)+\bar{F}\big(\bar{X}(z)\big)\big], S_1(t-z)(-A_1)^{\frac{\theta}{2}}\chi \brangle_\h dz
	\\& +\frac{1}{\delta}\int_{s}^{t}   \blangle \Psi^\epsilon_2\big(\bar{X}(z),Y_n^{\epsilon,u}(z)\big)\big[A_2 Y_n^{\epsilon,u}(z)+ P_nG\big( \bar{X}(z), Y^{\epsilon,u}(z) \big)\big],S_1(t-z) (-A_1)^{\frac{\theta}{2}}\chi \brangle_\h dz\\&+ \frac{1}{2\delta}  \int_{s}^{t} \text{tr}\big[ P_nD^2_y\Phi^\epsilon_{ S_1(t-z)(-A_1)^{\frac{\theta}{2}}\chi}\big(\bar{X}(z),Y_n^{\epsilon,u}(z)\big)\big] dz \\&
	+\frac{h(\epsilon)}{\sqrt{\delta}}      \int_{s}^{t}  \blangle \Psi^\epsilon_2\big(\bar{X}(z),Y_n^{\epsilon,u}(z)\big)u_{2,n}(z), S_1(t-z)(-A_1)^{\frac{\theta}{2}}\chi \brangle_\h dz\\&+   \frac{1}{\sqrt{\delta}}    \int_{s}^{t}  \blangle  (-A_1)^{\frac{\theta}{2}}S_1(t-z)\Psi^\epsilon_2\big(\bar{X}(z),Y_n^{\epsilon,u}(z)\big) d w_{2,n}(z) ,\chi \brangle_\h \;.
	\end{aligned}
	\end{equation}
	\noindent In view of \eqref{Riesz}, we can express the sum of the third and fourth terms on the right-hand side of the last display in terms of the Kolmogorov operator $\mathcal{L}^x$ (see \eqref{Kolmopintro}) via the identity
	\begin{equation}\label{Kolmoito}
	\begin{aligned}
	&\frac{1}{\delta}\int_{s}^{t}   \blangle \Psi^\epsilon_2\big(\bar{X}(z),Y_n^{\epsilon,u}(z)\big)\big[A_2 Y_n^{\epsilon,u}(z)+ P_nG\big( \bar{X}(z), Y^{\epsilon,u}(z) \big)\big],S_1(t-z) (-A_1)^{\frac{\theta}{2}}\chi \brangle_\h dz\\&+ \frac{1}{2\delta}  \int_{s}^{t} \text{tr}\big[ P_nD^2_y\Phi^\epsilon_{ S_1(t-z)(-A_1)^{\frac{\theta}{2}}\chi}\big(\bar{X}(z),Y_n^{\epsilon,u}(z)\big)\big] dz= \frac{1}{\delta}\int_{s}^{t}\mathcal{L}^{\bar{X}(z)}\Phi_{S_1(t-z)(-A_1)^{\frac{\theta}{2}}\chi}^\epsilon\big(\bar{X}(z), Y_n^{\epsilon,u}(z)\big)  dz\\&+ \frac{\sqrt{\epsilon}h(\epsilon)}{\delta}T_3^{\epsilon,u}( s,t,n,\theta,\chi).
	\end{aligned}
	\end{equation}
	\noindent In view of \eqref{Kolmoito}, we return to \eqref{Ito1}, apply \eqref{Xiriesz} on the left-hand side and then multiply throughout by $\delta$ to obtain
	\begin{equation}\label{Ito2}
	\begin{aligned}
	&\delta\big[       \langle \Psi^\epsilon\big(\bar{X}(t),Y_n^{\epsilon,u}(t)\big), (-A_1)^{\frac{\theta}{2}}\chi\rangle_\h - \langle \Psi^\epsilon\big(\bar{X}(s),Y_n^{\epsilon,u}(s)\big), S_1(t-s)(-A_1)^{\frac{\theta}{2}}\chi\rangle_\h\big]\\&= \delta\int_{s}^{t}\blangle \Psi^\epsilon\big(\bar{X}(z),Y_n^{\epsilon,u}(z)\big), S_1(t-z)(-A_1)^{1+\frac{\theta}{2}}\chi\brangle_\h dz\\&
	+\delta\int_{s}^{t}    \blangle \Psi^\epsilon_1\big(\bar{X}(z),Y_n^{\epsilon,u}(z)\big)\big[ A_1 \bar{X}(z)+\bar{F}\big(\bar{X}(z)\big)\big], S_1(t-z)(-A_1)^{\frac{\theta}{2}}\chi \brangle_\h dz
	\\&+\int_{s}^{t}\mathcal{L}^{\bar{X}(z)}\Phi_{S_1(t-z)(-A_1)^{\frac{\theta}{2}}\chi}^\epsilon\big(\bar{X}(z), Y_n^{\epsilon,u}(z)\big)  dz
	\\&+\sqrt{\delta}h(\epsilon)    \int_{s}^{t}  \blangle \Psi^\epsilon_2\big(\bar{X}(z),Y_n^{\epsilon,u}(z)\big)u_{2,n}(z), S_1(t-z)(-A_1)^{\frac{\theta}{2}}\chi \brangle_\h dz\\&+   \sqrt{\delta}    \int_{s}^{t}  \blangle  (-A_1)^{\frac{\theta}{2}}S_1(t-z)\Psi^\epsilon_2\big(\bar{X}(z),Y_n^{\epsilon,u}(z)\big) d w_{2,n}(z) ,\chi \brangle_\h +\sqrt{\epsilon}h(\epsilon) T_3^{\epsilon,u}( s,t,n,\theta,\chi),
	\end{aligned}
	\end{equation}
	\noindent Since $\Phi^\epsilon_{\cdot}$ solves the Kolmogorov equation \eqref{Kolmeq},
	\begin{equation*}
	\begin{aligned}
	&\mathcal{L}^{\bar{X}(t)}\Phi_{S_1(t-z)(-A_1)^{\frac{\theta}{2}}\chi}^\epsilon\big(\bar{X}(z), Y_n^{\epsilon,u}(z)\big) =c(\epsilon)\Phi_{S_1(t-z)(-A_1)^{\frac{\theta}{2}}\chi}^\epsilon\big(\bar{X}(z), Y_n^{\epsilon,u}(z)\big)\\&- \blangle F\big(\bar{X}(z), Y_n^{\epsilon,u}(z) \big)-\bar{F}\big(\bar{X}(z) \big),  S_1(t-z)(-A_1)^{\frac{\theta}{2}}\chi\brangle_\h
	= c(\epsilon)\blangle \Psi^\epsilon\big(\bar{X}(z), Y_n^{\epsilon,u}(z)\big), S_1(t-z)(-A_1)^{\frac{\theta}{2}}\chi\brangle_\h\\&- \blangle F\big(\bar{X}(z), Y_n^{\epsilon,u}(z) \big)-\bar{F}\big(\bar{X}(z) \big),  S_1(t-z)(-A_1)^{\frac{\theta}{2}}\chi\brangle_\h\;.
	\end{aligned}
	\end{equation*}
	\noindent Consequently, we can rearrange \eqref{Ito2} to obtain
	\begin{equation}\label{Ito3}
	\begin{aligned}
	&	\int_{s}^{t}\blangle F\big(\bar{X}(z), Y_n^{\epsilon,u}(z) \big)-\bar{F}\big(\bar{X}(z) \big),  S_1(t-z)(-A_1)^{\frac{\theta}{2}}\chi\brangle_\h dz\\&=
	-\delta\big[       \langle \Psi^\epsilon\big(\bar{X}(t),Y_n^{\epsilon,u}(t)\big), (-A_1)^{\frac{\theta}{2}}\chi\rangle_\h - \langle \Psi^\epsilon\big(\bar{X}(s),Y_n^{\epsilon,u}(s)\big), S_1(t-s)(-A_1)^{\frac{\theta}{2}}\chi\rangle_\h\big]\\& +\delta\int_{s}^{t}\blangle \Psi^\epsilon\big(\bar{X}(z),Y_n^{\epsilon,u}(z)\big), S_1(t-z)(-A_1)^{1+\frac{\theta}{2}}\chi\brangle_\h dz\\&
	+\delta\int_{s}^{t}    \blangle \Psi^\epsilon_1\big(\bar{X}(z),Y_n^{\epsilon,u}(z)\big)\big[ A_1 \bar{X}(z)+\bar{F}\big(\bar{X}(z)\big)\big], S_1(t-z)(-A_1)^{\frac{\theta}{2}}\chi \brangle_\h dz
	\\&+c(\epsilon)\int_{s}^{t}\blangle \Psi^\epsilon\big(\bar{X}(z), Y_n^{\epsilon,u}(z)\big), S_1(t-z)(-A_1)^{\frac{\theta}{2}}\chi\brangle_\h dz \\&
	+\sqrt{\delta}h(\epsilon)    \int_{s}^{t}  \blangle \Psi^\epsilon_2\big(\bar{X}(z),Y_n^{\epsilon,u}(z)\big)u_{2,n}(z), S_1(t-z)(-A_1)^{\frac{\theta}{2}}\chi \brangle_\h dz\\&+   \sqrt{\delta}    \int_{s}^{t}  \blangle  (-A_1)^{\frac{\theta}{2}}S_1(t-z)\Psi^\epsilon_2\big(\bar{X}(z),Y_n^{\epsilon,u}(z)\big) d w_{2,n}(z) ,\chi \brangle_\h + \sqrt{\epsilon}h(\epsilon)T_3^{\epsilon,u}( s,t,n,\theta,\chi).
	\end{aligned}
	\end{equation}
	\noindent Regarding the second term on the right-hand side of the last display we can write
	\begin{equation*}
	\begin{aligned}
	\int_{s}^{t}\blangle& \Psi^\epsilon\big(\bar{X}(z),Y_n^{\epsilon,u}(z)\big), S_1(t-z)(-A_1)^{1+\frac{\theta}{2}}\chi\brangle_\h dz\\&= \blangle \Psi^\epsilon\big(\bar{X}(t),Y_n^{\epsilon,u}(t)\big), 	\int_{s}^{t}S_1(t-z)(-A_1)^{1+\frac{\theta}{2}}\chi dz\brangle_\h \\& + 	\int_{s}^{t}\blangle \Psi^\epsilon\big(\bar{X}(z),Y_n^{\epsilon,u}(z)\big)-\Psi^\epsilon\big(\bar{X}(t),Y_n^{\epsilon,u}(t)\big), S_1(t-z)(-A_1)^{1+\frac{\theta}{2}}\chi\brangle_\h dz\\&
	= \blangle \Psi^\epsilon\big(\bar{X}(t),Y_n^{\epsilon,u}(t)\big), 	\big[I-S_1(t-s)\big](-A_1)^{\frac{\theta}{2}}\chi \brangle_\h\\&
	+ \int_{s}^{t}\blangle \Psi^\epsilon\big(\bar{X}(z),Y_n^{\epsilon,u}(z)\big)-\Psi^\epsilon\big(\bar{X}(t),Y_n^{\epsilon,u}(t)\big), S_1(t-z)(-A_1)^{1+\frac{\theta}{2}}\chi\brangle_\h dz,
	\end{aligned}
	\end{equation*}
	where we used the fact that $S_1(t-z)(-A_1)^{1+\frac{\theta}{2}}\chi= \frac{d}{dz}S_1(t-z)(-A_1)^{\frac{\theta}{2}}\chi$.  With $T_1^{\epsilon,u}$, $T_3^{\epsilon,u}$ as in \eqref{Tdec}, \eqref{T3} respectively, we can further rearrange \eqref{Ito3} and divide throughout by $\sqrt{\epsilon}h(\epsilon)$   to obtain
	
\begin{equation}\label{Itomain}\hspace*{-0.2cm}
	\begin{aligned}
	&T_1^{\epsilon,u}( s,t,n,\theta,\chi)=
	-\frac{\delta}{  \sqrt{\epsilon}h(\epsilon)    }\big[       \langle \Psi^\epsilon\big(\bar{X}(t),Y_n^{\epsilon,u}(t)\big)-\Psi^\epsilon\big(\bar{X}(s),Y_n^{\epsilon,u}(s)\big), S_1(t-s)(-A_1)^{\frac{\theta}{2}}\chi\rangle_\h\big]\\& -\frac{\delta}{  \sqrt{\epsilon}h(\epsilon)    }\int_{s}^{t}\blangle \Psi^\epsilon\big(\bar{X}(t),Y_n^{\epsilon,u}(t)\big)-\Psi^\epsilon\big(\bar{X}(z),Y_n^{\epsilon,u}(z)\big), S_1(t-z)(-A_1)^{1+\frac{\theta}{2}}\chi\brangle_\h dz\\&
	+\frac{\delta}{  \sqrt{\epsilon}h(\epsilon)    }\int_{s}^{t}    \blangle \Psi^\epsilon_1\big(\bar{X}(z),Y_n^{\epsilon,u}(z)\big)\big[ A_1 \bar{X}(z)+\bar{F}\big(\bar{X}(z)\big)\big], S_1(t-z)(-A_1)^{\frac{\theta}{2}}\chi \brangle_\h dz
	\\&+\frac{c(\epsilon)}{  \sqrt{\epsilon}h(\epsilon)    }\int_{s}^{t}\blangle \Psi^\epsilon\big(\bar{X}(z), Y_n^{\epsilon,u}(z)\big), S_1(t-z)(-A_1)^{\frac{\theta}{2}}\chi\brangle_\h dz \\&
	+\frac{\sqrt{\delta}} {\sqrt{\epsilon}  }  \int_{s}^{t}  \blangle \Psi^\epsilon_2\big(\bar{X}(z),Y_n^{\epsilon,u}(z)\big)u_{2,n}(z), S_1(t-z)(-A_1)^{\frac{\theta}{2}}\chi \brangle_\h dz\\&+  \frac{\sqrt{\delta}}{\sqrt{\epsilon}h(\epsilon)}    \int_{s}^{t}  \blangle  (-A_1)^{\frac{\theta}{2}}S_1(t-z)\Psi^\epsilon_2\big(\bar{X}(z),Y_n^{\epsilon,u}(z)\big) d w_{2,n}(z) ,\chi \brangle_\h + T_3^{\epsilon,u}( s,t,n,\theta,\chi).
	\end{aligned}
	\end{equation}

		In view of \eqref{Tdec}, the argument is complete upon adding $T_2^{\epsilon,u}( s,t,n,\theta,\chi)$ in both sides of the last display. 	\end{proof}

	\bibliographystyle{plain}
\nocite{*}
\bibliography{References3}

\end{document}